\documentclass{article}
\usepackage[english]{babel}

\usepackage[letterpaper,top=2cm,bottom=2cm,left=3cm,right=3cm,marginparwidth=1.75cm]{geometry}

\usepackage{amsmath, amsthm, amssymb}
\usepackage{graphicx,  enumerate, comment,mathtools,array,mdframed, paralist, enumerate}
\usepackage[colorlinks=true, allcolors=blue]{hyperref}

\usepackage{autonum,tabularx}

\newcommand{\lcm}[2]{(#1\vee#2)}
\renewcommand{\gcd}[2]{(#1\wedge#2)}

\newtheorem*{question*}{Main problem}

\newtheorem{theorem}{Theorem}
\newtheorem{lemma}[theorem]{Lemma}
\newtheorem{proposition}[theorem]{Proposition}
\newtheorem{corollary}[theorem]{Corollary} 
\newtheorem{classicalFact}[theorem]{Classical fact}
\theoremstyle{definition}
\newtheorem{example}{Example}
\newtheorem{remark}[example]{Remark}

\newtheorem{definition}[example]{Definition}

\newcommand{\cint}{c_{\textrm{int}}}

\newcommand{\capprox}{c_{\textrm{approx}}}

\newcommand{\w}{\omega}
\newcommand{\bw}{\boldsymbol{\omega}}

\newcommand{\prim}{\pi}
\newcommand{\bfprim}{\boldsymbol{\pi}}

\newcommand{\bfr}{\mathbf{r}}
\newcommand{\bfs}{\mathbf{s}}
\newcommand{\bfw}{\mathbf{w}}
\newcommand{\bfx}{\mathbf{x}}
\newcommand{\bfy}{\mathbf{y}}

\newcommand{\cQ}{\mathcal{Q}}
\newcommand{\cC}{\mathcal{C}}
\newcommand{\cM}{\mathcal{M}}
\newcommand{\cT}{\mathcal{T}}
\newcommand{\cP}{\mathcal{P}}

\newcommand{\congplane}{\stackrel{pl}{=}}
\newcommand{\conglab}{\stackrel{lab}{=}}

\newcommand{\bbP}{\mathbb P}

\newcommand{\bbE}{\mathbb E}

\newcommand{\Var}{\mathrm{Var}}
\newcommand{\PBGW}{\bbP_{\scriptscriptstyle BGW}}
\newcommand{\EBGW}{\bbE_{\scriptscriptstyle BGW}}
\newcommand{\muBGW}{\mu_{\scriptscriptstyle BGW}}

\newcommand{\Supp}{\textrm{Supp}}
\newcommand{\defi}{:=}

\title{Composition of random functions and word reconstruction}

\author{Guillaume Chapuy%
\thanks{Université Paris Cité, CNRS, IRIF, F-75013, Paris, France
Email:~{\tt guillaume.chapuy@irif.fr}.
}
\and Guillem Perarnau%
\thanks{Departament de Matem\`atiques and IMTECH, Universitat Polit\`ecnica de Catalunya (UPC), Barcelona, Spain. Centre de Recerca Matemàtica, Barcelona, Spain. Email:~{\tt guillem.perarnau@upc.edu}. 
}
}

\begin{document}
\maketitle

\begin{abstract}
Given two functions $\mathbf{a}\!:\! [n] \rightarrow [n]$ and $\mathbf{b}\!:\!  [n] \rightarrow [n]$ chosen uniformly at random, any word $w=w_1w_2\dots w_k\in \{a,b\}^k$ induces a random function $\mathbf{w}\!:\!  [n] \rightarrow [n]$ by composition, i.e. $\mathbf{w}=\phi_{w_k}\circ  \dots \circ \phi_{w_1}$ with $\phi_a=\mathbf{a}$ and $\phi_b=\mathbf{b}$.
We study the following question: assuming $w$ is fixed but unknown, and $n$ goes to infinity, does the random function $\mathbf{w}$ carry enough information to (partially) recover the word $w$ with good enough probability, in one or several i.i.d. samples?

We prove that the random functions stemming from two non-isomorphic words can be discriminated with probability arbitrarily close to $1$ with a bounded number of samples, when $n$ goes to infinity. Equivalently, the total variation distance between their distributions is bounded away from zero. The proof relies on the study of certain \emph{auto-correlation} functions appearing in the variance of the weighted number of \emph{quasi-leaves}.

	Whether the total variation distance goes to 1, or equivalently, whether {\it one} sample is enough to distinguish the words with high probability, is the major question we leave open. We show that this is the case when the two words have different lengths or different exponents.
\end{abstract}


\begin{figure}[h]
\includegraphics[width=0.49\linewidth]{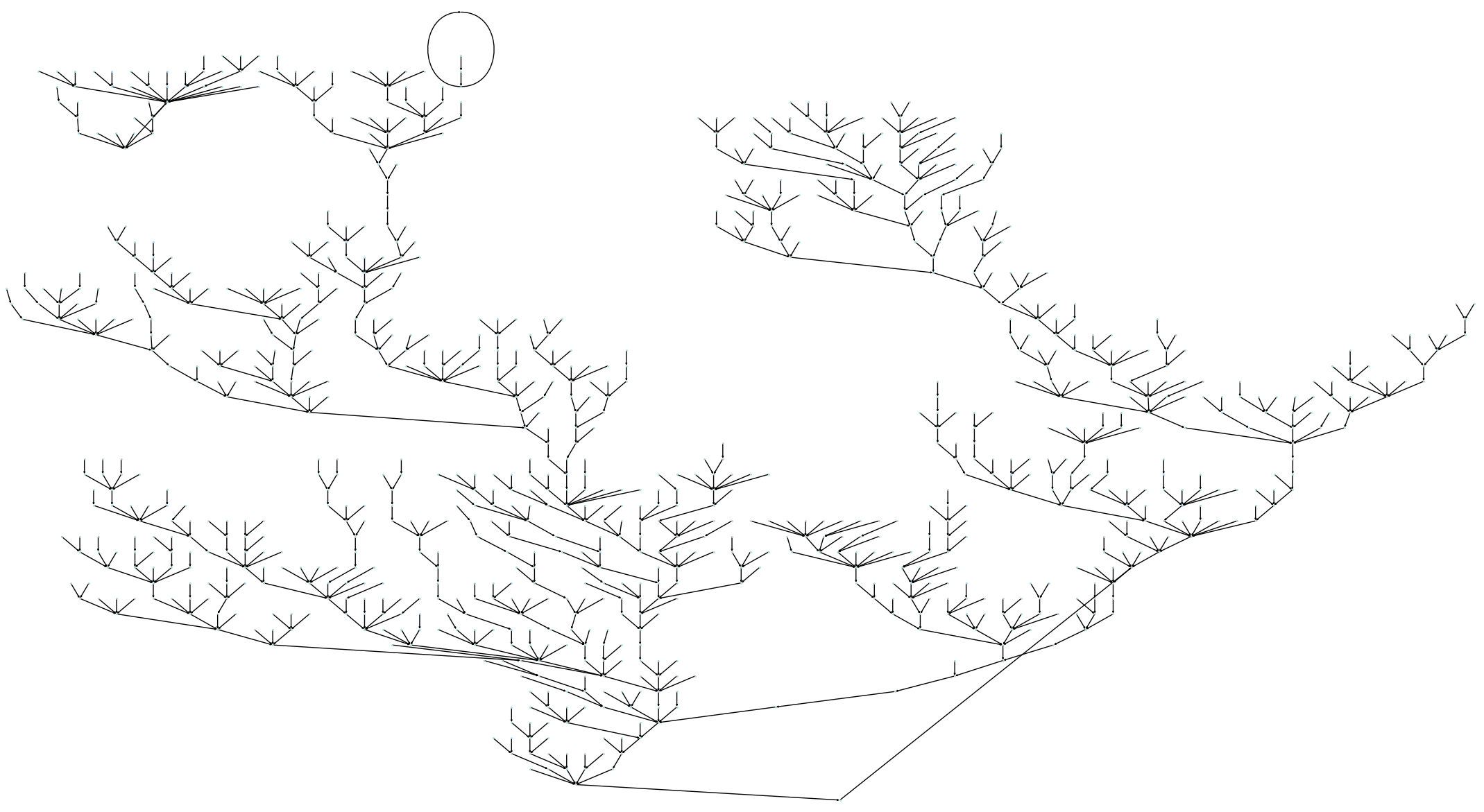}\hfill
\includegraphics[width=0.49\linewidth]{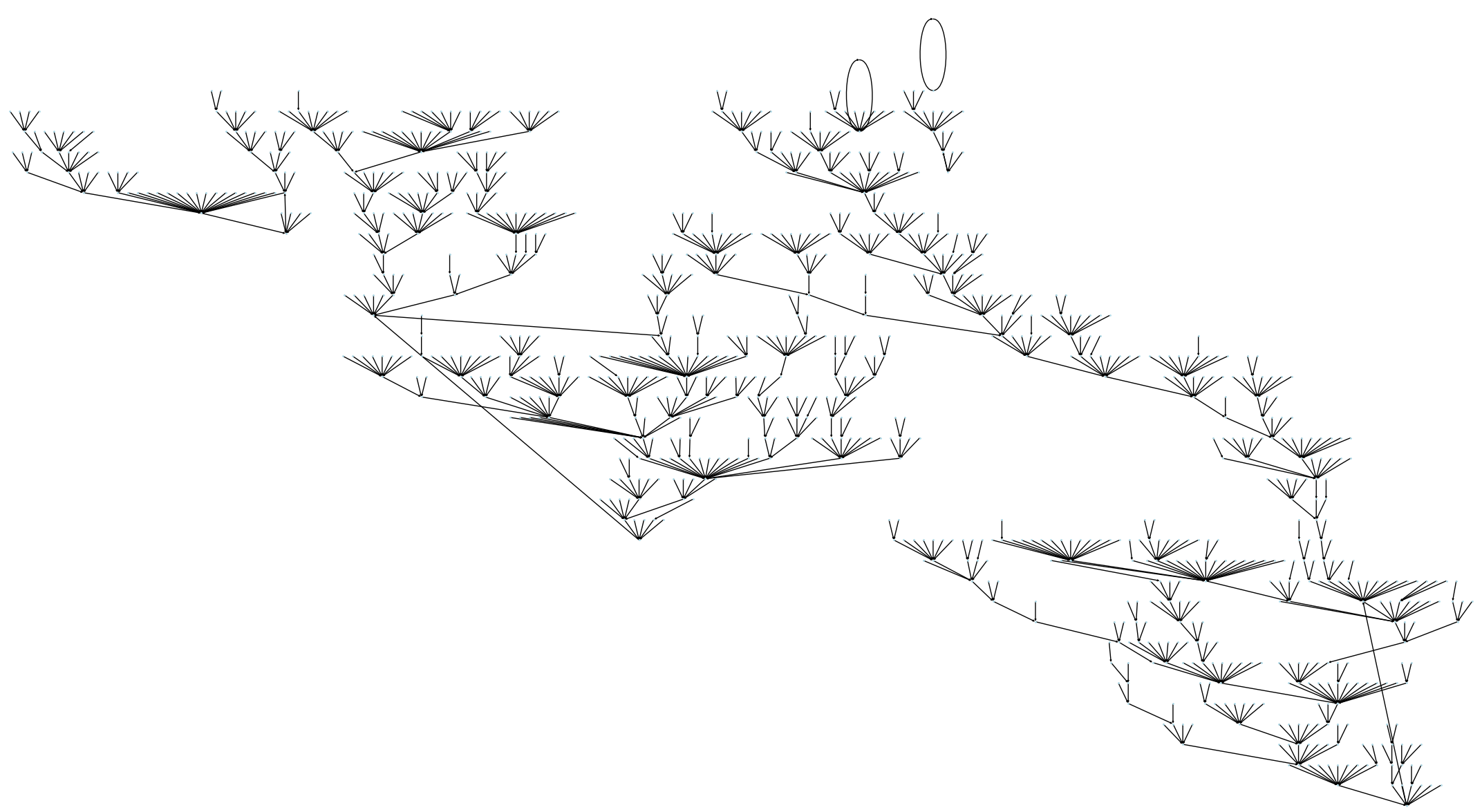}
\caption{Samples of the random function $\bfw$ for $n=1000$, for $w=ab$ (left) and $w=aaabaa$ (right). Since the two words have different length one can distinguish them by looking at the proportion of leaves (Theorem~\ref{thm:main_leaves}).}
\end{figure}

\section{Main question and main results}
\label{sec:intro}

\subsection{Presentation of the problem}

The problem studied in this paper is motivated by the study of random automata, random permutations, and random trees.

Throughout this paper, $n\geq1$ is an integer that will eventually tend to infinity. We consider two functions $\mathbf{a}, \mathbf{b}\!:\![n]\longrightarrow [n]$ chosen independently and uniformly at random, where we write $[n]\defi \{1,2,\dots,n\}$. One can picture this data as a deterministic finite automaton (DFA), in which the elements in $[n]$ are states and the two functions $\mathbf{a}$ and $\mathbf{b}$ represent the transitions associated to the letters $a$ and $b$, respectively. In graph theory language, this object can be represented as a directed graph with two types of arcs (labelled with $a$ or $b$), such that each vertex has exactly one outgoing arc of each type. Choosing the functions independently and uniformly at random is equivalent to choosing the DFA uniformly at random, and equivalent to saying that each outgoing arc of each vertex is chosen uniformly at random in $[n]$, independently of everything else.

In the context of automata, the notion of $w$-transition is very natural: the $w$-transition from a state is obtained by following successively the functions $\mathbf{a}$ and $\mathbf{b}$ according to the letters of the word $w$. Formally, given an integer $k\geq 1$ and a word $w=w_1w_2\dots w_k$ from $\{a,b\}^*$, the set of words on $\{a,b\}$, the \emph{$w$-transitions} are represented by the function $\bfw\!:\![n] \longrightarrow [n]$ defined as
$$
\bfw = \phi_{w_k}\circ \dots\circ \phi_{w_2}\circ\phi_{w_1}, \ \ \text{ where } \phi_a=\mathbf{a} \ \text{and}\  \phi_b=\mathbf{b}.
$$
The random function $\bfw$ is a priori non-uniform. The goal of this paper is to study it and understand how it depends on the associated word $w$. 
We defer a detailed discussion of the connections with random automata and random permutations to Section~\ref{sec:permutations} at the end of the introduction, beginning instead with our main questions and results.

Our main question can be stated in several forms:
	\begin{compactitem}
		\item 	What are the properties of the random function $\bfw$, and to what extent does $\bfw$ encode information about the word $w$ used to construct it?
		\item More precisely, which properties of the word $w$ can be guessed with probability close to one in one sample of the random function $\bfw$? and in a bounded number of samples?
\end{compactitem}
A natural way to attack this question is to find statistics, measurable observables of the function $\bfw$, that behave differently for different words $w$.
	\begin{compactitem}
		\item
	Which observables of the random function $\bfw$ have characterizable asymptotic behaviour that depends on $w$ or uniquely determines $w$ up to isomorphism?
\end{compactitem}
For example, we will see that the proportion of leaves of the function $\bfw$ (see below for a precise definition) satisfies an almost sure convergence theorem, with a deterministic limit only depending on the length of the word $w$. Thus the word length can be guessed, with high probability, in one sample.

A natural way to formalize these problems is measuring the distance between the distributions of two random functions. Recall that the \emph{total variation distance} (\emph{TV-distance}) between two distributions $\mu$ and $\nu$ on some finite set $\Omega$, is
$$
d_{TV}(\mu,\nu) \defi \frac{1}{2}\sum_{x \in \Omega} | \mu(x) - \nu(x)|.
$$
An equivalent statistically-flavored definition is that the TV-distance measures how difficult it is to distinguish between samples from	 $\mu$ and $\nu$. Namely, $1-d_{TV}(\mu,\nu)$ is the minimum of the sum of Type I and Type II errors, among all statistical tests between the two distributions.	
Equivalently, $\frac{1}{2}(1+d_{TV}(\mu,\nu))$ is the maximum probability to guess the source urn if one is presented a sample taken from one of two urns (each chosen with probability~$1/2$), sampling respectively the distributions $\mu$ and $\nu$; see e.g.~\cite[Equation~(1.1)]{Lugosi}. 
Note that the TV-distance takes values in $[0,1]$.

We say that two words are \emph{isomorphic} if one can be obtained from the other by permuting the letters of the alphabet; if the words are on $\{a,b\}$, the equivalence class of a non-empty word has size exactly two, with a unique representative starting with an $a$. If $w$ and $w'$ are isomorphic, then $\bfw$ and $\bfw'$ follow the same law. Our main question can thus be formulated as follows. 

\begin{question*}[\bf  Total variation distance between random functions]
	Fix $w, w' \in \{a,b\}^*$ two non-isomorphic words, and let $\bfw, \bfw'$ be the associated random functions. Which asymptotic bounds hold for their TV-distance? More precisely, determine which of the following two inequalities are strict:
	\begin{align}\label{eq:bounds_TVd}
	0 \leq\liminf_{n\to\infty} d_{TV}(\bfw,\bfw') \ \ , \ \  \limsup_{n\to\infty} d_{TV}(\bfw,\bfw') \leq 1. 
	\end{align}
\end{question*}

In this paper, we solve the question completely for the first inequality, which we show is always strict (main result in Theorem~\ref{thm:cw_separate}). 
The second question remains wide open in general, and we do not dare to make a conjecture about it. Nevertheless, we do prove that the TV-distance goes to $1$ under much stronger hypotheses on the words $w$ and $w'$ (different lengths or different exponents, Corollaries~\ref{cor:length} and~\ref{cor:distinguish_d} respectively). In the next section we state these results precisely.

\begin{remark}[\bf Reformulation in the urn setting]\label{rem:urn}
	Given two non-isomorphic words $w,w'$ and their associated random functions $\bfw$ and $\bfw'$, the inequalities in~\eqref{eq:bounds_TVd} admit an equivalent formulation in the urn setting as above, in which two urns produce i.i.d. samples of the function $\bfw$ and $\bfw'$ respectively, we are presented the output of only one of the two urns, each with probability $1/2$, and our goal is to guess which urn has produced it. The first inequality is strict if and only if a bounded number of i.i.d. samples suffices to guess the correct urn with probability arbitrarily close to $1$. The second inequality is an equality if a single sample suffices to guess it with high probability.
\end{remark}

\subsection{Main results}

Throughout the paper we follow the notation of the previous section:  $n\geq 1$ is an integer and $\mathbf{a},\mathbf{b} \! : \! [n]\to [n]$ are independent uniform random functions. We use $\bbP$ to refer to the probability measure induced by $\mathbf{a}, \mathbf{b}$, $\bbE$ for its expectation and $\Var$ for its variance. Occasionally, we extend this space with one (or more) vertex $R$ chosen independently and uniformly at random, keeping the same notation.
We let $w\in \{a,b\}^*$ be a word and we denote its length by $|w|$ and its associated random function by $\bfw$. For the statements concerning two words, we use the notation $w'$ for the second word, and $\bfw'$ for the associated random function. Unless otherwise specified, asymptotic statements are taken as $n \to \infty$, and occasionally as $L \to \infty$ or $\epsilon \to 0$. Throughout this analysis, the lengths of the words $w$ and $w'$ remain fixed. Consequently, any implicit constants in the $O$-notation or limits may depend on $|w|$ and $|w'|$; dependence on any other parameter $u$ will be indicated explicitly as $O_u$.

\subsubsection{Guessing the length of the word}

We define the following sequence of functions: $g_0(s) = s$ and 
\begin{align}\label{eq:def_g}
	g_{t}(s) := f(g_{t-1}(s)) = \exp(\exp(\dots \exp(s-1)-1)\dots )-1),\quad \text{for all }t\geq 1,
\end{align}
where $f(s):=e^{s-1}$.
We also let $\eta_t:=g_t(0)$.
As we will see below, $g_t(s)$ is the probability generating function of generation $t$ in a Poisson(1)-Bienaymé-Galton-Watson tree (BGWT), and $\eta_t$ is the probability that it has no progeny at generation $t$.
Note that the sequence $\eta_t$ is strictly increasing.

We say that $x$ is a \emph{leaf} if it has no $\bfw$-preimage and a \emph{quasi-leaf} if all its $\bfw$-preimages are leaves. In particular all leaves are quasi-leaves. Let $\cQ_n^w$ be the set of quasi-leaves of $\bfw$. Let $d^\bfw(x)$ be the number of $\bfw$-preimages of $x\in [n]$.

In the rest of the paper $u\in[0,1]$ is a real parameter, and we consider the following weighted number of quasi-leaves:
\begin{align}
Q_n^w(u):&=\sum_{x\in [n]} u^{d^{\bfw}(x)}\mathbf{1}_{x\in \cQ_n^w}.
\end{align}
We insist that $u$ belongs to  $[0,1]$ and for simplicity we will not try to extend our results to a larger domain\footnote{Some of the limit quantities we obtain will be extended in Section~\ref{sec:independence} to $u\in\mathbb{R}$, but $Q_n^w(u)$ itself will be only considered for $u\in [0,1]$.}. All our error terms will be uniform over $u\in [0,1]$.

If the word $w$ (or function $\bfw$) is clear from the context we will write $d(x)$, $\cQ_n $ and $Q_n(u)$.

\begin{theorem}[\bf Statistics of quasi-leaves]\label{thm:main_leaves}
For any integer $k\geq 1$, $w\in \{a,b\}^k$ and $u\in [0,1]$, we have 
\begin{align}\label{eq:main_leaves}
\frac{Q_n^w(u)}{n} \xrightarrow{n \to \infty} g_k(u\eta_k),\quad \text{in probability}.
\end{align}
\end{theorem}

By setting $u=0$ and $u=1$, we obtain the following consequence.
\begin{corollary}[\bf Number of leaves and quasi-leaves]\label{coro:number_leaves}
	For any  integer $k\geq 1$  and any $w\in \{a,b\}^k$, the proportion of vertices in $[n]$ which are leaves and quasi-leaves of $\bfw$ converge in probability respectively to $\eta_k$ and to $\eta_{2k}$. 
\end{corollary}

Note that quasi-leaves in $\bfw$ are leaves in $\bfw^2$ (associated to the word $w^2$ of length $2k$), so the second part of the corollary is in fact a consequence of the first. As $\eta_t$ is a strictly increasing in $t\geq 0$, we obtain the following consequence.
\begin{corollary}[\bf Guessing the length]\label{cor:length}
	The length of $w$ can be recovered from a single sample of~$\bfw$ with probability tending to $1$ as $n$ tends to infinity.
	In particular, if $w,w'$ are two words with $|w|\neq |w'|$, then $d_{TV}(\bfw,\bfw')\rightarrow 1$ when $n$ goes to infinity.
\end{corollary}

\subsubsection{Guessing the exponent of the word}

The \emph{exponent of $w$} is the largest integer $d\geq 1$ such that $w$ is a $d$-power, i.e. there exists a word $\prim$ such that $w=\prim^d$ is the concatenation of $d$ copies of $\prim$.
The word $\prim$ has exponent $1$ and is called \emph{primitive}. 
As we will see, for primitive words $w$ the random function $\bfw$ has a short cycle structure close to that of uniform random functions or permutations: for a fixed $j\geq 1$, the number of cycles of length $j$ converges in distribution to a Poisson($1/j$) as $n$ goes to infinity, and the number of cycles of length at most $L$ is concentrated around $\log L$, for large~$L$.

We refine this observable by weighting the number of cycles of length $j$ with a periodic sequence of weights of period $g\geq1$. More precisely, for $g\geq1$ and a sequence of weights $\mathbf{z}=(z_0,\dots,z_{g-1})\in \mathbb{C}^g$, we define
$$
D^w_n(L,g;\mathbf{z}) \defi \sum_{j=1}^{L} z_{(j\text{ mod } g)} C^w_{j,n},
$$
where $C^w_{j,n}$ is the number of cycles of length $j$ in the random function $\bfw$.

\begin{theorem}[\bf Weighted cycle-count]\label{thm:Dn_limit}
	Let $w\in\{a,b\}^*$ be a word of exponent $d$ and let $g\geq 1$. Then for any tuple $\mathbf{z}=(z_0,\dots,z_{g-1})\in \mathbb{C}^g$, we have the convergence in distribution
\begin{align}\label{eq:weighted_cycle_count}
		\lim_{L\rightarrow \infty}
		\lim_{n\rightarrow \infty}
		\frac{D^w_n(L,g;\mathbf{z})}{\log L} 
=
		\frac{1}{\mathrm{lcm}(d,g)}\sum_{r=1}^{\mathrm{lcm}(d,g)}z_{(r\textnormal{ mod } g) }\bigl|\{c:c|d, \mathrm{gcd}(c,r))=1\}\bigr|.
	\end{align}
	Moreover, the convergence also holds for the expectation, and in probability.
\end{theorem}
We refer to Section~\ref{sec:cycles} for the precise meaning of the bivariate convergence in probability in the last part of theorem.

\begin{example}[\bf Cycles of length at most $L$]
	For $g=1$, $D^w_n(L,1;(1))$ counts the number of cycles of $\bfw$ of length at most $L$.
	The limiting constant in~\eqref{eq:weighted_cycle_count} is equal to
	\begin{align}
		\frac{1}{d} \sum_{r=1}^d \bigl|\{ c: c|d, \mathrm{gcd}(c,r)=1\}\bigr| = \frac{1}{d} \sum_{c\in[d]:c\mid d} \bigl|\{r\in [d]: \mathrm{gcd}(c,r)=1\}\bigr| =  \sum_{c\in[d]:c|d} \frac{\phi(c)}{c},
\end{align}
	where $\phi$ is Euler's totient function.
This sequence starts as follows, 
	\begin{align}\label{eq:seqphi_0}
	\left(\sum_{c\in[d]:c|d} \frac{\phi(c)}{c}\right)_{d\geq1} = \left(1, \frac{3}{2}, \frac{5}{3}, 2, \frac{9}{5}, \frac{5}{2}, \frac{13}{7}, \frac{5}{2}, \frac{7}{3}, \frac{27}{10},\dots\right).
	\end{align}
	For example, if two words have respectively exponents $1$ and $2$, then the quantity $D^w_n(L,1;(1))/\log(L)$, for $L$ and $n$ large enough, will concentrate respectively around $1$ and $\frac{3}{2}$, which enables one to discriminate these two cases, with high probability. Since there are repeated values in~\eqref{eq:seqphi_0}, this observable is not sufficient to discriminate arbitrary exponents (for example the cases $d=6$ and $d=8$ share the value $5/2$, but other repetitions occur later in the sequence. 
\end{example}

Despite the previous example, in Section~\ref{sec:cycles} we will see that we can choose finitely many observables of the form $D^w_n(L,g;\mathbf{z})$ to uniquely determine all exponents, leading to the following corollary.

\begin{corollary}[\bf Guessing the exponent]\label{cor:distinguish_d}
	The exponent of $w$ can be recovered from a single sample of $\bfw$ with probability tending to $1$ as $n$ tends to infinity.
	In particular, if $w, w'$ are two words with different exponents, then $d_{TV}(\bfw,\bfw')\rightarrow 1$ when $n$ goes to infinity.
\end{corollary}

\begin{remark}\label{rem:cyclesPermutations}
	We also obtain, for fixed $L\geq 1$, the limiting joint distribution of the vector $(C^w_{1,n},\dots,C^w_{L,n})$ (Theorem~\ref{thm:main_cycle} in Section~\ref{sec:cycles}).
Anticipating over Section~\ref{sec:permutations}, the limit is the same as the one appearing when composing random permutations instead of random functions. This law was first described by Nica~\cite{Nica1994} for the $d$-th power of a single random permutation, and shown to be the limit for composition of independent random permutations according to a word of exponent $d$ in~\cite{KammounMaida2022}. These references state that the limiting law depends only $d$, but do not seem to establish that, when $L$ goes to infinity, these distributions actually become fully separated in TV-distance. The proof of the corollary above therefore also establishes this fact; see Section~\ref{sec:cycles}.
\end{remark}

\subsubsection{Auto-correlation constants and TV-separation}

We have seen in Theorem~\ref{thm:main_leaves} that the expected number of leaves satisfies $\bbE[Q_n^w(u)]\sim  g_k(u\eta_k)n$, a quantity that depends only on $k=|w|$. To obtain an asymptotic behaviour that depends more subtly on the word $w$, we study the second moment and variance of $Q_n^w(u)$ (at the price of much more technical difficulties, as we will see). We have: 
\begin{theorem}[\bf Variance and second moment for the number of leaves]\label{thm:main_variance}
For any integer $k\geq 1$, $w\in \{a,b\}^k$ and $u\in [0,1]$,  we have
\begin{align}%
	\frac{\Var(Q_n^w(u))}{n} \xrightarrow{n \to \infty} c(w,u),
\end{align}
	where the function $c(w,u)$ is a computable ``measure of auto-correlation'' of $w$, whose explicit expression is given in Section~\ref{subsec:def-cw-selfcontained}.
	Moreover, if $w$ and $w'$ are not isomorphic, the functions $u\mapsto c(w,u)$ and $u\mapsto c(w',u)$ are restrictions of distinct analytic functions, and they are distinct almost everywhere on $[0,1]$.
\end{theorem}

We will deduce the following consequence of the previous results:
\begin{theorem}[\bf TV-separation]\label{thm:cw_separate}
	If $w,w'$ are non-isomorphic words,
	then $\liminf_{n\to\infty} d_{TV}(\bfw,\bfw')>0$.
	In particular, in the urn formulation of the problem (see Remark~\ref{rem:urn}), one can guess the correct urn with probability arbitrarily close to $1$ in a bounded number of samples, with high probability.
\end{theorem}
The proof of Theorem~\ref{thm:main_variance}, the explicit expression for $c(w,u)$, and their consequence Theorem~\ref{thm:cw_separate} occupy the longer and most technical part of this paper.

\begin{figure}
	\centering
	\begin{minipage}{0.70\linewidth}
	\includegraphics[height=6.8cm]{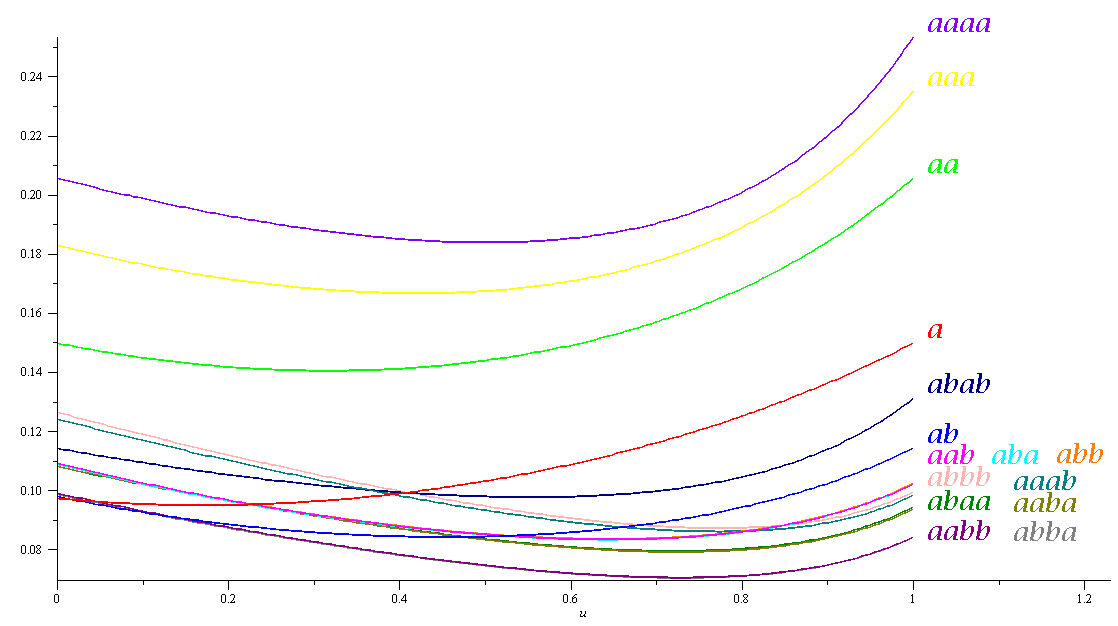}
	\end{minipage}
	\hspace{10mm}
	\begin{minipage}{0.19\linewidth}
		\includegraphics[height=3.3cm]{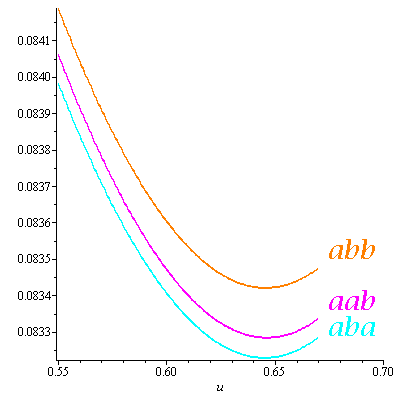}
		\includegraphics[height=3.3cm]{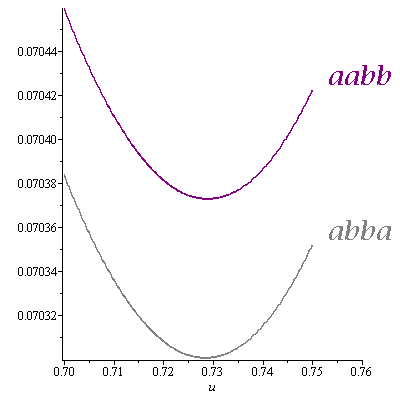}
	\end{minipage}
	\caption{Left: Plot of the constant $c(w,u)$ for the 15 non-isomorphic words $w$ of length at most $4$. As stated in Theorem~\ref{thm:main_variance}, they are pairwise distinct as analytic functions and thus almost everywhere. Right: zoom on portions of curves that cannot be distinguished on the main picture.}
	\label{fig:cw-statistics}
\end{figure}

For the interested reader, we provide a Maple worksheet accompanying this paper that computes $c(w,u)$ in closed form for any given word~\cite{maple}. 

 Figure~\ref{fig:cw-statistics} shows a plot of this function for all words of length at most $4$.
Table~\ref{tablec} provides the exact values of $c(w,u)$, as well as approximated values for $c(w,0)$, for all non-isomorphic words of length at most $3$. Note that these expressions quickly get more convoluted as the length of the word increases; as an illustration we also include the value of $c(w,u)$ for $w=aaaba$.

\begin{table}[h!!!]
\centering
	{\footnotesize
\def\arraystretch{1.5}
\begin{tabular}{|>{\centering\arraybackslash}p{8mm}|>{\centering\arraybackslash}p{12cm}|>{\centering\arraybackslash}p{13mm}|}
\hline
	$w$ & ${c}(w,u) $ & ${c}(w,0)\approx$
\\ \hline 
$a$
 &
 $
 -3\,u\eta_{{1}}g_{{1}}^{2}+g_{{1}}f_{{0,1,0}}-2\,g_{{1}}^{2}+2\,f_
{{0,1,0}}+h_{{1}}
$
 &
 $0.097209$\\ \hline 
$aa$
 &
 $
 -2\,u\eta_{{1}}\eta_{{2}}g_{{1}}g_{{2}}^{2}+2\,u\eta_{{1}}\eta_{{2}}
g_{{1}}g_{{2}}-2\,u\eta_{{2}}g_{{1}}g_{{2}}^{2}+2\,u\eta_{{2}}g_{{1}
}g_{{2}}-3\,g_{{1}}g_{{2}}^{2}+f_{{0,1,0}}g_{{2}}+2\,g_{{2}}f_{{0,1,
3}}-2\,g_{{2}}^{2}+h_{{2}}
$
 &
 $0.149768$\\ \hline 
$ab$
 &
 $
 -2\,{u}^{2}{\eta_{{1}}}^{2}{\eta_{{2}}}^{2}g_{{1}}^{2}g_{{2}}^{2}+
4\,{u}^{2}{\eta_{{2}}}^{2}g_{{1}}^{2}g_{{2}}^{2}\eta_{{1}}-2\,{u}^
{2}{\eta_{{2}}}^{2}g_{{1}}^{2}g_{{2}}^{2}-4\,g_{{1}}^{2}{g_{{2}}
}^{2}u\eta_{{1}}\eta_{{2}}+2\,u\eta_{{1}}\eta_{{2}}g_{{1}}g_{{2}}^{2
}+4\,u\eta_{{2}}g_{{1}}^{2}g_{{2}}^{2}-4\,u\eta_{{2}}g_{{1}}{g_{{2
}}}^{2}+2\,u\eta_{{2}}g_{{1}}g_{{2}}-2\,g_{{1}}^{2}g_{{2}}^{2}+g_{
{1}}g_{{2}}^{2}+f_{{0,1,0}}g_{{2}}-2\,g_{{2}}^{2}+h_{{2}}
$
 &
 $0.097922$\\ \hline 
$aaa$
 &
 $
 -2\,g_{{3}}^{2}u\eta_{{1}}\eta_{{2}}\eta_{{3}}g_{{1}}g_{{2}}+2\,u
\eta_{{1}}\eta_{{2}}\eta_{{3}}g_{{1}}g_{{2}}g_{{3}}-2\,g_{{3}}^{2}u
\eta_{{2}}\eta_{{3}}g_{{1}}g_{{2}}+2\,u\eta_{{2}}\eta_{{3}}g_{{1}}g_{{
2}}g_{{3}}-2\,ug_{{3}}^{2}\eta_{{3}}g_{{1}}g_{{2}}+2\,u\eta_{{3}}g_{
{1}}g_{{2}}g_{{3}}-2\,g_{{1}}g_{{2}}g_{{3}}^{2}+2\,g_{{2}}g_{{3}}f_{
{0,2,4}}-3\,g_{{2}}g_{{3}}^{2}+f_{{0,1,0}}g_{{3}}+2\,g_{{3}}f_{{0,1,
5}}-2\,g_{{3}}^{2}+h_{{3}}
$
 &
 $0.182808$\\ \hline 
$aab$
 &
 $
 -2\,{u}^{2}\eta_{{1}}{\eta_{{2}}}^{2}{\eta_{{3}}}^{2}g_{{1}}^{2}{g_{
{2}}}^{2}g_{{3}}^{2}+2\,{u}^{2}g_{{3}}^{2}{\eta_{{3}}}^{2}{g_{{1}}
}^{2}g_{{2}}^{2}\eta_{{1}}\eta_{{2}}-2\,{u}^{2}{\eta_{{2}}}^{2}{\eta
_{{3}}}^{2}g_{{1}}^{2}g_{{2}}^{2}g_{{3}}^{2}+4\,{u}^{2}{\eta_{{3
}}}^{2}g_{{1}}^{2}g_{{2}}^{2}\eta_{{2}}g_{{3}}^{2}-2\,{u}^{2}{
\eta_{{3}}}^{2}g_{{1}}^{2}g_{{2}}^{2}g_{{3}}^{2}-2\,g_{{2}}^{2
}g_{{3}}^{2}u\eta_{{1}}\eta_{{2}}\eta_{{3}}g_{{1}}-2\,g_{{1}}^{2}{
g_{{2}}}^{2}u\eta_{{2}}\eta_{{3}}g_{{3}}^{2}+2\,g_{{3}}^{2}u\eta_{
{1}}\eta_{{2}}\eta_{{3}}g_{{1}}g_{{2}}-4\,g_{{2}}^{2}u\eta_{{2}}\eta
_{{3}}g_{{1}}g_{{3}}^{2}+2\,ug_{{2}}^{2}g_{{3}}^{2}\eta_{{3}}g_{
{1}}f_{{2,4,1}}+2\,g_{{3}}^{2}u\eta_{{2}}\eta_{{3}}g_{{1}}g_{{2}}+4
\,ug_{{2}}^{2}g_{{3}}^{2}\eta_{{3}}g_{{1}}-4\,ug_{{3}}^{2}\eta_{
{3}}g_{{1}}g_{{2}}+2\,u\eta_{{3}}g_{{1}}g_{{2}}g_{{3}}-2\,g_{{2}}^{2
}g_{{1}}g_{{3}}^{2}+2\,g_{{2}}g_{{3}}^{2}f_{{1,2,1}}-2\,g_{{2}}^
{2}g_{{3}}^{2}+g_{{2}}g_{{3}}^{2}+f_{{0,1,0}}g_{{3}}-2\,g_{{3}}^
{2}+h_{{3}}
$
 &
 $0.109260$\\ \hline 
$aba$
 &
 $
 -2\,{u}^{2}{\eta_{{1}}}^{2}{\eta_{{2}}}^{2}{\eta_{{3}}}^{2}g_{{1}}^{
2}g_{{2}}^{2}g_{{3}}^{2}+4\,{u}^{2}\eta_{{1}}{\eta_{{2}}}^{2}{\eta
_{{3}}}^{2}g_{{1}}^{2}g_{{2}}^{2}g_{{3}}^{2}-2\,{u}^{2}g_{{3}}
^{2}{\eta_{{3}}}^{2}g_{{1}}^{2}g_{{2}}^{2}\eta_{{1}}\eta_{{2}}-2\,
{u}^{2}{\eta_{{2}}}^{2}{\eta_{{3}}}^{2}g_{{1}}^{2}g_{{2}}^{2}{g_{{
3}}}^{2}+2\,g_{{1}}^{2}g_{{2}}^{2}{u}^{2}\eta_{{2}}\eta_{{3}}{g_{{
3}}}^{2}\eta_{{1}}-4\,u\eta_{{1}}\eta_{{2}}\eta_{{3}}g_{{1}}^{2}{g_{
{2}}}^{2}g_{{3}}^{2}+4\,g_{{2}}^{2}g_{{3}}^{2}u\eta_{{1}}\eta_{{
2}}\eta_{{3}}g_{{1}}+4\,g_{{1}}^{2}g_{{2}}^{2}u\eta_{{2}}\eta_{{3}
}g_{{3}}^{2}-4\,g_{{3}}^{2}u\eta_{{1}}\eta_{{2}}\eta_{{3}}g_{{1}}g
_{{2}}-4\,g_{{2}}^{2}u\eta_{{2}}\eta_{{3}}g_{{1}}g_{{3}}^{2}-2\,{g
_{{1}}}^{2}g_{{2}}^{2}g_{{3}}^{2}u\eta_{{3}}+2\,u\eta_{{1}}\eta_{{
2}}\eta_{{3}}g_{{1}}g_{{2}}g_{{3}}+2\,g_{{3}}^{2}u\eta_{{2}}\eta_{{3
}}g_{{1}}g_{{2}}+2\,g_{{2}}^{2}g_{{1}}g_{{3}}^{2}uf_{{2,3,1}}-2\,u
g_{{3}}^{2}\eta_{{3}}g_{{1}}g_{{2}}-2\,g_{{1}}^{2}g_{{2}}^{2}{g_
{{3}}}^{2}+2\,u\eta_{{3}}g_{{1}}g_{{2}}g_{{3}}+4\,g_{{2}}^{2}g_{{1}}
g_{{3}}^{2}-4\,g_{{1}}g_{{2}}g_{{3}}^{2}-2\,g_{{2}}^{2}g_{{3}}
^{2}+2\,g_{{2}}g_{{3}}f_{{0,2,1}}+g_{{2}}g_{{3}}^{2}+f_{{0,1,0}}g_{{
3}}-2\,g_{{3}}^{2}+h_{{3}}
$
 &
 $0.108728$\\ \hline 
$abb$
 &
 $
 -2\,{u}^{2}{\eta_{{1}}}^{2}{\eta_{{2}}}^{2}{\eta_{{3}}}^{2}g_{{1}}^{
2}g_{{2}}^{2}g_{{3}}^{2}+4\,{u}^{2}g_{{3}}^{2}{\eta_{{3}}}^{2}{g
_{{1}}}^{2}g_{{2}}^{2}\eta_{{1}}\eta_{{2}}-2\,{u}^{2}{\eta_{{3}}}^{2
}g_{{1}}^{2}g_{{2}}^{2}\eta_{{2}}g_{{3}}^{2}-2\,{u}^{2}{\eta_{{3
}}}^{2}g_{{1}}^{2}g_{{2}}^{2}g_{{3}}^{2}+2\,g_{{1}}^{2}{g_{{2}
}}^{2}g_{{3}}^{2}{u}^{2}\eta_{{3}}f_{{3,4,1}}-4\,u\eta_{{1}}\eta_{{2
}}\eta_{{3}}g_{{1}}^{2}g_{{2}}^{2}g_{{3}}^{2}+2\,g_{{3}}^{2}u
\eta_{{1}}\eta_{{2}}\eta_{{3}}g_{{1}}g_{{2}}+4\,g_{{1}}^{2}g_{{2}}
^{2}g_{{3}}^{2}u\eta_{{3}}-2\,g_{{3}}^{2}u\eta_{{2}}\eta_{{3}}g_{{
1}}g_{{2}}-2\,ug_{{2}}^{2}g_{{3}}^{2}\eta_{{3}}g_{{1}}+2\,u\eta_{{
3}}g_{{1}}g_{{2}}g_{{3}}f_{{0,4,1}}-4\,ug_{{3}}^{2}\eta_{{3}}g_{{1}}
g_{{2}}+2\,g_{{1}}g_{{2}}g_{{3}}^{2}uf_{{1,3,1}}-2\,g_{{1}}^{2}{g_
{{2}}}^{2}g_{{3}}^{2}+2\,u\eta_{{3}}g_{{1}}g_{{2}}g_{{3}}+2\,g_{{1}}
g_{{2}}g_{{3}}^{2}-3\,g_{{2}}g_{{3}}^{2}+f_{{0,1,0}}g_{{3}}+2\,g_{
{3}}f_{{0,1,1}}-2\,g_{{3}}^{2}+h_{{3}}
$
 &
 $0.109157$

	 \\ \hline
	$aaaba$ &
	$\tiny
-2\,g_{{5}}^{2}g_{{2}}g_{{3}}g_{{4}}+2\,u\eta_{{5}}g_{{1}}g_{{2}}g_{
{3}}g_{{4}}g_{{5}}-2\,u\eta_{{5}}g_{{1}}g_{{2}}g_{{3}}^{2}g_{{4}}^
{2}g_{{5}}^{2}-2\,ug_{{5}}^{2}\eta_{{5}}g_{{1}}g_{{2}}g_{{3}}g_{{4
}}-2\,g_{{5}}^{2}{u}^{2}{\eta_{{3}}}^{2}{\eta_{{4}}}^{2}{\eta_{{5}}}
^{2}g_{{1}}^{2}g_{{2}}^{2}g_{{3}}^{2}g_{{4}}^{2}\eta_{{1}}\eta
_{{2}}-2\,u\eta_{{1}}\eta_{{2}}\eta_{{3}}\eta_{{4}}\eta_{{5}}g_{{1}}g_
{{2}}g_{{3}}g_{{4}}g_{{5}}^{2}-2\,u\eta_{{1}}\eta_{{2}}\eta_{{3}}
\eta_{{4}}\eta_{{5}}g_{{1}}g_{{2}}g_{{3}}^{2}g_{{4}}^{2}g_{{5}}^
{2}+2\,g_{{5}}^{2}g_{{4}}^{2}u\eta_{{1}}\eta_{{2}}\eta_{{3}}\eta_{
{4}}\eta_{{5}}g_{{1}}g_{{2}}g_{{3}}+2\,u\eta_{{1}}\eta_{{2}}\eta_{{3}}
\eta_{{4}}\eta_{{5}}g_{{1}}g_{{2}}g_{{3}}g_{{4}}g_{{5}}+2\,g_{{5}}^{
2}{u}^{2}{\eta_{{4}}}^{2}{\eta_{{5}}}^{2}g_{{1}}^{2}g_{{2}}^{2}{g_
{{3}}}^{2}g_{{4}}^{2}\eta_{{1}}\eta_{{2}}\eta_{{3}}-2\,g_{{3}}^{2}
g_{{4}}^{2}g_{{5}}^{2}g_{{2}}+2\,g_{{4}}g_{{5}}f_{{0,2,1}}+2\,{g_{
{4}}}^{2}g_{{5}}^{2}g_{{2}}g_{{3}}f_{{2,4,2}}+2\,g_{{2}}g_{{3}}g_{{4
}}g_{{5}}f_{{0,4,1}}-2\,u\eta_{{2}}\eta_{{3}}\eta_{{4}}\eta_{{5}}g_{{1
}}g_{{2}}g_{{3}}g_{{4}}g_{{5}}^{2}+2\,g_{{1}}^{2}g_{{2}}^{2}{g_{
{3}}}^{2}g_{{4}}^{2}g_{{5}}^{2}{u}^{2}\eta_{{5}}\eta_{{4}}f_{{5,7,
1}}+2\,u\eta_{{4}}\eta_{{5}}g_{{1}}g_{{2}}g_{{3}}g_{{4}}g_{{5}}f_{{0,7
,1}}+2\,g_{{5}}^{2}g_{{3}}^{2}g_{{4}}^{2}u\eta_{{4}}\eta_{{5}}g_
{{1}}g_{{2}}f_{{3,7,3}}+2\,g_{{2}}^{2}g_{{3}}^{2}g_{{4}}^{2}{g_{
{5}}}^{2}u\eta_{{5}}g_{{1}}\eta_{{4}}f_{{4,7,2}}-2\,g_{{4}}^{2}{g_{{
5}}}^{2}-2\,g_{{5}}^{2}g_{{1}}g_{{2}}g_{{3}}g_{{4}}-2\,u\eta_{{2}}
\eta_{{3}}\eta_{{4}}\eta_{{5}}g_{{1}}g_{{2}}g_{{3}}^{2}g_{{4}}^{2}
g_{{5}}^{2}+2\,g_{{4}}^{2}u\eta_{{3}}\eta_{{4}}\eta_{{5}}g_{{1}}g_
{{2}}g_{{3}}g_{{5}}^{2}\eta_{{2}}+2\,{u}^{2}{\eta_{{4}}}^{2}{\eta_{{
5}}}^{2}g_{{1}}^{2}g_{{2}}^{2}g_{{3}}^{2}g_{{4}}^{2}\eta_{{3}}
g_{{5}}^{2}\eta_{{2}}-2\,{u}^{2}{\eta_{{3}}}^{2}{\eta_{{4}}}^{2}{
\eta_{{5}}}^{2}g_{{1}}^{2}g_{{2}}^{2}g_{{3}}^{2}g_{{4}}^{2}{g_
{{5}}}^{2}\eta_{{2}}-4\,g_{{5}}^{2}g_{{3}}g_{{4}}-2\,{u}^{2}{\eta_{{
4}}}^{2}{\eta_{{5}}}^{2}g_{{1}}^{2}g_{{2}}^{2}g_{{3}}^{2}{g_{{4}
}}^{2}g_{{5}}^{2}+4\,g_{{5}}^{2}g_{{3}}^{2}g_{{4}}^{2}u\eta_{{
4}}\eta_{{5}}g_{{1}}g_{{2}}+2\,u\eta_{{4}}\eta_{{5}}g_{{1}}g_{{2}}g_{{
3}}g_{{4}}g_{{5}}^{2}-4\,u\eta_{{4}}\eta_{{5}}g_{{1}}g_{{2}}g_{{3}}{
g_{{4}}}^{2}g_{{5}}^{2}+2\,u\eta_{{3}}\eta_{{4}}\eta_{{5}}g_{{1}}g_{
{2}}g_{{3}}g_{{4}}g_{{5}}f_{{0,8,1}}-2\,g_{{3}}^{2}g_{{1}}g_{{2}}{g_
{{4}}}^{2}g_{{5}}^{2}+h_{{5}}-4\,g_{{3}}^{2}g_{{4}}^{2}g_{{5}}
^{2}u\eta_{{3}}\eta_{{4}}\eta_{{5}}g_{{1}}g_{{2}}-2\,g_{{5}}^{2}{g_{
{2}}}^{2}g_{{3}}^{2}g_{{4}}^{2}u\eta_{{3}}\eta_{{4}}\eta_{{5}}g_{{
1}}-2\,{u}^{2}g_{{5}}^{2}{\eta_{{5}}}^{2}g_{{1}}^{2}g_{{2}}^{2}{
g_{{3}}}^{2}g_{{4}}^{2}\eta_{{3}}\eta_{{4}}-2\,g_{{5}}^{2}{u}^{2}{
\eta_{{3}}}^{2}{\eta_{{4}}}^{2}{\eta_{{5}}}^{2}g_{{1}}^{2}g_{{2}}^
{2}g_{{3}}^{2}g_{{4}}^{2}+4\,{u}^{2}{\eta_{{4}}}^{2}{\eta_{{5}}}^{
2}g_{{1}}^{2}g_{{2}}^{2}g_{{3}}^{2}g_{{4}}^{2}\eta_{{3}}{g_{{5
}}}^{2}-2\,g_{{5}}^{2}+2\,g_{{3}}g_{{4}}^{2}g_{{5}}^{2}f_{{2,3,3
}}+2\,g_{{3}}g_{{4}}g_{{5}}f_{{0,3,1}}+g_{{5}}^{2}g_{{4}}-4\,{g_{{5}
}}^{2}u\eta_{{3}}\eta_{{4}}\eta_{{5}}g_{{1}}g_{{2}}g_{{3}}g_{{4}}+4\,{
g_{{4}}}^{2}u\eta_{{3}}\eta_{{4}}\eta_{{5}}g_{{1}}g_{{2}}g_{{3}}{g_{{5
}}}^{2}-2\,g_{{1}}^{2}g_{{2}}^{2}g_{{3}}^{2}g_{{4}}^{2}u\eta_{
{3}}\eta_{{4}}\eta_{{5}}g_{{5}}^{2}+g_{{5}}f_{{0,1,0}}+2\,g_{{4}}^
{2}g_{{5}}^{2}g_{{1}}g_{{2}}g_{{3}}uf_{{2,5,1}}-2\,g_{{3}}^{2}{g_{
{4}}}^{2}g_{{5}}^{2}+4\,g_{{3}}g_{{4}}^{2}g_{{5}}^{2}
$
	& 0.110714
	 \\ \hline
\end{tabular}
	\caption{Values of $c(w,u)$ for non-isomorphic words of length at most 3, and an example of length $5$. Here for fixed length $k=|w|$ we use the shorcut notation $g_i=g_i(u \eta_k)$, $h_k=g_k(u^2 \eta_k)$, and moreover we write $f_{i,j,\ell}$ for the function $F_2^{i,j,\ell}(u)$ explicitly given in Proposition~\ref{prop:cross_monster}. All these functions are towers or branched towers of exponentials, for example for $k=3$ the function $f_{0,2,1}$ appearing in the case $w=aba$ is equal to
	$f_{0,2,1}=g_1(u\eta_3g_2(u\eta_3))=
{{\rm e}^{{{\rm e}^{{{\rm e}^{u{{\rm e}^{-1+{{\rm e}^{-1+{{\rm e}^{-1}
}}}}}-1}}-1}}u{{\rm e}^{-1+{{\rm e}^{-1+{{\rm e}^{-1}}}}}}-1}}
	$.
	}
\label{tablec}
	}
\end{table}

In the way of establishing Theorem~\ref{thm:cw_separate} we will also obtain bounds for higher-order moments of $Q_n^w(u)$: in Section~\ref{sec:higher_moments} we show that 
\begin{align}
\bbE[(Q_n^w(u)-g_k(u\eta_k) n)^m] = O(n^{\lfloor m/2\rfloor}).
\end{align}
These bounds imply that $Q_n^w(u)$, under appropriate renormalization, is sub-Gaussian. Following the approach for the exact computation of the variance in Section~\ref{sec:second_moments}, one can show that $n^{-1/2}(Q_n^w(u)-g_k(u\eta_k) n)$ converges in distribution to a Gaussian variable.

\subsection{Random automata, random permutations, and comments}
\label{sec:permutations}

As mentioned above, the study of the function $\bfw$ appears naturally, although maybe not as explicitly stated as in this paper, in the field of random automata, which is a classical subject at the intersection of random graphs and computer science~\cite{Nicaud:survey, CaiDevroye,QuattropaniSau,ABBP:rout-digraphs}. 

In a recent work~\cite{GCGP:synchronisation}, motivated by the study of the shortest synchronisation word for random automata~\cite{KKS-sqrtn, Nicaud}, we introduced the notion of \emph{$w$-trees}, which are automata in which the $w$-transitions induce a cycle-rooted tree. 
A key element in that work is to show that a random automaton is a $w$-tree with probability asymptotic to $\frac{|w|}{n}$, uniformly for all $w$ of at most logarithmic length. In a sense, the bulk of~\cite{GCGP:synchronisation} shows that, for the function property of being a cycle-rooted tree, the random function $\bfw$ behaves like a uniform random function, in some qualitative and quantitative sense. The upper bound on the shortest synchronizing word in~\cite{GCGP:synchronisation} was later slightly tightened by Martinsson in the preprint~\cite{Martinsson}, with no references to $w$-trees, but using $w$-transitions for some ``structured'' words $w$.

The common difficulty in all these papers comes from the fact that it is difficult to analyze the $w$-transitions of elements in $[n]$ simultaneously, due to their lack of independence.
In some way, they all have to do with the question of understanding the random function $\bfw$ for various choices of the word $w$, which led us to formalizing the question in this paper. 

\smallskip

Apart from random automata, another natural motivation for this work comes from random permutations. Our main question can be asked similarly if the uniform random functions $\mathbf{a}, \mathbf{b}$ are replaced by uniform elements $\sigma,\pi$ in the symmetric group $\mathfrak{S}_n$. Note that in this context, it is also natural to allow the symbols $\sigma^{-1}$ and $\pi^{-1}$ in the word $w$.
The distribution of the number of short cycles in such compositions is well understood~\cite{Nica1994, HananyPuder2020,KammounMaida2020, KammounMaida2022} with a limit law depending only on the exponent of the underlying word. This problem has deep connections to free probability; see the introduction of~\cite{KammounMaida2022}.
As it turns out, the limiting distribution of short cycles in our setting is the same as for random permutations, and in particular some of our results also apply to the latter case; see Remark~\ref{rem:cyclesPermutations}.

In the context of random permutations, the question of the TV-distance to the uniform measure on the symmetric or alternating group could be studied by means of representation theory, at least for certain simple words $w$ which have some symmetries related to conjugacy invariance. We refer to Féray's formulation of the problem~\cite{feray}, who mentions that the smallest word for which no such property is known is $w=\sigma\tau\sigma^2\tau^2$.

Note that the references above often consider the case of an alphabet of arbitrary fixed size  rather than $\{a,b\}$: we leave the reader to check that all our proof techniques and results are easily adapted to this context. Note also that considering an alphabet of size growing with $n$ would not make sense in our setting, since we only work with words $w$ of fixed length --- however, this could lead to interesting new questions.

\smallskip

To conclude this discussion, let us also mention that random functions $[n]\rightarrow [n]$ are a standard and very well studied object in discrete probability, with deep relations to random trees, see e.g.~\cite{Flajolet1989RandomMS, Pitman2006, GCGP:AMM}. A natural model interpolating between uniform random functions and uniform random permutations 
is to consider a random function $\mathbf{f}$ chosen with probability proportional to the weight $x^{c(\mathbf{f})}$, where $c(\mathbf{f})$ denotes the number of vertices belonging to a cycle in $\mathbf{f}$, and $x\geq 0$ is a real parameter. In this distribution of ``Boltzmann'' type, $x$ acts as a fugacity parameter controlling the ``cyclicity'' of the function. The classical ensembles are recovered as limiting cases of this parameter: if $x \to 0$, the measure concentrates on functions with minimal cyclic content (only one single fixed point), yielding uniform rooted Cayley trees; if $x = 1$, the weight is uniform for all $\mathbf{f}$, yielding uniform random functions; and if $x \to \infty$, the measure concentrates on functions with maximal cyclic content (i.e. where every vertex is cyclic), yielding uniform random permutations.
All the questions asked in this paper for uniform random functions could be studied for this general model, which might lead to interesting new phenomena.

\subsection{The case $u=0$ and the first version of this article
}

In this section we discuss the differences between this article and the first version that was made public (still available on arXiv as 2603.28936v1, \cite{ourpaper_v1}).

In the first version, we where only considering the case $u=0$, i.e. our main observable was not $Q_n^w(u)$ but the number of leaves of $\bfw$. In that setting we were not able to prove Theorem~\ref{thm:cw_separate} fully: we could prove it only under the hypothesis that $c(w,0)\neq c(w',0)$. We conjectured that $c(w,0)\neq c(w',0)$ for non-isomorphic words $w,w'$, and we proved this conjecture conditionally to Schanuel's conjecture, a major open problem in transcendental number theory.  Because we now establish TV-separation unconditionally, we have removed the discussion on $c(w,0)$ and this conditional theorem in this new version --- the interested reader can consult the first version.

In this first version we were also considering a variant of the variance, obtaining a similar but different limit constant $\tilde{c}(w)$, with analogue results. Now that we can prove Theorem~\ref{thm:cw_separate} unconditionally, the extra work needed to get these further results seems disproportionate and we have removed them. In this removed part we were obtaining a precise limit theorem for the number of ``cyclic leaves'' (leaves of the function $\bfw$ containing in some sense a cycle of the underlying $\{a,b\}$-digraph) at the price of introducing relatively subtle variants of branching processes with self-replication, which we completely avoid here. Again, we refer the interested reader to the first version.

At a technical level, this new version is very similar to the first one, in some sense we only ``introduce'' the parameter $u$ all around. It is somehow even simpler since the discussion of the Schanuel conjecture, and the extra work related to the variant constant $\tilde{c}(w)$, have been replaced by the asymptotic independence discussion in Section~\ref{sec:independence}.
So the cost of introducing the parameter $u$ is small, while now we obtain an unconditional TV-separation result in Theorem~\ref{thm:cw_separate}.

\subsection{Plan of the paper}

In Section~\ref{sec:GW}, we start by recalling standard facts about Bienaymé-Galton-Watson trees (BGWTs)  and how they approximate local in-neighbourhoods of uniform random functions. This section also establishes the convergence of the expected count of quasi-leaves $Q_n(u)$, in Theorem~\ref{thm:main_leaves}, and sets the necessary notation and toolbox for the next sections where more involved local approximations are studied, both at the first and second order.

In Section~\ref{sec:second_moments}, we compute the limiting variance of the number of leaves of $\bfw$, which is of linear order with leading constant $c(w,u)$.
This requires to study second-order approximations of local neighbourhoods in our random functions by BGWTs and their variants. 
In particular we introduce certain coupled pairs of BGWTs that describe in some sense the second-order corrections to the second moment of $Q_n(u)$.
At the end of the section we provide a self-contained explicit description of the constants $c(w,u)$ (Section~\ref{subsec:def-cw-selfcontained}).

In Section~\ref{sec:independence} we prove that the functions $c(w,u)$ characterize the word $w$ up to isomorphism. This part is largely independent from the rest of the paper and relies on proving that the building blocks (towers and branched towers of exponentials) that appear in the explicit expressions of $c(w,u)$ have some extent of algebraic independence, which we prove by considering their behaviour as $u\rightarrow \infty$. We also discuss the link with Schanuel's conjecture and the first version of this work; see Remark~\ref{rem:old_schanuel}.

In Section~\ref{sec:high_moments}, we focus on higher-order  moments of $Q_n(u)$. Our main result here bounds the $q$-th (shifted) moment, hinting a Gaussian-type behaviour. In particular, this gives us control on its fourth moment and enables us to translate differences in variance into  TV-distance separation, proving Theorem~\ref{thm:cw_separate}. 

Finally, Section~\ref{sec:cycles} is devoted to the short cycle distribution of $\bfw$ and the proof of Theorem~\ref{thm:Dn_limit}.

\section{First order statistic on quasi-leaves}
\label{sec:GW}

From now on we fix $k\geq 1$ and the word $w=w_1 w_2\dots w_k$. We will often view the functions $\mathbf{a}, \mathbf{b}, \bfw$ as directed graphs. More precisely, an \emph{$\{a,b\}$-digraph} $H$ is a labelled directed graph in which each edge has a type in $\{a,b\}$; the labels will be specified in each context. We use $|H|$ to denote the order of the graph. Henceforth, we will identify the pair of  functions $\mathbf{a},\mathbf{b}$ with the corresponding $\{a,b\}$-digraph (with an edge of type $x$ from $i$ to $\mathbf{x}(i)$ for each $x \in \{a,b\}$ and $i \in [n]$). Similarly we will identify $\bfw$ with its directed graph on $[n]$. 

\subsection{Plane trees vs labelled trees}
\label{subsec:auto}

All trees considered in this paper are rooted. Depending on context, the trees will be \emph{labelled} (with a vertex set which is a subset of positive integers), or \emph{plane} (with an order on the children of each vertex, say from left to right).

Because neighborhoods of vertices in a labelled graph are naturally labelled, and BGWTs are naturally plane, we will have to work with both. Readers already familiar with trees and their automorphisms will find this very natural and may skip this part. For others, the purpose of this section is to clarify the relation between the two notions and to set up the notation that we will use throughout, sometimes implicitly.

We use $\mathfrak{L}$ to denote the set of labelled trees, where a tree on $n$ vertices is labelled with labels in $[n]$, and use capital letters such as $T$ to refer to one of its elements. We use $\mathfrak{P}$ to denote the set of plane trees and we use Greek letters such as $\tau$ to refer to one of its elements. We use $|T|$ and $|\tau|$ to denote the order of $T$ and $\tau$, respectively. 

We will write $\congplane$ to say that two rooted plane trees are equal.  
We will also use the equivalence relation $\conglab$ defined on rooted labelled graphs having positive integer labels (but possibly spanning different sets, for example subgraphs of a given labelled graph) defined as follows:  $T\conglab T'$ if and only if the unique increasing function $V(T)\rightarrow V(T')$ is a root-preserving graph isomorphism.

If $U$ is a rooted unlabelled tree with $n$ vertices, the number of inequivalent labellings of $U$ with labels in $[n]$ is equal to $\textrm{lab}(U)=\frac{n!}{\mathrm{Aut}_r(U)}$, where $\mathrm{Aut}_r(U)$ is the number of root-preserving automorphisms of $U$. Note also that the number of ways to equip $U$ with both a plane structure and a labelling in $[n]$ is $\textrm{lab(U)}\prod_{v\in V(U)} n_v!$, where $n_v$ is the number of children of the vertex~$v$, since once a labelling is chosen, all orderings of children give inequivalent plane structures as labelled objects. The number of such structures can also be counted by first fixing one of the $\textrm{pl}(U)$ many inequivalent representations of $U$ as a plane tree, and then labelling the vertices in one of the $n!$ different ways. Therefore the number of automorphisms, plane representations, and labellings of the tree $U$ are related by
\begin{align}\label{eq:triple}
\frac{\textrm{pl}(U)}{\prod_{v\in V(U)} n_v!} = \frac{\textrm{lab}(U)}{n!} = \frac{1}{\mathrm{Aut}_r(U)}.
\end{align}
In what follows we will sometimes need to go back and forth between plane, non-plane, labelled, unlabelled structures in a similar way, but in slightly more subtle situations. 

Finally, in order to compare the neighbourhood of a random vertex in a labelled graph (which is labelled itself) to a BGWT (which is naturally plane), we will use the following notation: if $\tau\in \mathfrak{P}$ is a plane tree, we let $\tau^{(lab)}\in\mathfrak{L}$ be the random labelled tree obtained by labelling uniformly at random the vertices of $\tau$ with labels in $[|\tau|]$, and forgetting the plane structure. This notation will often be used in cases where $\tau$ itself is already a random tree.

We could have avoided talking about labelled trees in this paper and use unlabelled trees everywhere, using the graph-isomorphism relation $\cong$ 
instead of $\conglab$ and carrying inverse automorphism factors in formulas rather than inverse factorials. However, working with labelled graphs avoids the need to discuss graph automorphisms, which would become subtle in some of the most involved sections. Our choice also leads to slightly stronger results since labels are controlled in the approximations.

\subsection{Poissonian Bienaymé-Galton-Watson Trees}
\label{subsec:GW_gf}

For a real $\lambda \geq 0$, we use Po($\lambda$) to denote a Poisson random variable of parameter $\lambda$.
A \emph{Po(1)-Bienaymé-Galton-Watson Tree} (\emph{Po(1)-BGWT} for short, also referred to as \emph{branching process}) is a stochastic process $\mathbf{Z}=(Z_L)_{L\geq 0}$  defined by  $Z_0 \defi 1$ and $Z_{L+1} \defi \sum_{i=1}^{Z_L} \xi_{L,i}$,
where the $\xi_{L,i}$ are i.i.d. Poisson random variables with mean $\lambda=1$.  
For all\footnote{In this paper $\mathbb{N}=\{0,1,2,3,\dots\}$.} $A \subseteq \mathbb{N}$, we write $Z_A\defi\sum_{l\in A} Z_l$, and $Z\defi Z_\mathbb{N}$ for the total progeny. 

We denote by $\PBGW$ the probability distribution of $\mathbf{Z}$, and by $\EBGW$ its expectation. (We will occasionally use branching processes with different offspring distribution; if the distribution is clear from the context, we will also use the same notation to refer to its probability function and expectation.)

We often use the graphical representation of the branching process as a rooted plane tree, that we denote by $\tau_{\mathbf{Z}}$. Indeed, 
the rooted tree associated to a branching process is naturally equipped with a plane structure, in the sense that the children of each vertex are ordered (from left to right, say). We use $\muBGW$ to denote the distribution on plane trees induced by $\tau_{\mathbf{Z}}$: 
\begin{align}\label{eq:GW}
\muBGW(\tau) \defi\PBGW(\tau_{\mathbf{Z}}=\tau) = \frac{e^{-|\tau|}}{\prod_{v\in V(\tau)} n_v!},\quad \text{for all } \tau\in\mathfrak{P}, 
\end{align}
where $n_v$ is the number of children of $v$ in $\tau$. Together with the previous discussion (Equation~\eqref{eq:triple}) this has the following classical consequence: 
\begin{align}\label{eq:classical}
	\PBGW(\tau_{\mathbf{Z}}^{(lab)}=T) = \frac{e^{-|T|}}{|T|!},\qquad \text{for all } T\in\mathfrak{L}, 
\end{align}
and in particular this probability is uniform once conditioned on trees of given size (we will not use this last fact).

The probability generating function of a Po(1) variable $\xi$ is $f(s)  = e^{s-1},$ 
which directly implies that the probability generating function of the progeny at the $L$-th level, 
\begin{align}
g_{L}(s)\defi\EBGW[s^{Z_L}],
\end{align}
satisfies the recursion: $g_0(s) = s$ and 
\begin{align}\label{eq:recu_g}
g_{L}(s) = f(g_{L-1}(s)),\quad \text{for all }L\geq 1.
\end{align}
In particular, it coincides with the definition given~\eqref{eq:def_g}.

The extinction probability at time $L\geq 0$ is
\begin{equation}
	\eta_L\defi\PBGW(Z_L=0) = \EBGW[\mathbf{1}_{Z_L=0}]=g_L(0).
\end{equation}
It is useful to observe that
\begin{equation}\label{eq:SOEDS}
	\eta_L=\EBGW[\mathbf{1}_{Z_L=0}]= \EBGW[\eta_{L-i}^{Z_i}], \quad \text{for all }0\leq i\leq L,
\end{equation}
which is direct by formulas above, and also directly follows from the more refined statement
\begin{equation}\label{eq:SOEDS-conditional}
	\EBGW[\mathbf{1}_{Z_L=0}\mid Z_i]= \eta_{L-i}^{Z_i}, \quad \text{for all }0\leq i\leq L,
\end{equation}
which says that the tree is extinct by time $L$ if and only if the subtree produced by each individual present at time $i$ survives for less than $L-i$ generations.

It will be useful to introduce the multivariate generating function of the process:%
\begin{equation}\label{eq:GF_multivar}
G_L(x_0,x_1,\dots,x_L)
\defi\EBGW\bigl[x_0^{Z_0}x_1^{Z_1}\cdots x_L^{Z_L}\bigr].
\end{equation}
Similarly as in~\eqref{eq:recu_g}, we have: $G_0(x_0)=x_0$ and 
\begin{align}
	G_L(x_0,x_1,\dots,x_L)=x_0f\bigl(G_{L-1}(x_1, x_2,\dots,x_{L})\bigr)= x_0f\bigl(x_1 f(x_{2} f(\dots f(x_{L-1} f(x_L)),\quad \text{for all }L\geq 1.
\end{align}
This generating function gives access to moments: for any $m_1,\dots,m_L \in \mathbb{N}$ we obtain
\begin{equation}
	\EBGW\bigl[Z_1^{m_1}\cdots Z_{L}^{m_{L}} x_1^{Z_1}\dots x_{L}^{Z_{L}}] = \Bigl(\frac{x_1\partial}{\partial x_1}\Bigr)^{m_1}\hspace{-0.3cm}\cdots\Bigl(\frac{x_{L}\partial}{\partial x_{L}}\Bigr)^{m_{L}} G_{L}(x_0,x_1,\dots, x_{L}). \label{eq:momentsGF}
\end{equation}
In particular, by~\eqref{eq:SOEDS-conditional} we can compute, for $k<L$,
\begin{equation}\label{eq:WPEOW}
	\EBGW\bigl[ Z_1^{m_1}\cdots Z_{L}^{m_{L}}u^{Z_k} \mathbf{1}_{Z_L=0}\bigr] =
	\Xi_L^{\mathbf{x},u}\left[	\Bigl(\frac{x_1\partial}{\partial x_1}\Bigr)^{m_1}\hspace{-0.3cm}\cdots\Bigl(\frac{x_{L}\partial}{\partial x_{L}}\Bigr)^{m_{L}} G_{L}(x_0,x_1,\dots, x_{L})\right],
\end{equation}
where $\Xi_L^{\mathbf{x},u}$ is the operator that specializes $\mathbf{x}=(x_0,x_1,\dots,x_{L})$ to 
$(\underbrace{1,\dots,1}_{k},u,\underbrace{1,\dots,1}_{L-k-1},0)$, i.e. sets $x_k=u,x_{L}=0$,  and the rest of the $x_i$ to $1$.

\medskip

We will sometimes need to use several independent copies of the branching process $\mathbf{Z}$. We will denote them by $\mathbf{Z}^{(1)},\mathbf{Z}^{(2)}, \mathbf{Z}^{(3)},\dots$. For $m\in\mathbb{N}$, we will use the notation $\PBGW^{m}$ and $\EBGW^{m}$ for the probability distribution and expectation over $\mathbf{Z}^{(1)},\dots ,\mathbf{Z}^{(m)}$ i.i.d. Po(1)-BGWTs.

\subsection{In-balls of a vertex and universal covers}

Let $x\in [n]$ and $\omega\in\{a,b\}^L$. We write most of this section for a general word $\omega$, although we will set $\omega=w^2$ in our applications.
For $0\leq \ell\leq L$, let $\w^{(\ell)}= \w_{L-\ell+1}\dots \w_L$ be the suffix of $\w$ of length $\ell$ and $\bw^{(\ell)}=\phi_{\w_L}\circ \dots\circ \phi_{\w_{L-\ell+1}}$ the corresponding random function induced by the choices of $\mathbf{a,b}$. Define the \emph{$\w$-in-ball of $x$} (or just \emph{in-ball} if $\w$ is clear from context), $B_{\w}(x)$, as the rooted $\{a,b\}$-digraph with root $x$ obtained by adding an arc $(y,\phi_{\w_{L-\ell+1}}(y))$ with label $\w_{L-\ell+1}$ if there are $y\in [n]$ and $\ell\in [L]$ such that $\mathbf{\w}^{(\ell)}(y)=x$;
see Figure~\ref{fig:w-in-ball}.
Note that if the $\w$-in-ball is a tree, the arc labels can be recovered from their height in the tree and the knowledge of $\w$, but this is not true in general.

We also define the \emph{universal $\w$-in-cover of $x$} (or just \emph{in-cover}), $T_{\w}(x)$, as the root-directed labelled tree with root at $(x,0)$ whose vertices are labelled by $(y,\ell)$ with $\bw^{(\ell)}(y)=x$ and where we add the arc $(y,\ell)\to (y',\ell-1)$ if $\phi_{\w_{L-\ell+1}}(y)=y'$. 
Note that the universal $\w$-in-cover of $x$ has always depth (longest directed path to $x$) at most $L$.

\begin{figure}[h]
	\centering
\includegraphics[width=0.75\linewidth]{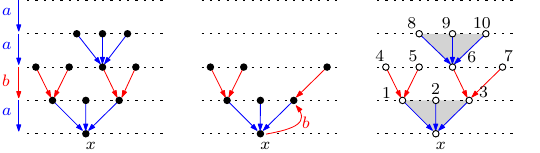}
	\caption{The $\w$-in-ball and universal $\w$-in-cover of a vertex $x$ for $\w=aaba$. On the left figure, the ball is a directed tree so the cover is isomorphic to the ball. On the middle figure, the ball is not a tree. The in-cover of the middle example is displayed on the right figure, with vertices arbitrarily labeled, and assuming that vertices $8,9,10$ have no $\mathbf{a}$-preimages.
	The following pairs of vertices in the right figure correspond to the same vertex in the $\w$-in-ball: $(6,x)$, $(1,8)$, $(2,9)$ and $(3,10)$.
Note that the right example coincide, as a tree, with the left one, in particular the left and middle examples have the same $\w$-in-cover.
	}
	
	\label{fig:w-in-ball}
\end{figure}

For all $L \geq 1$, define the \emph{$L$-in-ball of $x$}, $B_{L}(x)$, to be the union of the balls $B_{\w}(x)$ for all $\w\in \{a,b\}^L$, which is naturally rooted at $x$.
Note that $B_{\w}(x)\subset B_{L}(x)$, for any $\w\in \{a,b\}^L$. Next auxiliary result shows that the size of any $L$-in-ball has exponential tails.

\begin{lemma}\label{lemma:exponential-in-ball}
	There exists $c=c(L)>0$ such that for any $x\in [n]$ and $t\geq 1$, 
	\begin{align}\label{eq:exponential-in-ball}
		\bbP\bigl(|B_{L}(x)|> t\bigr) \leq e^{-ct}.
	\end{align}
\end{lemma}
\begin{proof}
	 Define the set $U_i\subset [n] $ recursively for $i\geq 0$ by $U_0\defi \{x\}$ and $U_{i+1}\defi \mathbf{a}^{-1}(U_i)\cup \mathbf{b}^{-1}(U_i)$. Let also $V_i\defi \cup_{0\leq j\leq i}U_j$. Note that $V(B_{L}(x))= V_L$. 

	 If $X_{i+1}\defi |U_{i+1}\setminus V_{i}|$, we claim that $(X_i)_{i\geq 0}$ is stochastically dominated by a branching process $\tilde{\mathbf{Z}}^{(n)}$ with offspring distribution $\xi^{(n)}\defi\mathrm{Bin}(2n,\frac{1}{n})$. 
	Indeed, explore the in-ball of $x$ by revealing $U_1,U_2, \dots, $ successively. To reveal $U_{i+1}\setminus V_i$, we only need to reveal $(\mathbf{a}^{-1}(y),\mathbf{b}^{-1}(y))$ for all $y \in U_i\setminus V_{i-1}$ (since the preimages of $y \in V_{i-1}$ have already been revealed in previous steps, so they are in $V_i$). Now reveal the preimages $(\mathbf{a}^{-1}(y),\mathbf{b}^{-1}(y))$ of each $y \in U_{i}\setminus V_{i-1}$ in some order. The sizes of these preimages are distributed as two independent binomial variables $\mathrm{Bin}(n-A,\frac{1}{n})$ and $\mathrm{Bin}(n-B,\frac{1}{n})$, where $A$ and $B$ are the (random) number of vertices whose $\mathbf{a}$-image and $\mathbf{b}$-image has already been revealed in previous steps of the process, respectively. Therefore the total number of new vertices revealed is dominated by $\mathrm{Bin}(2n-A-B,\frac{1}{n})$, hence by $\mathrm{Bin}(2n,\frac{1}{n})$. This proves the claim. 

	Now, let $f^{(n)}(x)\defi \bbE[x^{\xi^{(n)}}]=(1+\frac{x-1}{n})^{2n}$. We have, with the branching process $\tilde{\mathbf{Z}}^{(n)}$ as above,
	\begin{align}\label{eq:branching-analytic}
	\EBGW\Bigl[x_1^{\tilde{Z}_1^{(n)}}\cdots x_{L}^{\tilde{Z}_{L}^{(n)}}\Bigr] = f^{(n)}(x_1f^{(n)}(x_2\dots f^{(n)}(x_{L})\dots)).
\end{align}
	Now the sequence of functions $(f^{(n)})_{n\geq 1}$ is uniformly continuous at $x=1$ with $f^{(n)}(1)=1$, i.e. for any $\epsilon$ there exists $\delta=\delta(\epsilon)$ such that $f^{(n)}([1\pm \delta])\subset [1\pm \epsilon]$ for all $n\geq 1$.
	Iterating, there exists $\delta'=\delta'(L,\epsilon) >0$ such that~\eqref{eq:branching-analytic} is bounded by $1+\epsilon$ for $(x_1,\dots,x_L) \in [1\pm \delta']^L$.
	It follows that
	\begin{align}
		\EBGW\bigl[(1+\delta')^{\tilde{Z}_1^{(n)}+\dots+\tilde{Z}_L^{(n)}}\bigr] \leq C,
	\end{align}
	for some constant $C$, uniformly in $n$.
	By Markov inequality, we obtain
	\begin{align}
		\PBGW[\tilde{Z}_1^{(n)}+\dots+\tilde{Z}_L^{(n)}\geq  t ] \leq C(1+\delta')^{-t}.
	\end{align}
Recall that $\tilde{Z}_1^{(n)}+\dots+\tilde{Z}_L^{(n)}$ stochastically dominates $X_1+\dots+X_L=V_L-1$, so we are done.
\end{proof}

\subsection{Expected value of $Q_n^w(u)$}
\label{subsec:probaleaf}

For a letter $c\in \{a,b\}$, of respective capital version $C\in \{A,B\}$, we define the sets
\begin{align}
C^+&=C^+(\w)\defi \{1\leq \ell \leq L: \,\w_{L-\ell+1}=c\},\quad\\
C^-&=C^-(\w)\defi  \{0\leq \ell <L: \,\w_{L-\ell}=c\}.\quad\;\;\;\,
\end{align}

If $T\in \mathfrak{L}$ has root $r$, we
let $d_T(\cdot,r)$ be the \emph{distance-to-root function}. 
We define $T_\ell$ as the set of vertices $v$ of $T$ with $d_T(v,r)=\ell$.
The \emph{height} of $T$ is defined as $\max_{v\in V(T)} d_T(v,r)$.
Let $\mathfrak{L}_{<L}$ be the set of trees in $\mathfrak{L}$ of height less than $L$. For any letter $c\in\{a,b\}$ and $*\in\{+,-\}$, we define
\begin{align}\label{eq:def_t_c^*}
t_c^*&=t_c^*(\w)\defi  \sum_{\ell\in C^*} |T_\ell|.
\end{align}
Note that $t_a^++t_b^+=|T|-1$ and $t_a^-+t_b^-=|T|$. These parameters are independent of the labelling of $T$, and so are also well defined for a plane tree $\tau\in\mathfrak{P}$. Note that, although we do not carry it in the notation, $t_c^*$ depends on the underlying tree $T$ or $\tau$, which will always be clear from context. 

We can extend the definition to any labelled set-rooted $\{a,b\}$-digraph $H$, with root-set $\mathcal{R}\subset V(H)$. Most of the time our graphs will be rooted at a single vertex $r$ (i.e. $\mathcal{R}=\{r\}$ is a singleton), but we will also need to consider the more general case where a subset of vertices $\mathcal{R}$ is the root. We define similarly as above 
\begin{align}\label{eq:def_h_c^*}
h_c^+&=h_c^+(\w)\defi |\{v\in V(H): \exists \ell\geq 1,\, \exists r\in \mathcal{R},\,\,  v \xrightarrow[{H}]{{{\w}^{(\ell)}}} r,\, \w_{L-\ell+1}=c\}|,\\
h_c^-&=h_c^-(\w)\defi |\{v\in V(H): \exists \ell\geq 0,\, \exists r\in \mathcal{R},\,\,  v \xrightarrow[{H}]{{\w^{(\ell)}}} r,\, \w_{L-\ell}=c\}|,
\end{align}
where we used $v \xrightarrow[{H}]{{\w^{(\ell)}}} r$ to denote the event that there exists a path from $v$ to $r$ in $H$ the concatenation of whose arc-labels is $\w^{(\ell)}$. Observe that this coincides with the  definition of $t^+_c(\w)$ and $t^-_c(\w)$, when $H$ is a tree.

\begin{example}
	For the left tree in Figure~\ref{fig:w-in-ball}, we have $t_a^+=6$, $t_b^+=4$ (respectively the number of edges on an $a$-layer and on a $b$-layer) and $t_a^-=8$, $t_b^-=3$ (respectively the number of vertices whose $a$-input and $b$-input is determined, or in other words, vertices sitting just below an $a$-layer and a $b$-layer, respectively).
\end{example}

The following statement says that the $\w$-in-ball of a random vertex is well approximated by a Po(1)-BGWT, similarly as what happens for a single random function. Note however that we control the second-order term, in which the dependency on $\w$ is visible through the quantities $t_a^+, t_b^+, t_a^-, t_b^-$. This will play a crucial role in upcoming sections.
\begin{lemma}[\bf BGWT-approximation of $\w$-in-balls]\label{lem:prob_ball_given_tree}
Let $R$ be chosen uniformly at random in $[n]$ and $T\in \mathfrak{L}$. We have
\begin{align}
\bbP(B_{\w}(R)\conglab T) = 
\PBGW\bigl(\tau_\mathbf{Z}^{(lab)}= T\bigr) \left(1-\frac{\capprox^T(\w)}{n}+O\Big(\frac{|T|^3}{n^2}\Big)\right),
\end{align}
where 
\begin{align}\label{eq:c_approx_1_T_b}
\capprox^T(\w)&\defi  \frac{1}{2}\bigl[|T|(|T|-1)+t_a^-(\w)(t_a^-(\w)-2t_a^+(\w)) + t_b^-(\w)(t_b^-(\w)-2t_b^+(\w))\bigr].
\end{align}
\end{lemma}
\begin{proof}

	If the $\w$-in-ball has size $t$, there are $(n)_{t}/t!$ ways to relabel the vertices of $T$ with elements of $[n]$
	which are compatible with the relation $\conglab$.
	Given such a relabelling $T^{*}$ with root $r$, the event $\{B_{\w}(r)= T^*\}$ is equivalent to the intersection of the following three events: (i) $R$ must be $r$; (ii) each of the $t-1$ edges of $T^*$ has to be present; and (iii) no other edge has to be present in the $\w$-in-ball: this means that no vertex in $[n]$ can be mapped by the function $\mathbf{a}$ to one of the $t_a^-$ vertices of tree sitting below an $a$-layer, except for the vertices  from where the $t_a^+$ edges present in the tree origin (and similarly for the letter $b$). This happens with probability 
\begin{align}\label{eq:exact_prob_tree2}
	\bbP(B_{\w}(R)= T^*)=\frac{1}{n}\cdot\frac{1}{n^{t-1}}\left(1-\frac{t_{a}^-}{n}\right)^{n-t_a^+}\left(1-\frac{t_{b}^-}{n}\right)^{n-t_b^+},
\end{align}
and we obtain
\begin{equation}\label{eq:exact_prob_tree}
	\bbP(B_{\w}(R)\conglab T) =  \frac{(n)_{t}}{t!n^{t}}\left(1-\frac{t_{a}^-}{n}\right)^{n-t_a^+}\left(1-\frac{t_{b}^-}{n}\right)^{n-t_b^+}.
\end{equation}
	Now the lemma follows from~\eqref{eq:classical} and  a direct asymptotic expansions (write each product as an exponential of sums of logarithms, and expand each logarithm to second order).
\end{proof}

\begin{remark}
The constant $c_{approx}^T(\w)$ does not depend on the labelling of $T$, only on the underlying rooted tree structure. Thus, for all $\tau\in\mathfrak{P}$ we will write 
$c_{approx}^{\tau}(\w)$ to denote $c_{approx}^{\tau^{(lab)}}(\w)$.
\end{remark}

We recall the definition of a key notion. 
\begin{definition}[\bf Quasi-leaf]
Let $w\in\{a,b\}^k$. A vertex $x\in [n]$ is a \emph{$w$-leaf} (or just \emph{leaf} if $w$ is clear from the context) if $\bfw^{-1}(x)=\emptyset$, or equivalently if there are no vertices at distance $k$ from $x$ in $B_{w}(x)$. A vertex $x\in [n]$ is a \emph{$w$-quasi-leaf} (or just \emph{quasi-leaf} if $w$ is clear from the context) if $\bfw^{-1}(x)$ is a subset of leaves, or equivalently if there are no vertices at distance $2k$ from $x$ in $B_{w^2}(x)$. 
We denote by $\cQ_n=\cQ_n^w$ the (random) set of quasi-leaves.
\end{definition}

\begin{definition}[\bf Acyclic and cyclic vertex]
A vertex $x\in [n]$ is \emph{$w$-acyclic} (or just \emph{acyclic} if $w$ is clear from the context) if $B_{w^2}(x)$ is a rooted directed tree, and it is \emph{$w$-cyclic} (or just \emph{cyclic}) otherwise. We denote by $\cT_n=\cT_n^w$ the (random) set of acyclic vertices and by $\cC_n=\cC_n^w$ the (random) set of cyclic ones.

\end{definition}

\begin{lemma}[\bf Acyclic quasi-leaves]\label{lem:leaf_and_tree}
For all $x\in [n]$ and $u\in [0,1]$, we have
\begin{equation}
\bbE\bigl[u^{d(x)}\mathbf{1}_{x\in \cQ_n\cap \cT_n}\bigr] = g_k(u\eta_k) - \frac{\capprox^1(w,u)}{n}+O\Big(\frac{1}{n^2}\Big),
\end{equation}
where  $d(x)$ is the number of $\bfw$-preimages of $x\in [n]$ and $\capprox^1(w,u)\defi  \EBGW[u^{Z_k}\capprox^{\tau_{\mathbf{Z}}}(w^2)\mathbf{1}_{Z_{2k}=0}]$.
\end{lemma}
See also Section~\ref{subsec:def-cw-selfcontained}  for a self-contained expression of~$\capprox^1(w,u)$.
\begin{proof}
By symmetry, the events $\{x\in \cQ_n\cap \cT_n\}$ and $\{R\in \cQ_n\cap \cT_n\}$, where $R$ is uniformly chosen in $[n]$, have the same probability, and the latter can be written as $\cup_{T\in \mathfrak{L}_{<2k}} \{B_{\omega}(R)\conglab T\}$.
From Lemma~\ref{lem:prob_ball_given_tree},
\begin{equation}
\begin{aligned}
\bbE\bigl[u^{d(R)}\mathbf{1}_{R\in \cQ_n\cap \cT_n}\bigr]
&= \sum_{T\in \mathfrak{L}_{<2k}} \bbE\Bigl[u^{d(R)}\mathbf{1}_{B_{\omega}(R)\conglab T}\Bigr] \\
&= \sum_{T\in \mathfrak{L}_{<2k}} u^{|T_k|}\bbP(B_{\omega}(R)\conglab T) \\
&=\sum_{T\in \mathfrak{L}_{<2k}} u^{|T_{k}|}\PBGW\bigl(\tau_\mathbf{Z}^{(lab)}=T\bigr)\left(1-\frac{\capprox^T(w^2)}{n}+O\Big(\frac{|T|^3}{n^2}\Big)\right)\\ 
	&= \EBGW(u^{Z_{k}}\mathbf{1}_{Z_{2k}=0}) - \frac{\EBGW\bigl[u^{Z_k}\capprox^{\tau_\mathbf{Z}}(w^2)\mathbf{1}_{Z_{2k}=0}\bigr]}{n}+ \frac{O(1)}{n^2}\EBGW[Z^3u^Z\mathbf{1}_{Z_{2k}=0}]\\
&= g_k(u\eta_k) -\frac{\capprox^1(w,u)}{n}+ O\Big(\frac{1}{n^2}\Big).
\end{aligned}
\end{equation}
	The last equality follows from the fact that $\EBGW[Z^3u^Z\mathbf{1}_{Z_{2k}=0}] <\infty$, which is direct from the explicit expression~\eqref{eq:WPEOW} and the fact that $u\leq 1$.\qedhere

\end{proof}

The following lemma gives a rough bound for the probability of a cyclic vertex, this will be sufficient for our purposes.
\begin{lemma}[\bf Cyclic vertex]\label{lem:not_a_tree}
For all $x\in [n]$, we have
\begin{equation}
\bbP\bigl(x\in \cC_n\bigr) = O\Bigl(\frac{1}{n}\Bigr). 
\end{equation}
\end{lemma}

\begin{proof}
Let $\w=w^2$. Given that the ball $|B_{\w}(x)|$ has size $t$, the probability that $B_{\w}(x)$ is not a tree is at most $t^2/n$: indeed, in the BFS exploration of the $w^2$-in-ball, at each of the $t$ steps, the probability that the new in-edge has its origin at a vertex already present in the in-ball is at most $t/n$.

Let $c=c(2k)$ be the constant from Lemma~\ref{lemma:exponential-in-ball}. It follows that
\begin{align}
\bbP(x\in \cC_n)
&= \sum_{t\geq 1} \bbP(|B_{\w}(x)|= t) 
\bbP\bigl(x\in \cC_n\mid |B_{\w}(x)|= t\bigr) 
\leq \frac{1}{n} \sum_{t\geq 1} t^2 e^{-c (t-1)} 
= O\Bigl(\frac{1}{n}\Bigr).
\qedhere
\end{align}

\end{proof}

The previous lemmas imply the following approximation of the expected number of weighted quasi-leaves:
\begin{align}\label{eq:expectation_full}
\bbE[Q_n^w(u)]]=  g_k(u\eta_k) n -\capprox^1(w,u)+\bbE\bigl[u^{d(1)}\mathbf{1}_{1\in \cQ_n\cap \cC_n}\bigr]n+O\Bigl(\frac{1}{n}\Bigr).
\end{align}

In particular, we obtain the convergence of the expectation in Theorem~\ref{thm:main_leaves}.
\begin{corollary}\label{cor:prob_leaf}
For every $u\in [0,1]$ we have 
\begin{align}
{\bbE[Q_n^w(u)]}= g_k(u\eta_k) n + O(1).
\end{align}
\end{corollary}

Note that we have not estimated very precisely the contribution from quasi-leaves that are cycles, namely the limiting value of $\bbE\bigl[u^{d(1)}\mathbf{1}_{1\in \cQ_n\cap \cC_n}\bigr]n$, which could certainly be computed. 
In fact, it was computed in the special case $u=0$ in the first version of this article~\cite{ourpaper_v1}. However, as will become clear later, this term plays no role in the computation of the variance of $Q_n^w(u)$. To avoid unnecessary technicalities, we therefore omit its evaluation here.

\section{Variance of $Q_n^w(u)$}
\label{sec:second_moments}

In this section we give an asymptotic expression for the variance of  $Q_n^w(u)$, concluding the proofs of Theorems~\ref{thm:main_leaves} and~\ref{thm:main_variance}.

\subsection{Pairs of quasi-leaves}

Recall that $d(x)$ is the number of $\bfw$-preimages of $x\in [n]$.
For all $x,y\in [n]$, $x\neq y$ and $\omega:=w^2$, we may write following equality:
\begin{align}
\bbE\bigl[u^{d(x)+d(y)} \mathbf{1}_{x,y\in\cQ_n}\bigr] &= 
\bbE\bigl[u^{d(x)+d(y)} \mathbf{1}_{x,y\in\cQ_n\cap \cT_n}\mathbf{1}_{B_{\omega}(x)\cap B_{\omega}(y)= \varnothing}\bigr]\label{eq:first_term}\\ 
&\hspace{1cm}+ 
\bbE\bigl[u^{d(x)+d(y)} \mathbf{1}_{x,y\in\cQ_n\cap \cT_n}\mathbf{1}_{B_{\omega}(x)\cap B_{\omega}(y)\neq \varnothing}\bigr]\label{eq:second_term} \\
&\hspace{1cm}+ 2\bbE\bigl[u^{d(x)+d(y)} \mathbf{1}_{x\in\cQ_n\cap \cC_n}\mathbf{1}_{y\in \cQ_n}\bigr]\label{eq:third_term} \\
&\hspace{1cm}- \bbE\bigl[u^{d(x)+d(y)} \mathbf{1}_{x,y\in\cQ_n\cap \cC_n}\bigr]\label{eq:fourth_term}. 
\end{align}
As we will see, the first term~\eqref{eq:first_term} will give the only constant contribution. It will also give an explicit error term stemming from the BGWT-approximation of order $1/n$. The second and third terms,~\eqref{eq:second_term} and~\eqref{eq:third_term}, each covering an unlikely event (either intersecting in-balls or the presence of a cycle) will also give corrections of order $1/n$. We will compute~\eqref{eq:second_term} explicitly, while the contribution of~\eqref{eq:third_term} will vanish in computing the variance. The last term,~\eqref{eq:fourth_term}, which requires two exceptional events to happen, will be of order $1/n^2$ and thus negligible.

Let us first study~\eqref{eq:first_term}. We approximate the mutually exclusive $w$-in-balls of $x,y$ by Po(1)-BGWTs. As in Lemma~\ref{lem:leaf_and_tree}, we obtain a precise expression for the second-order term in this approximation.
 \begin{lemma}[\bf Acyclic non-intersecting quasi-leaves]\label{lemma:noninttreeleaves}
For all $x,y\in[n]$ with $x\neq y$ and uniformly over $u\in [0,1]$, we have
\begin{align}\label{eq:OEMVD}
 \bbE\bigl[u^{d(x)+d(y)} \mathbf{1}_{x,y\in\cQ_n\cap \cT_n}\mathbf{1}_{B_{\omega}(x)\cap B_{\omega}(y)= \varnothing}\bigr] = g_k(u\eta_k)^2 - \frac{\capprox^2(w,u)}{n}+ O\Big(\frac{1}{n^2}\Big),
\end{align}
where
\begin{align}
\capprox^2(w,u)\defi  \frac{1}{2} \EBGW^2\Bigl[u^{Z_{k}^{(1)}+Z_{k}^{(2)}}\bigl(({Z}^{(1)}+{Z}^{(2)}-2)({Z}^{(1)}+{Z}^{(2)}+1)+\overline{Z}^a+\overline{Z}^b\bigr)\mathbf{1}_{Z_{2k}^{(1)}=Z_{2k}^{(2)}=0}\Bigr].
\end{align}
	 Here we recall the definition of the sets $A^+, A^-, B^+,B^-$ from Section~\ref{subsec:probaleaf}, and for a letter $c\in \{a,b\}$ of respective capital version $C\in \{A,B\}$, we let $\overline{Z}^c = \overline{Z}_{C^-}(\overline{Z}_{C^-}-2\overline{Z}_{C^+})$ with $\overline{Z}_U=Z^{(1)}_U+Z^{(2)}_U$ for every $U\subseteq \mathbb{N}$.
\end{lemma}

See also Section~\ref{subsec:def-cw-selfcontained}  for a self-contained expression of~$\capprox^2(w,u)$.
\begin{proof}
For all $T,T'\in \mathfrak{L}$, $t=|T|$ and $t'=|T'|$, we have, arguing as in the proof of Lemma~\ref{lem:prob_ball_given_tree},  with $(R,R')$ a uniformly chosen ordered pair of distinct elements in $[n]$,
\begin{equation}
\begin{aligned}
\bbP(B_{\w}(R)\conglab T,&\, B_{\w}(R')\conglab T',B_{\w}(R)\cap B_{\w}(R')= \varnothing)\\
&=\frac{(n)_{t+t'}}{t!t'!n(n-1)n^{t+t'-2}}  \left(1-\frac{m_a^-}{n}\right)^{n-m_a^+}\left(1-\frac{m_b^-}{n}\right)^{n- m_a^+}\\
	&= \PBGW^{2}(\tau_{\mathbf{Z}^{(1)}}^{(lab)}= T, \tau_{\mathbf{Z}^{(2)}}^{(lab)}= T') \left(1-\frac{\capprox^{T,T'}(w)}{n}+O\Big(\frac{{|T+T'|^3}}{n^2}\Big)\right),
\end{aligned}
\end{equation}
where
\begin{align}
\capprox^{T,T'}(w)\defi  \frac{1}{2}\bigl((|T|+|T'|-2)(|T|+|T'|+1)+m_a^-(m_a^--2m_a^+) + m_b^-(m_b^--2m_b^+)\bigr),
\end{align}
and  $m_c^* \defi  t_c^* +(t')_c^*$ for $c\in \{a,b\}$ and $*\in \{+,-\}$. The end of the proof is analogous to the one of Lemma~\ref{lem:leaf_and_tree}.
\end{proof}

\subsection{Size-biased Bienaymé-Galton-Watson Trees and spine decomposition}
\label{subsec:sizebiased}
Define the \emph{size-biased Po(1)-BGWT with respect to generation $i\geq 0$},   $\mathbf{Z}^{\langle i\rangle}$, through its probability function $\PBGW^{\langle i\rangle}$ as follows: for any event $A$,
\begin{align}
\PBGW^{\langle i\rangle}(\mathbf{Z}^{\langle i\rangle}\in A) \defi  \EBGW[Z_i\mathbf{1}_{\mathbf{Z}\in A}].
\end{align}
Observe that it is a probability distribution, since $\EBGW[Z_i]=1$. In particular, by \eqref{eq:SOEDS-conditional}, for all $0\leq i\leq L$ 
\begin{align}\label{eq:SPEOD}
\PBGW^{\langle i\rangle}(Z^{\langle i\rangle}_L=0)= \EBGW\bigl[Z_i \mathbf{1}_{Z_L=0}\bigr]= \EBGW[Z_i\eta_{L-i}^{Z_i}].
\end{align}
It is well known that a size-biased Po(1)-BGWT admits a particularly transparent description through the \emph{spine decomposition}, which we now recall (we will need to adapt it to more complicated cases later). It relies on the following direct observation (which can be proved from the fact that the probability generating function of a Po(1) satisfies $\frac{sd}{ds}f(s)=sf(s)$):
\begin{classicalFact}\label{classicalFact:classical}
	Let $X$ and $Y$ be respectively a \textnormal{Po(1)} random variable and a size-biased \textnormal{Po(1)} random variable\footnote{i.e. $\bbE[f(Y)]=\bbE[Xf(X)]$ for any measurable function $f$.}. Then $Y$ has the same law as $1+X$.
\end{classicalFact}

Now the spine decomposition works as follows: Given a size-biased Po(1)-BGWT, mark a uniform vertex at generation $i$. Consider the ancestral line from the root to this vertex: at every step of this path, one child of the current vertex is chosen as the next descendent, hence the random variable, say $Y$, giving the offspring of that vertex is a \emph{size-biased} Po(1)-variable. The number of children not in the branch at each step is thus distributed as a $Y-1$, which by the previous fact is an unbiased Po(1). 
Therefore a size-biased Po(1)-BGWT can be produced by the following procedure: first sample a line of height $i$ originating from the root, and then on each vertex of this line attach independent (unbiased) Po(1)-BGWTs. See Figure~\ref{fig:branch-simple}. Note that on this representation and the remaining ones, we disregard the plane structure of the BGWT: to recover it fully, one should reshuffle the order of children of vertices along the spine (so that the index of the spine child among all children of a spine vertex is uniform).

From this equivalence, we have two ways to express the probability that the size-biased Po(1)-BGWT is extinct by generation~$L$ in~\eqref{eq:SPEOD}:
\begin{align}
\PBGW^{\langle i\rangle}(Z^{\langle i\rangle}_L=0)&= 
	\Xi_L^{\mathbf{x},1}\left[ \frac{x_i\partial}{\partial x_i}G_{L}(x_0,x_1,\dots,x_{L})\right]\\
	&= \prod_{t=0}^{i} \eta_{L-t}. \label{eq:branch-simple}
\end{align}
The first expression is~\eqref{eq:WPEOW} with $m_i=1$ and the other ones equal to zero, and the second follows either by calculations or by the spine decomposition: conditioned to a marked vertex at generation $i$, the process is extinct at time $L$ if and only if each independent Po(1)-BGWT hanging from the ancestral line of that vertex, starting at height $0,1,\dots,i$ has height at most $L,L-1,\dots,L-i$ respectively; see Figure~\ref{fig:branch-simple}. 

We will often consider the random pair $(\tau_{\mathbf{Z}^{\langle i\rangle}},V)$, where $V$ is a vertex chosen uniformly at random among vertices at height $i$; this is well defined as by definition $Z^{\langle i\rangle}_i>0$. For all $\tau\in \mathfrak{P}$ of height at least $i$, and $v$ a vertex of $\tau$ at height $i$, we have
\begin{align}\label{eq:un-biasing}
\muBGW^{{\langle i\rangle}}(\tau;v)
\defi \PBGW^{{\langle i\rangle}}(\tau_{\mathbf{Z}^{\langle i\rangle}}=\tau,V=v)
= \EBGW[Z_i\mathbf{1}_{\tau_{\mathbf{Z}}=\tau}\mathbf{1}_{V=v}] = \muBGW(\tau).
\end{align}

\begin{figure}
	\centering
	\includegraphics[width=0.3\linewidth]{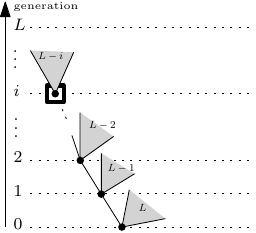}
	\caption{Direct interpretation of the formula in~\eqref{eq:branch-simple} for $\PBGW^{\langle i\rangle}(Z^{\langle i\rangle}_L=0)$. The marked vertex at generation $i$ (among $Z^{\langle i\rangle}_i$) is squared. The number in each subtrees indicate a strict upper bound on the height of its offspring under the event $Z^{\langle i\rangle}_L=0$, each contributing a factor $\eta_{l-t}$ in~\eqref{eq:branch-simple}. }
	\label{fig:branch-simple}
\end{figure}

\subsection{Colliding branching processes}\label{sec:coll-BP}

In order to understand the term~\eqref{eq:second_term}, corresponding to pairs of quasi-leaves whose corresponding $w^2$-in-balls intersect, we now introduce a BGWT-based stochastic process, that can be seen as a coupling of two Po(1)-BGWTs that guarantees that one element at generation $i$ of the first process and one element at generation $j$ of the second one have identical offspring for $\ell$ generations. These two common subprocesses will correspond to the intersection of two in-balls in random functions.

Given $\tau\in\mathfrak{P}$ and a vertex $v$, we denote by $\tau_v$ the subtree of $\tau$ rooted at $v$ and by $\tau_v(\ell)$ the subtree induced by the vertices at distance at most $\ell$ from $v$ in $\tau_v$; we use the analogous notation for all $T\in \mathfrak{L}$.

\begin{definition}[Free vertices]\label{def:free}
	For all $i,j,\ell\geq 0$, we let $\mathfrak{P}^{i,j,\ell}$ be the set of tuples $(\tau,\tau';v,v')$ where $\tau, \tau'\in\mathfrak{P}$, $v$ is a vertex of $\tau$ at height $i$, $v'$ is a vertex of $\tau'$ at height $j$, such that $\tau_v(\ell)\congplane \tau'_{v'}(\ell)$. Then we let $\mathcal{F}_\ell(\tau,\tau';v,v')\defi V(\tau)\cup V(\tau')\setminus V(\tau'_{v'}(\ell-1))$ be the set of \emph{free vertices}, and $f_\ell(\tau,\tau';v,v')\defi |\mathcal{F}_\ell(\tau,\tau';v,v')|$ be their number.
\end{definition}

We now construct a random variable with values in $\mathfrak{P}^{i,j,\ell}$ through a random process. 
Let $0 \le i < j \le L-1$ and $0\leq \ell\leq L-j$. Consider the joint  process $(\mathbf{Y},\mathbf{Y}')$ constructed as follows. 
Let the first $i$ generations of $\mathbf{Y}$ be distributed as the first $i$ generations of $\mathbf{Z}^{\langle i \rangle}$, and, independently, the first $j$ generations of $\mathbf{Y}'$ be distributed as the first $j$ generations of $\mathbf{Z}^{\langle j \rangle}$. Choose independently and uniformly at random  $V$ at generation $i$ of $\mathbf{Y}$ and $V'$ at generation $j$ of $\mathbf{Y}'$. Reveal the rest of the process in such a way that the descendants of these two marked vertices evolve jointly for $\ell$ generations. Precisely, (i) attach independent Po(1)-BGWTs to all individuals at generation $i$ of $\mathbf{Y}$, and to all individuals but $V'$ at generation $j$ of $\mathbf{Y}'$, (ii) reveal the descendants of $V'$ for $\ell$ generations so they coincide with the descendants of $V$ for $\ell$ generations (this step is deterministic), and (iii) attach independent Po(1)-BGWTs to all descendants of $V'$ whose offspring has not been revealed yet. Let $\bbP^{i,j,\ell}_{\mathbf{Y},\mathbf{Y}'}$ and $\bbE^{i,j,\ell}_{\mathbf{Y},\mathbf{Y}'}$ respectively denote the probability distribution and the expectation of the process.

Observe that the marginals are distributed as $\mathbf{Z}^{\langle i \rangle}$ and $\mathbf{Z}^{\langle j \rangle}$ respectively, i.e. $(\mathbf{Y},\mathbf{Y}')$ is a coupling of $\mathbf{Z}^{\langle i \rangle}$ and $\mathbf{Z}^{\langle j \rangle}$.

Let $\tau_\mathbf{Y}$ and $\tau_{\mathbf{Y'}}$ be the tree representations of $\mathbf{Y}$ and $\mathbf{Y}'$, and $\widehat{\tau}_\mathbf{Y}$ and $\widehat{\tau}_{\mathbf{Y'}}$ their truncations at height $i$ and $j$ respectively. Define
\begin{align}
\mu^{i,j,\ell}_{\mathbf{Y},\mathbf{Y}'}(\tau,\tau';v,v')\defi  \bbP^{i,j,\ell}_{\mathbf{Y},\mathbf{Y}'} (\tau_{\mathbf{Y}}=\tau,\tau_{\mathbf{Y}'}=\tau',V=v,V'=v'), \qquad \text{for all }(\tau,\tau';v,v')\in \mathfrak{P}^{i,j,\ell}.
\end{align}

\begin{lemma}\label{lemma:measure_tau2_u_v}
	Let $(\tau,\tau';v,v') \in \mathfrak{P}^{i,j,\ell}$. 
We have
\begin{equation}\label{eq:measure_tau2_u_v}
	\mu^{i,j,\ell}_{\mathbf{Y},\mathbf{Y}'}(\tau,\tau';v,v') = 
	\frac{1}{\prod_{x\in \mathcal{F}_\ell(\tau,\tau';v,v')} n_x!}
	e^{-f_\ell(\tau,\tau';v,v')},
\end{equation}
	where for all $x \in V(\tau)\cup V(\tau')$, $n_x$ is the number of children of the vertex $x$ in the corresponding tree.
\end{lemma}
\begin{proof}

Let $\widehat{\tau}$ and  $\widehat{\tau}'$ be the truncations of $\tau$ and $\tau'$ at heights $i$ and $j$, respectively. Let $V_{<i}$ $V'_{<j}$ be the sets of vertices of $\tau$ and $\tau'$ at height less than $i$ and $j$, respectively.

The event $(\tau_\mathbf{Y},\tau_\mathbf{Y'};V,V')=(\tau,\tau';v,v')$ is equivalent to the fact that the two following properties hold: (i) $(\widehat{\tau}_\mathbf{Y},\widehat{\tau}_\mathbf{Y'};V,V')=(\widehat{\tau},\widehat{\tau}';v,v')$,
	and (ii) the sequence of number of children of the free vertices which are not already in $V_{<i} \cup V'_{<j}$ is equal to the same sequence in $(\tau,\tau';v,v')$ both listed in canonical depth-first-search, left-to-right order.
This follows from the fact that $\tau_\mathbf{Y}$ and $\tau_\mathbf{Y'}$ are fully determined by the offspring of the free vertices: indeed, during the construction, the vertices of $(\tau_\mathbf{Y'})_{v'}(\ell-1)$ inherit their offspring by replicating $(\tau_\mathbf{Y})_{v}(\ell-1)$ in  a deterministic way. Free vertices in $V_{<i} \cup V'_{<j}$  are taken into account by the first property, and the remaining ones by the second one.

By~\eqref{eq:un-biasing}, the probability of the event $(\widehat{\tau}_{\mathbf{Y}},\widehat{\tau}_{\mathbf{Y}'};V,V')=(\widehat{\tau},\widehat{\tau}';v,v')$ is the same as the probability of the event $(\widehat{\tau}_{\mathbf{Z}^{(1)}},\widehat{\tau}_{\mathbf{Z}^{(2)}})=(\widehat{\tau},\widehat{\tau}')$, where $\mathbf{Z}^{(1)}$ and $\mathbf{Z}^{(2)}$ are independent Po(1)-BGWTs, which holds with probability
\begin{align}
\prod_{x\in V_{<i}} \frac{e^{-1}}{n_x!}\prod_{x\in V'_{<j}} \frac{e^{-1}}{n_x!},
\end{align}
Moreover, the offspring of the remaining free vertices (elements of $\mathcal{F}_{\ell}(\tau,\tau';v,v')\setminus (V_{<i}\cup V'_{<j})$) is distributed as a sequence of independent Po(1) variables.

Taking the two properties into account, the desired event holds with probability
\begin{align}
\prod_{x\in V_{<i}\cup V'_{<j}} \frac{e^{-1}}{n_x!}
\prod_{x\in \mathcal{F}_\ell(\tau,\tau';v,v')\setminus (V_{<i}\cup V'_{<j})} \frac{e^{-1}}{n_x!}&=
\prod_{x\in \mathcal{F}_\ell(\tau,\tau';v,v')} \frac{e^{-1}}{n_x!}.\qedhere
\end{align}
\end{proof}

\begin{figure}
	\centering
	\includegraphics[width=\linewidth]{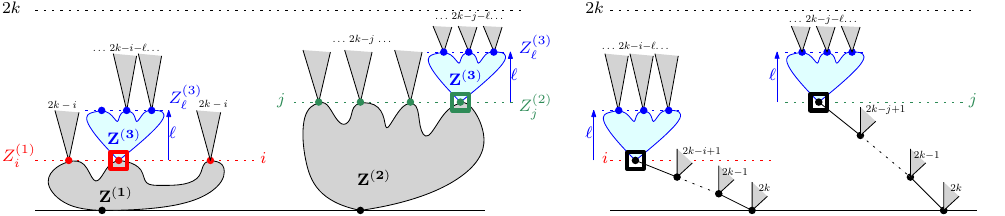}
	\caption{Left: Two branching processes evolving jointly according to the law $\mu^{i,j,\ell}_{\mathbf{Y},\mathbf{Y}'}$: they have a marked vertex, respectively at heights $i$ and $j$, and are conditioned to the fact that the subtrees above these marked vertices coincide for $\ell$ generations. Here the event $Y_{2k}=Y'_{2k}=0$ is considered, which contrains the maximal possible height of each other subtree as indicated. Right: spine decomposition of this event.
	In what follows we will need to weight these configurations by the number of vertices appearing at level $k$, see Figure~\ref{fig:cases} for the different cases involved.	}
	\label{fig:branch-2}
\end{figure}

The next proposition gives an explicit expression for the probability of the process being extinct by generation $2k$, weighted by vertices at generation $k$ (we could adress the case of extinction at any generation $L>k$ but only $L=2k$ will be useful).

\begin{proposition}\label{prop:cross_monster}  For every $u\in \mathbb{R}$, we have
\begin{align}\label{eq:GKEDS}
\bbE^{i,j,\ell}_{\mathbf{Y},\mathbf{Y}'}
\bigl[u^{Y_{k}+Y_{k}'}\mathbf{1}_{Y_{2k}=Y_{2k}'=0}\bigr]
	= F_1^{i,j} (u)F_2^{i,j,\ell} (u),
\end{align}
where 
\begin{align}
	F_1^{i,j} (u):&=	\bigl(\prod_{s=0}^{i-1}g_{k-s}(u^{\mathbf{1}_{\{s\leq k\}}}\eta_{k})\bigr) \bigl(\prod_{s=0}^{j-1}g_{k-s}(u^{\mathbf{1}_{\{s\leq k\}}}\eta_{k})\bigr)
	\\&=
	\bigl(\prod_{s=0}^{(i-1)\wedge k}g_{k-s}(u \eta_{k})\bigr) \bigl(\prod_{s=0}^{(j-1)\wedge k}g_{k-s}(u \eta_{k})\bigr) \bigl(\prod_{s=k+1}^{i-1} \eta_{2k-s}\bigr)  \bigl(\prod_{s=k+1}^{j-1} \eta_{2k-s}\bigr),
\end{align}
and 
\begin{align}
F_2^{i,j,\ell} (u):=\begin{cases}
g_{\ell}(\eta_{2k-i-\ell}\eta_{2k-j-\ell}) &\text{if $k<i$,}\\
g_{k-i}(u g_{i+\ell-k}(\eta_{2k-i-\ell}\eta_{2k-j-\ell})) &\text{if $i\leq k <j$ and $k\leq i+\ell$,}\\
g_{\ell}(g_{k-i-\ell}(u\eta_{k})\eta_{2k-j-\ell}) &\text{if $i\leq k <j$ and $i+\ell<k$,}\\
g_{k-j}(u g_{j-i}(u g_{i+\ell-k}(\eta_{2k-i-\ell}\eta_{2k-j-\ell}))) &\text{if $j\leq k$ and $k\leq i+\ell$,}\\
g_{k-j}(u g_{\ell+j-k} (g_{k-i-\ell}(u\eta_{k})\eta_{2k-j-\ell})) &\text{if $j\leq k$ and $i+\ell< k<j+\ell$,}\\
g_{\ell}(g_{k-i-\ell}(u\eta_{k})g_{k-j-\ell}(u\eta_{k})) &\text{if $j\leq k$ and $j+\ell\leq k$}.
\end{cases}
\end{align}
\end{proposition} 
\begin{proof}
We first prove the case $u=1$, to exemplify the argument. Fix the spines
to the marked individuals in each process, similarly as what we did in the simpler case of Figure~\ref{fig:branch-simple}. For each height $s$ of a spine individual, the subtrees growing from its siblings must become extinct by time $2k-s$, producing the factor $\bigl(\prod_{s=0}^{i-1}\eta_{2k-s}\bigr) \bigl(\prod_{s=0}^{j-1}\eta_{2k-s}\bigr)$. Finally, the common depth-$\ell$ subtree contributes the factor
	$\EBGW[\eta_{2k-i-\ell}^{Z_\ell}\eta_{2k-j-\ell}^{Z_\ell}]=g_\ell(\eta_{2k-i-\ell}\eta_{2k-j-\ell})$, as there are $Z_{\ell}$ descendants at height $\ell$ from $v$ in $\tau_{\mathbf{Y}}$ and from $v'$ in $\tau_{\mathbf{Y}'}$, whose offspring is independent and which have to be extinct by times $2k-i-\ell$ and $2k-j-\ell$, respectively. See Figure~\ref{fig:branch-2}--Right. It follows that
	\begin{align}\label{eq:DIDKS}
	\begin{aligned}
	\bbE^{i,j,\ell}_{\mathbf{Y},\mathbf{Y}'}[\mathbf{1}_{Y_{2k}=Y_{2k}'=0}]
	&= \bbP^{i,j,\ell}_{\mathbf{Y},\mathbf{Y}'}(Y_{2k}=Y_{2k}'=0)\\
	&= \bigl(\prod_{s=0}^{i-1}\eta_{2k-s}\bigr) \bigl(\prod_{s=0}^{j-1}\eta_{2k-s}\bigr) g_\ell (\eta_{2k-i-\ell}\eta_{2k-j-\ell}) \\
	&= F_1^{i,j} (1)F_2^{i,j,\ell} (1).
	\end{aligned}
	\end{align}
	
	We now tackle the general case. Any term of the form $\eta_{2k-s}$ in~\eqref{eq:DIDKS} corresponds to a Po(1)-BGWT rooted at absolute height $s$ and being extinguished by absolute height $2k$. If $s\leq k$ such a tree contains $Z_{k-s}$ vertices at absolute height $k$, that we should mark with the variable $u$, which can be done by replacing the factor $\eta_{2k-s}$ by $\EBGW[u^{Z_{k-s}} \mathbf{1}_{Z_{2k-s}=0}]=\EBGW[u^{Z_{k-s}} \eta_k^{Z_{k-s}}]=g_{k-s}(u \eta_{k})$.
	Therefore the contribution of the first two products in~\eqref{eq:DIDKS}, i.e. the contribution of the spine (not including the terminal vertices $v,v'$) becomes:
\begin{align}
F_1^{i,j}(u):=\bigl(\prod_{s=0}^{(i-1)\wedge k}g_{k-s}(u \eta_{k})\bigr) \bigl(\prod_{s=0}^{(j-1)\wedge k}g_{k-s}(u \eta_{k})\bigr) \bigl(\prod_{s=k+1}^{i-1} \eta_{2k-s}\bigr)  \bigl(\prod_{s=k+1}^{j-1} \eta_{2k-s}\bigr).
\end{align}

For the second part, let $T_{v}=(\tau_{\mathbf{Y}})_{v}(\ell)$ and $T_{v'}=(\tau_{\mathbf{Y}'})_{v'}(\ell)$. The factor $F_2^{i,j,\ell}(u)$, which accounts for the contribution of the common depth-$\ell$ subtrees $T_{v}$ and $T_{v'}$, is defined piece-wise depending on the relative order of $i,j,k,\ell$ (we refer to Figure~\ref{fig:cases} to clarify some of the expressions):
\begin{item}
\item[-] $k<i$ (Case A): we have $F_2^{i,j,\ell}(u):=g_{\ell}(\eta_{2k-i-\ell}\eta_{2k-j-\ell})$ as in the case $u=1$, as there is no need to mark any of the vertices in $T_{v}$ and $T_{v'}$ or above them.
\item[-] $i\leq k <j$: we split into
\begin{itemize}
	\item[-] $k\leq i+\ell$ (Case B): we have $F_2^{i,j,\ell}(u):=g_{k-i}(u g_{i+\ell-k}(\eta_{2k-i-\ell}\eta_{2k-j-\ell}))$.  Indeed, since the vertices at height $k$ are marked with the variable $u$ and $i\leq k$, in the expression obtained in Case A we need to take into account the variable $u$ in the contribution of the offspring of $v$. However, since $j>k$, we do not have to change anything for the offspring of $v'$. Precisely, we need to mark the vertices of $T_{v}$ at relative height $k-i$ (these are at absolute height $k$ in $\tau_{\mathbf{Y}}$).
		Each of these vertices have an independent offspring in $T_v$ which are copied in $T_{v'}$ for $i+\ell-k$ layers. Finally, any remaining offspring at the top of $T_v$ has a height bounded by $2k-i-\ell$, and similarly for the ones at the top of $T_{v'}$ that have height at most $2k-j-\ell$, and all these offsprings are mutually independent. This gives the contribution
		$$\EBGW^2\big[ \big(u  (\eta_{2k-i-\ell}\eta_{2k-j-\ell})^{Z^{(2)}_{i+\ell-k}} \big)^{Z^{(1)}_{k-i}} \big]=g_{k-i}(ug_{i+\ell-k}(\eta_{2k-i-\ell}\eta_{2k-j-\ell})).$$

For the remaining cases C-F the arguments are similar, and we give a bit less detail.
\item[-] $i+\ell< k$ (Case C): we have $F_2^{i,j,\ell}(u):=g_{\ell}(g_{k-i-\ell}(u\eta_{k})\eta_{2k-j-\ell})$, as we only need to mark the vertices of $\tau_{\mathbf{Y}}$ at absolute height $k$ (at relative height $k-i$ with respect to $u$).
\end{itemize}
\item[-] $j\leq k$: we split into
\begin{itemize}
\item[-] $k\leq i+\ell$ (Case D): we have $F_2^{i,j,\ell}(u):=g_{k-j}(u g_{j-i}(u g_{i+\ell-k}(\eta_{2k-i-\ell}\eta_{2k-j-\ell})))$, as we need to mark the vertices of $T_{v'}$ at height $k-j$ and the vertices of $T_{v}$ at height $k-i$ (note that $j-i=(k-i)-(k-j)$).
\item[-] $i+l< k<j+\ell$ (Case E): we have $F_2^{i,j,\ell}(u):=g_{k-j}(u g_{\ell+j-k} (g_{k-i-\ell}(u\eta_{k})\eta_{2k-j-\ell}))$, as we need to mark the vertices of $T_{v'}$ at heigh $k-j$ and the vertices of $\tau_{\mathbf{Y}}$ at absolute height $k$ (at relative height $k-i$ with respect to $u$).
\item[-] $j+\ell\leq k$ (Case F): we have $F_2^{i,j,\ell}(u):=g_{\ell}(g_{k-i-\ell}(u\eta_{k})g_{k-j-\ell}(u\eta_{k}))$, as we need to mark the vertices of $\tau_{\mathbf{Y}}$ and $\tau_{\mathbf{Y}'}$ at absolute height $k$ (relative heights $k-i$ and $k-j$ with respect to $v$ and $v'$, respectively).
\end{itemize}
\end{item}
	
	\begin{figure}
	\centering
	\includegraphics[width=\linewidth]{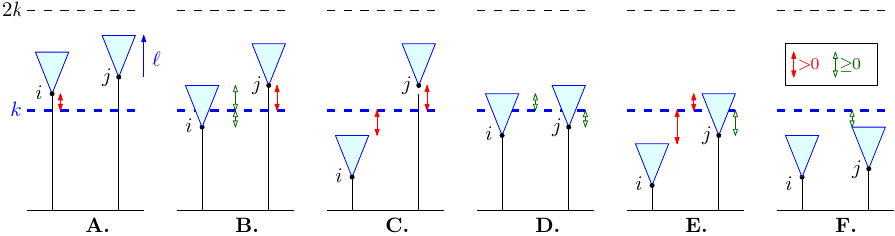}
	\caption{Cases of Proposition~\ref{prop:cross_monster}. In each subfigure, the left picture corresponds to $\tau_{\mathbf{Y}}$ and the right picture to $\tau_{\mathbf{Y}'}$; recall that $i<j$. The common depth-$\ell$ subtrees, $T_v$ and $T_{v'}$, are depicted. The absolute height $k$ is marked with a dashed line in each subfigure, indicating its relative order with respect to $i,i+\ell,j$ and $j+\ell$.
		}
	\label{fig:cases}
\end{figure}

\end{proof}

\subsection{Colliding quasi-leaves and computation of~\eqref{eq:second_term}}

Having introduced the process $(\mathbf{Y},\mathbf{Y}')$, we can now compute the contribution of ``colliding balls'' to the variance. This contribution depends on the set of ``maximal self-overlaps'' of the word $w^2$, which are captured by the set $\mathcal{M}$ and the quantities $(\ell_{i,j})_{(i,j)\in\mathcal{M}}$ in the following definition. We give this definition for a generic word $\omega$ of length $L\geq 1$, even if in our applications we will use $\omega = w^2$ so $L=2k$.

\begin{definition}[\bf Maximum self-overlap set]
	Given $\omega \in \{a,b\}^L$, define $\cM=\mathcal{M}(\omega)$ as the set of pairs $(i,j)$ with $0\leq i< j\leq L-1$, such that $i=0$  or $\omega_{L+1-i}\neq \omega_{L+1-j}$.
	For all $(i,j)\in \cM$, we let $\ell_{i,j}=\ell_{i,j}(\omega)$ be the largest integer $\ell \in \{0,1,\dots,L-j\}$ such that $\omega_{L+1-i-s}=\omega_{L+1-j-s}$ for all $s\in [\ell]$.
\end{definition}
In words, $\ell_{i,j}$ is the length of the maximal self-overlap of $\omega$ ending at positions $L-i$ and $L-j$ (thus an overlap to the left of these positions). The fact that we consider only $i$ and $j$ in $\cM$ ensures that this overlap cannot be extended into a longer overlap. 
Note that we require $i<j$: in this context, it makes no sense to have an overlap ending at the same position.

In what follows, depending on context and on the importance to insist on the word $\omega$, we will write $\mathcal{M}$ or $\mathcal{M}(\omega)$, and $\ell_{i,j}$ or $\ell_{i,j}(\omega)$. This should create no confusion.

\begin{example}\label{exm:set_M}
	For $\omega=ababaa$, we have  
	\begin{align}
		\cM(\omega)=\{(0,1),(0,2),(0,3),(0,4),(0,5),(1,3),(1,5),(2,3),(2,5),(3,4),(4,5)\}.
	\end{align}
	We have $\ell_{0,1}(\omega)=\ell_{0,3}(\omega)=\ell_{0,5}(\omega)=1$, $\ell_{1,3}(\omega)=3$, $\ell_{1,5}(\omega)=1$ corresponding respectively to the overlaps
$(ababa{\underline{a}},abab{\underline{a}}a)$, 
$(ababa{\underline{a}},ab{\underline{a}}baa)$, 
$(ababa{\underline{a}},{\underline{a}}babaa)$, 
$(ab{\underline{aba}}a,{\underline{aba}}baa)$, 
$(abab{\underline{a}}a,{\underline{a}}babaa)$, 
	while $\ell_{i,j}(\omega)=0$ for all other values $(i,j)\in \cM(\omega)$, since the letters at the corresponding positions differ.
	Note that the value $(3,5)$ (for example) is not part of $\cM(\omega)$:  we will never need to consider the overlap
	$(ab{\underline{a}}baa,{\underline{a}}babaa)$, which it is not maximal (it can be extended to the right to  a larger overlap by including the letters $ba$).
\end{example}

\begin{remark}\label{rem:M_identifies_w}
	The set $\mathcal{M}(\omega)$ and the $(\ell_{i,j}(\omega))_{(i,j)\in\mathcal{M}(\omega)}$ characterize the word $\omega$ up to isomorphism. 
	Indeed, for $1\leq j \leq L-1$ we always have $(0,j)\in \mathcal{M}$, and $\ell_{0,j}>0$ if and only if $\omega_{L-j}=\omega_L$.
	 For instance, in Example~\ref{exm:set_M}, we deduce that $\omega_6=\omega_5=\omega_3=\omega_1$ and $\omega_6\neq \omega_4=\omega_2$, so the word is either $ababaa$ or $bababb$.
	Moreover, in the case where $\omega=w^2$, with $L=2k$, to reconstruct $w$ up to isomorphism it suffices to consider the sequence $(\ell_{0,j}(\omega): 1\leq j<k)$: indeed for $1\leq j <k$ we have $w_{k-j}=w_{k}$ if and only if $\omega_{2k-j}=\omega_{2k}$, or equivalently $\ell_{0,j}(\omega)>0$.

	\end{remark}

With this definition at hand we can now express the wanted contribution~\eqref{eq:second_term}:

\begin{lemma}[\bf Acyclic intersecting quasi-leaves]\label{lemma:cint}
For all $x,y\in[n]$ with $x\neq y$  and uniformly over $u\in [0,1]$, we have
\begin{align}
\bbE\bigl[u^{d(x)+d(y)}\mathbf{1}_{x,y\in\cQ_n\cap \cT_n}  \mathbf{1}_{B_{\omega}(x)\cap B_{\omega}(y)\neq  \varnothing}\bigr]  = \frac{\cint(w,u)}{n}+ O\Bigl(\frac{1}{n^2}\Bigr),
\end{align}
where 
\begin{equation}\label{eq:def_cint}
\cint(w,u) \defi  2\sum_{(i,j)\in \cM} \bbE^{i,j,\ell_{i,j}}_{\mathbf{Y},\mathbf{Y}'} \bigl[u^{Y_k+Y_k'}\mathbf{1}_{Y_{2k}=Y_{2k}'=0}\bigr].
\end{equation}
where $\bbE^{i,j,\ell_{i,j}}_{\mathbf{Y},\mathbf{Y}'} \bigl[u^{Y_k+Y_k'}\mathbf{1}_{Y_{2k}=Y_{2k}'=0}\bigr]$ is explicitly given by Proposition~\ref{prop:cross_monster}, 
	and where $\cM=\cM(w^2)$, $\ell_{i,j}=\ell_{i,j}(w^2)$.
\end{lemma}

The self-contained expression of $\cint(w,u)$ will also be recalled in Section~\ref{subsec:def-cw-selfcontained}.

\begin{proof}
The proof combines the strategies used in the proofs of Lemmas~\ref{lem:prob_ball_given_tree} and~\ref{lem:leaf_and_tree}.
By symmetry, the events $\{x,y \in \cQ_n\cap \cT_n\}$ and $\{R,R' \in \cQ_n\cap \cT_n\}$, for $(R,R')$ a uniformly chosen ordered pair of distinct elements in $[n]$, have the same probability; we will bound the latter.

For the rest of the proof, set $\omega=w^2$.	We start with the contribution of cases in which $|B_{\omega}(x)\cap B_{\omega}(y))| =1 $.
	Define $\mathfrak{H}_{<2k}$ as the set of $\{a,b\}$-digraphs $H$ on $[|H|]$ carrying two roots (two ordered marked distinct vertices) that are trees and such that
$\mathbb P\bigl(R,R'\in \cQ_n,B_{\omega}(R)\cup B_{\omega}(R') \conglab H\bigr) > 0$. 
Fix $H\in \mathfrak{H}_{<2k}$ and let $h=|H|$. There are $(n)_{h}/h!$ many ways to label 
	$H$ with elements of $[n]$ 
	that are compatible with the relation $\conglab$. Given such a compatible relabelling $H^{*}$ with root-pair $(r,r')$, the event $\{B_{\omega}(r)\cup B_{\omega}(r')= H^{*}\}$ is equivalent to (i) $R=r$, (ii) $R'=r'$, (iii) each of the $h-1$ edges of $H^*$ being present and (iv) none of the in- and out-edges ``outside'' the in-balls pointing ``inside'' it; more precisely, for each $c\in\{a,b\}$ there are $n-h_c^+$ vertices whose $c$-transition should avoid a set of $h_c^-$ vertices in $H^*$.

Therefore, we have
\begin{align}
\begin{aligned}
\bbP(B_{\omega}(R)\cup B_{\omega}(R')\conglab H)
&=
\frac{(n)_{h}}{h!n^{h+1}}  \left(1-\frac{h_{a}^-}{n}\right)^{n-h_a^+}\left(1-\frac{h_{b}^-}{n}\right)^{n-h_b^+}\\
	&=\frac{ 1}{n}\cdot \frac{e^{-\tilde{h}}}{h!} \left(1+ O\Bigl(\frac{h^3}{n}\Bigr)\right),\label{eq:estimateproba2} 
\end{aligned}
\end{align}
	where $\tilde{h}\defi  h^-_a+h^-_b$. 

Now, let $T$ and $T'$ be the universal $\omega$-in-cover of the two roots of $H$, which are labelled trees. We let $v_H$ be the vertex of $H$ that minimizes the distance to its roots, among vertices of $H$ that are in $T$ and $T'$. The vertex $v_H$ corresponds to one vertex of $T$ and one vertex of $T'$. 
Let $v$ and $v'$ be the copies of $v_H$ in $T$ and $T'$, and let $i$ and $j$ be their respective heights.  Without loss of generality, we have $i<j$. Moreover, we have $(i,j)\in \mathcal{M}$, as otherwise $\phi_{{w}_{k-i+1}}(v_H)$ would be a vertex that appears in $T$ and $T'$ with strictly shorter distance to the roots. This enables us to partition $\mathfrak{H}_{<2k}=\biguplus_{(i,j)\in\mathcal{M}} \mathfrak{H}^{i,j}_{<2k}$ according to the value of the parameters $(i,j)$.

Fix $(i,j)\in \mathcal{M}$, $H\in \mathfrak{H}^{i,j}_{<k}$, and let $\ell=\ell_{i,j}$. 
First, observe that we have $\tilde{h}= |T\cup (T'\setminus T'_{v'}(\ell-1))|$. Indeed, in the count of $h^-_a+h^-_b$, each vertex of $(T\setminus T_{v}(\ell-1))\cup(T'\setminus T'_{v'}(\ell-1))$ appears once, and each pair of vertices one in $T_{v}(\ell-1)$ and one in $T'_{v'}(\ell-1)$ that correspond to the same vertex of $H$, appears also once 
	(one can think of vertices in $T'_{v'}(\ell-1)$ as ''copies'' of the ones in $T_{v}(\ell-1)$).

Second, by construction, $T_{v}(\ell)$ and $T'_{v'}(\ell)$ are isomorphic as unlabelled trees; let $\varphi$ be the isomorphism between them given by the common projection to $H$. It is possible to equip $T$ and $T'$ with plane structures and forget their labellings, obtaining $\tau$ and $\tau'$, in such a way that $\tau_{v}(\ell)\congplane \tau'_{v'}(\ell)$ and that the isomorphism in the relation $\congplane$ is given by $\varphi$.
More precisely, there are $n_1! \cdots n_f!$ inequivalent ways to equip $T$ and $T'$ with such plane structures, where $n_1,\dots,n_f$ are the number of children of the vertices in $T\cup (T'\setminus T'_{v'}(\ell-1))$: choosing the ordering of the children of free vertices fixes also the ordering of $\tau'_{v'}(\ell)$ so its compatible with $\varphi$. Note that this construction implies $(\tau,\tau';v,v')\in \mathfrak{P}^{i,j,\ell}$ and that $f=f_\ell(\tau,\tau';v,v')= |T\cup (T'\setminus T'_{v'}(\ell-1))|$ is precisely the number of free vertices in the sense of Definition~\ref{def:free}.

	Conversely\footnote{Note that here we are having a variant of the classical discussion of Section~\ref{subsec:auto}.}, consider a plane tree $(\tau,\tau';v,v')\in\mathfrak{P}^{i,j,\ell}$, such that $h=|\tau\cup (\tau'\setminus \tau'_{v'}(\ell))|$. Then we can label these vertices in $h!$ inequivalent ways to produce a rooted labelled $\{a,b\}$-digraph $H$ by identifying $T_{v}(\ell)$ and $T'_{v'}(\ell)$ via the isomorphism given by the fact that $\tau_{v}(\ell)\congplane \tau'_{v'}(\ell)$. 

Write $\mathfrak{P}^{i,j,\ell}_{<2k}$ for the set of $(\tau,\tau';v,v')\in \mathfrak{P}^{i,j,\ell}$ where $\tau$ and $\tau'$ have height less than $2k$.
From what precedes, instead of adding the contribution of $ \bbP(B_{\omega}(R)\cup B_{\omega}(R')\conglab H)$ over all graphs $H\in\mathfrak{H}^{i,j}_{<2k}$, we can add it over $(\tau,\tau';v,v')\in\mathfrak{P}^{i,j,\ell_{i,j}}_{<2k}$, with a compensating factor of $\frac{h!}{n_1!\cdots n_f!}$, with $h=|H|$ and $f=\tilde{h}$.


Using~\eqref{eq:estimateproba2}, we have

\begingroup
\allowdisplaybreaks
	\begin{align}
	\sum_{H\in \mathfrak{H}^{i,j}_{<2k}}&\bbE\Bigl[u^{d(R)+d(R')}\mathbf{1}_{B_{\omega}(R)\cup B_{\omega}(R')\conglab H}\Bigr]=\sum_{H\in \mathfrak{H}^{i,j}_{<2k}}u^{h_k}\bbP(B_{\omega}(R)\cup B_{\omega}(R')\conglab H)\\
&=\frac{1}{n}\sum_{H\in \mathfrak{H}^{i,j}_{<2k}} u^{h_k}\cdot \frac{e^{-\tilde{h}}}{h!} \left(1+ O\Bigl(\frac{h^3}{n}\Bigr)\right) \\
&=	\frac{1}{n}\sum_{(\tau,\tau';v,v')\in\mathfrak{P}^{i,j,\ell_{i,j}}_{<2k}} 
	u^{|\tau_k|+|\tau'_k|}\cdot \frac{h!}{n_1!\cdots n_f!}\frac{e^{-f}}{h!} \left(1+ O\Bigl(\frac{(|\tau|+|\tau'|)^3}{n}\Bigr)\right)\\
	&=\frac{1}{n}	\sum_{(\tau;v,v')\in\mathfrak{P}^{i,j,\ell_{i,j}}_{<2k}} u^{|\tau_k|+|\tau'_k|}\mu_{\mathbf{Y},\mathbf{Y}'}^{i,j,\ell_{i,j}}(\tau,\tau';v,v')
	 \left(1+ O\Bigl(\frac{(|\tau|+|\tau'|)^3}{n}\Bigr)\right)\\
&=\frac{1}{n}	\sum_{(\tau,\tau';v,v')\in\mathfrak{P}^{i,j,\ell_{i,j}}_{<2k}} u^{|\tau_k|+|\tau'_k|} \mu_{\mathbf{Y},\mathbf{Y}'}^{i,j,\ell_{i,j}}(\tau;v,v') +  O\Bigl(\frac{1}{n^2}\Bigr),
\end{align}
\endgroup
where $h_k$ is the number of vertices of $H$ from which starts a directed path (of length $k$) labelled by $w$ going to $R$ or $R'$,
$n_1,\dots, n_f$ are the number of children of $\mathcal{F}_{\ell_{i,j}}(\tau,\tau';v,v')$, $h=|\tau \cup (\tau'\setminus \tau'_{v'}(\ell))|$ and $f=\tilde{h}$. In the third equality we used Lemma~\ref{lemma:measure_tau2_u_v} and in the last equality, $\bbE_{\mathbf{Y},\mathbf{Y}'}^{i,j,\ell_{i,j}}[u^{Y_k+Y'_k}(Y+Y')^3\mathbf{1}_{Y_k=Y_k'=0}]<\infty$ as $u\leq 1$.
Adding over $(i,j)\in \mathcal{M}$ and recalling that we assumed $i<j$, so we must add the symmetric contribution of the cases where $j<i$ giving us a factor 2, we conclude that the expectation of acyclic pair of quasi-leaves that intersect only once is the desired one:
\begin{equation}
\bbE\bigl[u^{d(x)+d(y)}\mathbf{1}_{x,y\in\cQ_n\cap \cT_n}\mathbf{1}_{|B_{\omega}(x)\cap B_{\omega}(y)|=1}\bigr] = \frac{\cint(w,u)}{n}+ O\Bigl(\frac{1}{n^2}\Bigr).
\end{equation}

It only remains to rule out the contribution of cases in which $|B_{\omega}(x)\cap B_{\omega}(y))|\geq 2$. For this we use the same argument as in the proof of Lemma~\ref{lem:not_a_tree}. Given that $|B_{\omega}(x)\cup B_{\omega}(y)|=t$, the probability that $|B_{\omega}(x)\cap B_{\omega}(y))|\geq 2$ is at most $t^4/n^2$ and
\begin{align}\label{eq:PRKFD}
\bbE\bigl[u^{d(x)+d(y)}\mathbf{1}_{|B_{\omega}(x)\cap B_{\omega}(y))|\geq 2}\bigr]
&\leq \sum_{t\geq 1} \bbP(|B_{\omega}(x)|= t) \bbP\bigl( |B_{\omega}(x)\cap B_{\omega}(y))|\geq 2 \mid |B_{\omega}(x)|= t\bigr) \\
&\leq \frac{1}{n^2} \sum_{t\geq 1} t^4 e^{-c (t-1)} 
= O\Bigl(\frac{1}{n^2}\Bigr),
\end{align}
where we used that $u\leq 1$ in the first inequality.

\end{proof}

\subsection{Conditional in-ball lemmas}%

The following lemmas control the probability of certain events (e.g. observing a cycle, hitting a small part of the graph) or the distribution of in-balls, conditional on some part of the graph having already been revealed. 
To give a precise meaning to such a conditioning, we introduce the notion of induced event.
Roughly speaking, an induced event corresponds to the fact that the $\mathbf{a}$- or $\mathbf{b}$-preimages of certain vertices (the ones in $V_A^-$ or $V_B^-$ in the definition below) are known (this implies also that the images of certain vertices is known, these are the ones in $V_A^+$ or $V_B^+$).

\begin{definition}[Induced events]
For all $V_A^-,V_B^-,V_A^+,V_B^+\subseteq [n]$ and functions $\psi_a: V_A^+ \to V_A^-$ and $\psi_b: V_B^+ \to V_B^-$, we let $F(V_A^-,V_B^-,V_A^+,V_B^+,\psi_a,\psi_b)$ be the intersection of the following events (seen as events under the probability measure $\bbP$):
\begin{itemize}
\item[-] $\mathbf{a}(x)=\psi_a(x)$, for all $x\in V_A^+$;

\item[-] $\mathbf{b}(x)=\psi_b(x)$, for all $x\in V_B^+$;

\item[-]  $\mathbf{a}^{-1}(y)=\psi^{-1}_a(y)$, for all $y\in V_A^-$; and

\item[-]  $\mathbf{b}^{-1}(y)=\psi^{-1}_b(y)$, for all $y\in V_B^-$.
\end{itemize}
We denote $f_a^-=|V_A^-|$, $f_b^-=|V_B^-|$, $f_a^+=|V_A^+|$ and $f_b^+=|V_B^+|$. We call such an event an \emph{induced} event.  If the role of $V_A^-,V_B^-,V_A^+,V_B^+,\psi_a,\psi_b$ is irrelevant in the context, we write $F\defi F(V_A^-,V_B^-,V_A^+,V_B^+,\psi_a,\psi_b)$ to keep the notation light. We also write $\Supp(F)\defi V_A^-\cup V_B^-\cup V_A^+\cup V_B^+$ and $|F|\defi |\Supp(F)|$.
\end{definition}

\begin{lemma}\label{lemma:prob-ball-intersects}
	Let $L\geq 1$. 
	Let $F$ be an induced event and let $R$ be chosen independently and uniformly at random from $[n]$. We have
\begin{align}
	\bbP(B_L(R)\cap \Supp(F) \neq \emptyset \mid F)=O_L\left(\frac{|F|}{n}\right),
\end{align}
where the notation $O_L(\cdot)$ means that the constant hidden in big-O may depend on $L$.
\end{lemma}
\begin{proof}
	We explore the $L$-in-ball of $R$ as follows: at $t=0$ reveal the vertex $R$, and for each time $1\leq t \leq L$ reveal the $\mathbf{a}$- and $\mathbf{b}$-preimages of vertices revealed at time $t-1$. Stop the process at the random time $T\leq L$ where the $L$-in-ball is fully revealed or a vertex in $\Supp(F)$ is revealed. Denote by $D_t$ the number of vertices revealed at time $t$.

	As in the proof of Lemma~\ref{lemma:exponential-in-ball}, the process $(D_t)_{t\leq L}$ is dominated by $(\tilde{Z}_t^{(n)})_{t\leq L}$, where $\tilde{\mathbf{Z}}^{(n)}$ is a branching process of offspring distribution $\mathrm{Bin}(2n,\frac{1}{n-|F|})$: this is clear for $t<T$, and if $T\geq t$ then $D_t=0\leq \tilde{Z}_t^{(n)}$. We can assume that $|F|\leq n/2$ (otherwise the lemma trivially holds), so we have domination by a branching process with offspring distribution $\mathrm{Bin}(2n,\frac{2}{n})$, and $S_t\defi \sum_{s\leq t} D_s$ has exponential tails for any $t\leq L$. %

	Now, conditioned on $(D_s)_{s\leq t}$ and on the fact that $\Supp(F)$ has not been hit until time $(t-1)$, the probability to reveal a vertex in $\Supp(F)$ at time $t$ is at most $D_t |F| \min\bigl(\frac{1}{n-|F|-S_t},1\bigr)$, where the factor $D_t$ appear from a union bound on the preimage that hits $|F|$ among the $D_t$ ones revealed at time $t$, and where $n-|F|-S_t$ is a lower bound on the number of possible choices for a preimage at time $t$ (the minimum handles the case where this lower bound would equal zero). Using furthermore a union bound on times $t\leq L$, and taking the expectation on $(D_t)_{t\leq L}$, the probability we want to estimate is at most 
	\begin{align}
		 \bbE \left[ L S_L |F|/\max(n-|F|-S_L,1) \right]
	 \leq 2 L|F| \bbE \left[ S_L /\max(n-2S_L,1) \right],
\end{align}
	 assuming again $|F|\leq n/2$. Since $S_L$ has exponential tails, this quantity is $O_L(|F|/n)$ as wanted.
\end{proof}

\begin{lemma}\label{lemma:prob-nontree-small}
Let $F$ be an induced event and let $R$ be chosen independently and uniformly at random from $[n]$. We have	
	\begin{align}
		\bbP(R\in \cC_n\mid F)=O\left(\frac{|F|}{n}\right).
	\end{align}
\end{lemma}
\begin{proof}
	The proof is similar to the one of Lemma~\ref{lemma:prob-ball-intersects}, but we stop the process as soon as the $k$-in-ball of $R$ is fully revealed,  we reveal a vertex in $\Supp(F)$, or we create a cycle. With the notation of that proof, conditioned on $(D_s)_{s\leq t}$, the probability to create a cycle at time $t$ is at most $D_t^2/(n-|F|-S_t)$. Using union bound over $t\leq k$ and taking the expectation, the probability to create a cycle before hitting $\Supp(F)$ is $O(1/n)$, since $S_k$ has exponential tails. 
	Moreover, the probability to hit $\Supp(F)$ before completing the ball is $O\left(|F|/n\right)$ by Lemma~\ref{lemma:prob-ball-intersects}, so we are done.
\end{proof}

\begin{lemma}\label{lemma:tree-approx-conditional}

Let $F$ be an induced event and let $R$ be chosen independently and uniformly at random from $[n]$. 
	Let $T$ be a labelled rooted $\{a,b\}$-tree. 
	Let $s=|T|+|F|$,
	and assume that $s\leq \sqrt{n}/2$.

	We have
	\begin{align}\label{eq:tree-approx-conditional}
	\bbP(
	B_{\omega}(R)\conglab T
,	B_{\omega}(R)\cap \Supp(F) =\emptyset
		\mid F)	
		=\PBGW(\tau^{(lab)}_\mathbf{Z}= T) \left(1+O\left(\frac{ s^2}{n}\right)\right).
	\end{align}
\end{lemma}

\begin{proof}

The proof is similar to that of Lemma~\ref{lem:prob_ball_given_tree}, taking into account the information already revealed by $F$.
Let $t\defi|T|$ and $f\defi|F|$. For $c\in \{a,b\}$ and $*\in\{+,-\}$, let $t_a^+, t_a^-,t_b^+, t_b^-$ be defined as in \eqref{eq:def_t_c^*} and let $s_c^*\defi t_c^*+f_c^*$.

 There are $\frac{(n-f)_{t}}{t!}$ ways to relabel the vertices of $T$ with elements of $[n]\setminus \Supp(F)$ which are compatible with the relation $\conglab$. Fix such a labelling $T^{*}$.
	Given $F$ and $R\notin \Supp(F)$, one can sample the function $\textbf{a}$ by choosing uniformly $n-f_a^{+}$ outputs among $n-f_a^{-}$ possible values (and similarly for the function $\textbf{b}$). Now the event $\{B_{\omega}(R)=T^{*}\}$ is equivalent to $R=r$, where $R$ is chosen in $[n]\setminus \Supp(F)$, $t_a^+$ outputs of the function $\mathbf{a}$ that need to coincide with those in the tree, and $t_a^-$ values forbidden for this function apart from these outputs (and similarly for the function $\mathbf{b}$). Therefore, we obtain
\begin{align}
\begin{aligned}
		&
	\bbP(
		B_{\omega}(R)=T^{*}
		\mid
		F,R\not\in \Supp(F)) 
	\\&
		=\frac{1}{n-f}
		\left(\frac{1}{n-f_a^-}\right)^{t_a^+}
		\left(\frac{1}{n-f_b^-}\right)^{t_b^+}
		\left(1-\frac{t_{a}^-}{n-f_a^-}\right)^{n-s_a^+}
		\left(1-\frac{t_{b}^-}{n-f_b^-}\right)^{n-s_b^+},
\end{aligned}
\end{align}
which can be seen as a generalization of~\eqref{eq:exact_prob_tree2}, that would correspond to the case $\Supp(F)=\emptyset$.
Since $(n-f)_{t}=(n-f)(n-f-1)_{t_a^++t_b^+}$, we obtain 
\begin{align}
\begin{aligned}
		&
		\bbP(
		B_{\omega}(R)\conglab {T},	B_{\omega}(R)\cap \Supp(F) =\emptyset
		\mid
		F,R\not\in \Supp(F)) 
	\\&
		=
		\frac{1}{t!}
		\left(1-\frac{t_{a}^-}{n-f_a^-}\right)^{n-s_a^+}
		\left(1-\frac{t_{b}^-}{n-f_b^-}\right)^{n-s_b^+}
		\prod_{i=1}^{t_a^+} \left(\frac{n-f-i}{n-f_a^-}\right)
		\prod_{i=1}^{t_b^+} \left(\frac{n-f-t_a^+-i}{n-f_b^-}\right).
		\label{eq:4product}
\end{aligned}
\end{align}
If $s\leq \sqrt{n}/2$, then we have
\begin{align}
	\frac{t}{n-f_a^-}
	\leq
	\frac{s}{n-f_a^-}
	\leq
	\frac{st}{n-f_a^-}\leq\frac{s^2}{n-s}\leq \frac{1}{2}.
	\label{eq:simplebounds}
\end{align}

	Using the fact that $\ln(1-x)=O(x)$ and $\exp(x)=1+O(x)$, uniformly for all $x\in[0,\tfrac{1}{2}]$, the first product in~\eqref{eq:4product} is equal to:
	\begin{align}
		\prod_{i=1}^{t_a^+}
		\left(
1 - \frac{f-f_a^-+i}{n-f_a^-}
		\right)
		=
		\exp\left(\sum_{i=1}^{t_a^+} O\Bigl(\frac{s}{n-f_a^-}\Bigr)
		\right)
		= 1+O\left( \frac{s^2}{n-f_a^-}\right)
		=1+O\left(\frac{s^2}{n}\right).  
	\end{align}
	The same is true for the second product in~\eqref{eq:4product} by an analogous computation.

It remains to handle the first two factors in~\eqref{eq:4product}. Using that $\ln(1-x)=1-x+O(x^2)$ and $\exp(x)=1+x+O(x^2)$, uniformly for all $x\in[0,\tfrac{1}{2}]$, we have from~\eqref{eq:simplebounds} that

	\begin{align}
	\begin{aligned}
	\left(1-\frac{t_{a}^-}{n-f_a^-}\right)^{n-s_a^+}
		&=
		\exp 
		\left(
		(n-s_a^+) \left(-\frac{t_a^-}{n-f_a^-}  +O\left(
		\left(\frac{t_a^-}{n-f_a^-}\right)^2\right) \right)
		\right)
		\\& 
		=
		\exp 
		\left(
		- t_a^- \left(1+O\left(\frac{s}{n}\right)\right)  +
		 O\left(\frac{s^2}{n}\right) 
		\right)
		\\& 
		=
		\exp 
		\left(
		- t_a^-  +
		 O\left(\frac{s^2}{n}\right) 
		\right)=
		e^{-t_a^-} 
		\left(1+
		 O\left(\frac{s^2}{n}\right) 
		\right).
		\end{aligned}
\end{align}
Putting these computations into~\eqref{eq:4product} and using that $t_a^-+t_b^-=t$, we obtain
\begin{align}
		&
		\bbP(
		B_{\omega}(R)\conglab T
,	B_{\omega}(R)\cap \Supp(F) =\emptyset
		\mid 
		F,R\not\in \Supp(F)) 
		=
		\frac{e^{-t}}{t!} \left(1+O\left(\frac{s^2}{n}\right)\right).
\end{align}
We finally need to multiply this by the probability of $R \not \in \Supp(F)$, which is $1-O(s/n)$, and can thus be absorbed in the error term. By~\eqref{eq:classical}, we recognize what we wanted to prove.
\end{proof}

\begin{remark}[In-ball events]\label{rem:in-ball_events}
Given a sequence $\mathbf{x}=(x_1,\dots, x_p)$ and a sequence $\mathbf{H}^*$ of rooted labelled $\{a,b\}$-digraphs $(H_1^*,\dots, H_p^*)$ with labels in $[n]$, let $F_{\mathbf{x},\mathbf{H}^*}=\cap_{i=1}^p \{B_w(x_i)=H_i^*\}$ be the event that the $w$-in-balls of the vertices coincide with these labelled $\{a,b\}$-digraphs.
	If $p=1$, we use the notation $F_{x,H^*}$.  
	Note that such events are always induced (revealing the $w$-in-ball of a vertex is equivalent to revealing successively certain preimages of vertices of its in-neighbourhood), and thus we may apply any of the previous lemmas to them.

\end{remark}

\subsection{Pair of quasi-leaves with a cyclic one and computation of~\eqref{eq:third_term} and~\eqref{eq:fourth_term}}

With the previous lemmas in-hand, we can finish the computation of the terms~\eqref{eq:third_term} and~\eqref{eq:fourth_term}.

\begin{lemma}[\bf Pair of quasi-leaves, one of them cyclic]\label{lem:not_a_tree_asymp_prob_pairs_1}
For all $x,y\in[n]$ with $x\neq y$ and uniformly over $u\in [0,1]$, we have
\begin{equation}
\bbE\bigl[u^{d(x)+d(y)} \mathbf{1}_{x\in\cQ_n\cap\cC_n}\mathbf{1}_{y\in\cQ_n}\bigr] = 
\bbE\bigl[u^{d(x)} \mathbf{1}_{x\in\cQ_n\cap\cC_n}\bigr]g_k(u\eta_k) + O\Bigl(\frac{1}{n^2}\Bigr).
\end{equation}
\end{lemma}
\begin{proof}
By using conditional expectation we have
\begin{align}\label{eq:OEMFD}
\bbE\bigl[u^{d(x)+d(y)} \mathbf{1}_{x\in\cQ_n\cap\cC_n}\mathbf{1}_{y\in\cQ_n}\bigr]
= 
\bbE\bigl[u^{d(x)} \mathbf{1}_{x\in\cQ_n\cap\cC_n} \bbE\bigl[u^{d(y)}\mathbf{1}_{y\in\cQ_n}\mid x\in\cQ_n\cap\cC_n\bigr]\bigr]
\end{align}
Let $\mathfrak{H}^*_{<2k}$ be the set of rooted labelled $\{a,b\}$-digraphs $H^*$ with labels in $[n]$ 
for which $\bbP(x\in \cQ_n,B_{\omega}(x)= H^*)>0$. Recall that the event $F_{x,H^*}=\{B_{\omega}(x)=H^*\}$ is an induced event.
By symmetry, if $R$ is chosen independently and uniformly at random in $[n]\setminus\{x\}$,
\begin{align}\label{eq:SODAF}
\begin{aligned}
\bbE\bigl[u^{d(y)}\mathbf{1}_{y\in\cQ_n}\mid x\in\cQ_n\cap\cC_n\bigr] &=\bbE\bigl[u^{d(R)}\mathbf{1}_{R\in\cQ_n}\mid x\in\cQ_n\cap\cC_n\bigr] \\
&= \sum_{H^*\in \mathfrak{H}_{<2k}^*} \bbE\bigl[u^{d(R)}\mathbf{1}_{R\in\cQ_n}\mid F_{x,H^*}\bigr]\bbP(F_{x,H^*}\mid x\in \cQ_n\cap \cC_n),
\end{aligned}
\end{align}
where we used that $\{F_{x,H^*}: \, H^*\in \mathfrak{H}_{<k}^*\}$ is a partition of the event $\{x\in \cQ_n\cap \cC_n\}$.

Now, we have the upper and lower bounds:
	\begin{multline}\label{eq:sandwich}
\bbE\bigl[u^{d(R)}\mathbf{1}_{R\in \cQ_n\cap \cT_n}\mathbf{1}_{B_{\omega}(R)\cap B_{\omega}(x)=\emptyset}\mid F_{x,H^*}\bigr]
\leq
		\bbE\bigl[u^{d(R)}\mathbf{1}_{R\in \cQ_n}\mid F_{x,H^*}\bigr] \\
\leq 
\bbE\bigl[u^{d(R)} \mathbf{1}_{R\in \cQ_n\cap \cT_n}\mathbf{1}_{B_{\omega}(R)\cap B_{\omega}(x)=\emptyset}\mid F_{x,H^*}\bigr] + \bbE\bigl[u^{d(R)} \mathbf{1}_{B_{\omega}(R)\cap B_{\omega}(x)\neq\emptyset}\bigr]
+ \bbE\bigl[u^{d(R)} \mathbf{1}_{R\in \cC_n}\bigr].
\end{multline}

	By Lemma~\ref{lemma:tree-approx-conditional}, the lower bound in~\eqref{eq:sandwich} is equal to
\begin{align}
\sum_{T\in \mathfrak{L}_{<2k}} \bbE\bigl[u^{d(R)}\mathbf{1}_{B_{\omega}(R)\conglab T}\mathbf{1}_{B_{\omega}(R)\cap B_{\omega}(x) = \emptyset}\mid F_{x,H^*}\bigr] 
&= \sum_{T\in \mathfrak{L}_{<2k}} u^{|T_k|}\bbP(B_{\omega}(R)\conglab T,B_{\omega}(R)\cap B_{\omega}(x) = \emptyset\mid F_{x,H^*}) \\
&= g_k(u\eta_k)+ \EBGW[u^{Z_k}(Z+|H^*|)^2\mathbf{1}_{Z_{2k}=0}]\cdot \frac{1}{n}\\
&= g_k(u\eta_k)+ O(|H^*|^2/n).\label{eq:VMVDK_LB}
\end{align}
	Moreover, using Lemmas~\ref{lemma:prob-ball-intersects} and~\ref{lemma:prob-nontree-small}, and noting that $B_{\omega}(R)\subset B_{2k}(R)$ and $u\leq 1$, the difference between the upper and lower bound in~\eqref{eq:sandwich} is at most
\begin{align}\label{eq:VMVDK_UB}
	\bbP(B_{\omega}(R)\cap B_{\omega}(x)\neq\emptyset\mid F_{x,H^*}) +\bbP(R\in \cC_n \mid F_{x,H^*})
	=O(|H^*|/n),
\end{align}
which shows that the central quantity in~\eqref{eq:sandwich} satisfies 
$$
\bbE\bigl[u^{d(R)}\mathbf{1}_{R\in \cQ_n}\mid F_{x,H^*}\bigr]= g_k(u\eta_k)+ O(|H^*|^2/n).
$$
Plugging this estimate into~\eqref{eq:SODAF}, and then into~\eqref{eq:OEMFD}, yields the result. Indeed, we can bound the error term using the same ideas as in the proof of Lemma~\ref{lem:not_a_tree},
\begin{align}\label{eq:ODNES}
\frac{1}{n}\bbE[u^{d(x)}|B_{\omega}(x)|^2 \mathbf{1}_{x\in \cQ_n\cap \cC_n}]&\leq \frac{1}{n}
\bbE[|B_{\omega}(x)|^2 \mathbf{1}_{x\in \cQ_n\cap\cC_n}]\leq \frac{1}{n^2}\sum_{t\geq 1} t^4e^{-c(t-1)}= O\Bigl(\frac{1}{n^2}\Bigr).\qedhere
\end{align}

\end{proof}

\begin{lemma}[\bf Pair of cyclic quasi-leaves]\label{lem:not_a_tree_asymp_prob_pairs_2}
For all $x,y\in[n]$ with $x\neq y$ and uniformly over $u\in [0,1]$, we have
\begin{equation}
 \bbE\bigl[u^{d(x)+d(y)}\mathbf{1}_{x,y\in\cQ_n\cap\cC_n}\bigr]= O\Bigl(\frac{1}{n^2} \Bigr). 
\end{equation}
\end{lemma}
\begin{proof}
Since $u\leq 1$ and by Lemma~\ref{lem:not_a_tree} we have
\begin{equation}\label{eq:VHMDS}
\bbE\bigl[u^{d(x)+d(y)}\mathbf{1}_{x,y\in\cQ_n\cap\cC_n}\bigr]\leq \bbP\bigl(x,y\in\cQ_n\cap\cC_n\bigr)=  O\Bigl(\frac{1}{n}\Bigr)\cdot\bbP\bigl(y\in\cC_n\mid x\in\cQ_n\cap\cC_n\bigr).
\end{equation}
Let $\mathfrak{H}_{<2k}^*$ and $R$ be as in the proof of Lemma~\ref{lem:not_a_tree_asymp_prob_pairs_1}. By Lemma~\ref{lemma:prob-nontree-small} and using the same trick as in~\eqref{eq:SODAF}, we have
\begin{align}
\bbP\bigl(y\in\cC_n\mid x\in\cQ_n\cap\cC_n\bigr) 
&= \sum_{H^*\in \mathfrak{H}_{<2k}^*} \bbP\bigl(R\in\cC_n\mid F_{x,H^*}\bigr)\bbP(F_{x,H^*}\mid x\in \cQ_n\cap \cC_n) \\
&= O\Bigl(\frac{1}{n}\Bigr)  \sum_{H^*\in \mathfrak{H}_{<2k}^*}|H^*|\cdot \bbP(F_{x,H^*}\mid x\in \cQ_n\cap \cC_n) \\
&=O\Bigl(\frac{1}{n}\Bigr),
\end{align}
since $\bbE[|B_{\omega}(x)| \mid x\in \cQ_n\cap \cC_n]=O(1)$, as we argued at the end of the proof of Lemma~\ref{lem:not_a_tree_asymp_prob_pairs_1}. Putting it back into~\eqref{eq:VHMDS}, we conclude the proof of the lemma.
\end{proof}

\subsection{Final computation of variance}\label{subsec:final_obs}

We are now ready to give an asymptotic estimation of the variance.

\begin{theorem}\label{thm:observables}
For every $u\in [0,1]$ we have
\begin{align}
		\Var(Q_n^w(u)) &=  c(w,u)n + O(1),
\end{align}
where
	\begin{align}
		 c(w,u)\defi & g_k(u^2\eta_k)- g_k(u\eta_k)^2+ 2\capprox^1(w,u)g_k(u\eta_k)-\capprox^2(w,u)
+ \cint(w,u), 
	\end{align}
	where the constants $\capprox^1(w,u)$, $\capprox^2(w,u)$ and $\cint(w,u)$ are defined above and summarized altogether in Section~\ref{subsec:def-cw-selfcontained} below.
\end{theorem}

\begin{proof}
Let $Q_n=Q_n^w(u)$. We have
	$\Var(Q_n) = \mathbb{E}[Q_n^2] -  \mathbb{E}[Q_n]^2$
	with
	\begin{align}
	\mathbb{E}[Q_n^2]=n(n-1)\bbE\bigl[u^{d(1)+d(2)}\mathbf{1}_{1,2\in\mathcal{Q}_n}\bigr]+n\bbE\bigl[u^{2d(1)}\mathbf{1}_{1\in\mathcal{Q}_n}\bigr].
	\end{align} 
	The second term is equal to $n g_k(u^2\eta_k) +O(1)$ and the first to 
\begin{align}
n(n-1)
	\left(g_k(u\eta_k)^2+\frac{1}{n}\bigl(-\capprox^2(w,u)+\cint(w,u)+2g_k(u\eta_k) \bbE[u^{d(1)}\mathbf{1}_{1\in\cQ_n\cap \cC_n}]\bigr)+O\Bigl(\frac{1}{n^2}\Bigr) \right),
\end{align}
	by summing the estimates for \eqref{eq:first_term}, \eqref{eq:second_term}, \eqref{eq:third_term} and~\eqref{eq:fourth_term} obtained respectively in Lemmas~\ref{lemma:noninttreeleaves}, \ref{lemma:cint}, \ref{lem:not_a_tree_asymp_prob_pairs_1} and~\ref{lem:not_a_tree_asymp_prob_pairs_2}.
	Moreover $\mathbb{E}[Q_n]$ is equal to $g_k(u\eta_k) n -\capprox^1(w,u)+ \bbE[u^{d(1)}\mathbf{1}_{1\in\cQ_n\cap \cC_n}]n+O\bigl(\frac{1}{n}\bigr)$ by~\eqref{eq:expectation_full}. The coefficient of $n^2$ in the variance clearly cancels out. So does the leading term of the factor including $\bbE[u^{d(1)}\mathbf{1}_{1\in\cQ_n\cap \cC_n}]$ (recall that this expectation has order $O(1/n)$ by Lemma~\ref{lem:not_a_tree}). The expression of $c(w,u)$ follows by computing remainder of the coefficient of $n$.
\end{proof}

\subsection{Summary of constants}
\label{subsec:def-cw-selfcontained}

The goal of this section is to recapitulate all formulas necessary to
give a self-contained expression of the constant $c(w,u)$
in terms of the word $w$. We thus include here the expressions of
the constants $\capprox^1(w,u)$,  $\capprox^2(w,u)$ and $\cint(w,u)$
given above in
Lemmas~\ref{lem:leaf_and_tree},~\ref{lemma:noninttreeleaves} and~\ref{lemma:cint}
as well as the intermediate sets, operators, and generating functions
needed to compute them. All these formulas have either already been
given, or follow directly from the expressions given in the lemmas
together with the generating function expressions of
Section~\ref{subsec:GW_gf}, in particular~\eqref{eq:WPEOW}.

All the formulas of this section are easily implemented in a
mathematical software. See for example the Maple worksheet accompanying this paper~\cite{maple} which
has been used to generate our tables and figures.

\renewcommand{\arraystretch}{2.5} %
\noindent\begin{tabularx}{\textwidth}{|X|}

\hline
Fix $w\in \{a,b\}^k$ and let $\omega=w^2$.
\\
    \hline 
	Let $\cM=\cM(\omega)$ be the set of pairs $(i,j)$ with $0\leq i< j< 2k$, with $i=0$  or $\omega_{2k+1-i}\neq \omega_{2k+1-j}$. For all $(i,j)\in \cM$,  $\ell_{i,j}=\ell_{i,j}(\omega)$ is the largest integer $\ell \in \{0,1,\dots,2k-j\}$ such that $\omega_{2k+1-i-s}=\omega_{2k+1-j-s}$ for all $s\in [\ell]$. \\
    \hline
    The sets $A^+,B^+,A^-,B^-$ are defined as follows. For $(c,C)\in\{(a,A),(b,B)\}$:
    \begin{align}
    C^+= \{1\leq \ell\leq 2k: \,\omega_{2k-\ell+1}=c\} \quad \text{ and}\quad C^-= \{0\leq \ell<2k: \,\omega_{2k-\ell}=c\}.
    \end{align}
	Let $\overline{Z}^c = \overline{Z}_{C^-}(\overline{Z}_{C^-}-2\overline{Z}_{C^+})$ with $\overline{Z}_U=Z^{(1)}_U+Z^{(2)}_U$ for every $U\subseteq \mathbb{N}$, where $\mathbf{Z}^{(1)}$ and $\mathbf{Z}^{(2)}$ are independent Po(1)-BGWTs. For $\mathbf{Z}=(Z_0,Z_1,\dots)$, we write $Z_U= \sum_{i \in U} Z_i$ and
	$Z_{\leq 2k}= \sum_{i= 0}^{2k} Z_i$. \\
    \hline
    Define the polynomials:
    {\begin{align}
	    Q_1(Z_0,\dots,Z_{2k}) &\!\defi \frac{1}{2}\bigl[Z_{\leq 2k}(Z_{\leq 2k}-1)+Z_{A^-}(Z_{A^-}\!-2Z_{A^+}) +Z_{B^-}(Z_{B^-}\!-2Z_{B^+}) \bigr], \\
	    Q_2(Z^{(1)}_0,\dots,Z^{(1)}_{2k};Z^{(2)}_0,\dots,Z^{(2)}_{2k})&\!\defi \frac{1}{2} \bigl[({Z}^{(1)}_{\leq 2k}+{Z}^{(2)}_{\leq 2k}-2)({Z}^{(1)}_{\leq 2k}+{Z}^{(2)}_{\leq 2k}+1)+\overline{Z}^a+\overline{Z}^b\bigr].
    \end{align}}
\vspace{-12pt}    
    \\
    \hline
    The generating functions $g_t(s)$ and $G_t(x_0,\dots,x_t)$ are given by:
    {\begin{align}
        g_{t}(s) &= f(g_{t-1}(s)) \text{ for all } t\geq 1 \text{ and } g_0(s) = s, \\
        G_t(x_0,\dots,x_t) &= x_0 f\bigl(G_{t-1}(x_1,x_2,\dots,x_{t})\bigr) \text{ for } t\geq 1 \text{ and } G_0(x_0)=x_0,
    \end{align}}
    where $f(s)=e^{s-1}$. 
    For all $t\geq 0$, we write $\eta_t = g_t(0)$. \\
    \hline
	For $(z,\mathbf{z})\in \{(x,\mathbf{x}),(y,\mathbf{y})\}$ we write $\mathbf{z}= (z_0,z_1,\dots,z_{2k})$ and $\displaystyle\frac{\mathbf{z}\partial}{\partial\mathbf{z}}= \left(\frac{z_0\partial}{\partial z_0}, \dots, \frac{z_{2k}\partial}{\partial z_{2k}}\right)$. We define the operator $\Xi_{2k}^{\mathbf{z},u}$ that specializes the variable $\mathbf{z}$ to 
	$(\underbrace{1,\dots,1}_{k},u,\underbrace{1,\dots,1}_{k-1},0)$, i.e. sets $z_k=u,z_{2k}=0$, \\[-4.5ex]
	
	and all other $z_i$ to one. 
\\ 
	\hline
Let
{\begin{align}
F_1^{i,j} (u):=\bigl(\prod_{s=0}^{(i-1)\wedge k}g_{k-s}(u \eta_{k})\bigr) \bigl(\prod_{s=0}^{(j-1)\wedge k}g_{k-s}(u \eta_{k})\bigr) \bigl(\prod_{s=k+1}^{i-1} \eta_{2k-s}\bigr)  \bigl(\prod_{s=k+1}^{j-1} \eta_{2k-s}\bigr).
\end{align}
	}
	\\[-8ex]
	{
\begin{align}
F_2^{i,j,\ell} (u):=\begin{cases}
g_{\ell}(\eta_{2k-i-\ell}\eta_{2k-j-\ell}) &\text{if $k<i$,}\\
g_{k-i}(u g_{i+\ell-k}(\eta_{2k-i-\ell}\eta_{2k-j-\ell})) &\text{if $i\leq k <j$ and $k\leq i+\ell$,}\\
g_{\ell}(g_{k-i-\ell}(u\eta_{k})\eta_{2k-j-\ell}) &\text{if $i\leq k <j$ and $i+\ell<k$,}\\
g_{k-j}(u g_{j-i}(u g_{i+\ell-k}(\eta_{2k-i-\ell}\eta_{2k-j-\ell}))) &\text{if $j\leq k$ and $k\leq i+\ell$,}\\
g_{k-j}(u g_{\ell+j-k} (g_{k-i-\ell}(u\eta_{k})\eta_{2k-j-\ell})) &\text{if $j\leq k$ and $i+\ell< k<j+\ell$,}\\
g_{\ell}(g_{k-i-\ell}(u\eta_{k})g_{k-j-\ell}(u\eta_{k})) &\text{if $j\leq k$ and $j+\ell\leq k$}.
\end{cases}
\end{align}
	}
	\vspace{-10pt}
	\\
	\hline
\end{tabularx}

\vspace{12pt}

\renewcommand{\arraystretch}{2.5} %
\noindent\begin{tabularx}{\textwidth}{|X|}

    \hline 
\noindent With the previous notation, we have 
{\begin{align}
	\capprox^1(w,u)&= \Xi^{\mathbf{x},u}_{2k}\Bigl[
		Q_1\left(\frac{\bfx\partial}{\partial \bfx}\right)
        G_{k}(\mathbf{x})\Bigr],
                \\
        \capprox^2(w,u)&=
	\Xi^{\mathbf{x},u}_{2k}\Xi^{\mathbf{y},u}_{2k} \Bigl[ 
		Q_2\left(\frac{\bfx\partial}{\partial \bfx};\frac{\bfy\partial}{\partial \bfy}\right)
        G_{2k}(\mathbf{x})
        G_{2k}(\mathbf{y})\Bigr]
        ,
        \\
	\cint(w,u)&= 2\sum_{(i,j)\in \cM} F_1^{i,j} (u)F_2^{i,j,k,\ell_{i,j}} (u)
	. 
\end{align}}
and
{\begin{align}
		 c(w,u)= &g_k(u^2\eta_k)- g_k(u\eta_k)^2+ 2\capprox^1(w,u)g_k(u\eta_k)-\capprox^2(w,u)
+ \cint(w,u) 
		 .  
\end{align}} 
	\vspace{-20pt}
	\\
\hline
\end{tabularx}

\begin{proof}[Proof of formulas in this section]
        The only formulas which are not directly taken from previous sections
are the expressions of $\capprox^1(w)$ and $\capprox^2(w)$ involving differential operators. However they follow directly from their expression in
Lemmas~\ref{lem:leaf_and_tree} and~\ref{lemma:noninttreeleaves}, and from~\eqref{eq:WPEOW}, which can be rephrased as saying that for any polynomial $P$ in $2k+1$ variables,
\begin{align}\label{eq:SOEMD}
	\EBGW\bigl[u^{Z_k}P(Z_0,\dots,Z_{2k})\mathbf{1}_{Z_{2k}=0}\bigr]& =
	\Xi^{\mathbf{x},u}_{2k} \Bigl[ P\Bigl(\frac{\bfx\partial}{\partial \bfx}\Bigr)
	G_{2k}(\mathbf{x})\Bigr].\qedhere
\end{align}
\end{proof}

\section{Asymptotic independence and proof of Theorem~\ref{thm:cw_separate}}
\label{sec:independence}

Our goal in this section is to prove that the real functions $u\mapsto c(w,u)$ and $u\mapsto c(w';u)$ are different (as functions) as long as the words $w$ and $w'$ are non-isomorphic. In order to do that, we will first prove some sort of asymptotic independence of the different functions that appear as building blocks in our explicit expressions of $c(w,u)$. See also Remark~\ref{rem:old_schanuel} for a comment on our approach.

\subsection{Asymptotic independence lemmas}

First observe that the explicit expressions enable us to extend the domain from $u\in [0,1]$ to $u\in \mathbb{R}$ and  consider the functions $c(w,u):\mathbb{R}\to\mathbb{R}$, which are well-defined.  To study the algebraic independence of real functions, we examine their asymptotic behaviour at infinity. For the sake of simplicity, when the context is clear, we write $h(u)$ to denote the function $u\mapsto h(u)$.

\begin{definition}
	For any two real functions $h_1, h_2:\mathbb{R}\to\mathbb{R}$ which are positive for sufficiently large $u$, we write $h_1(u)\lll h_2(u)$ if $\ln(h_1(u)) = o(\ln(h_2(u))$ as $u\to \infty$.
\end{definition}

Recall we set $f(u)=e^{u-1}$ and defined $g_{t}(u)=f(g_{t-1}(u))$ with $g_0(u)=u$. We informally call the function $g_t$ a \emph{tower of exponentials} of height $t$; for instance $g_4(u)$ is a tower of exponentials that looks like
$$
g_4(u)=e^{e^{e^{e^{u-1}-1}-1}-1}.
$$
We also allow towers of exponentials to carry constant positive coefficients at each level; for instance, a possible tower could be
$$
f((1/3)f(4f(f(2u))))=e^{(1/3)e^{4e^{e^{2u-1}-1}-1}-1}.
$$

The next lemma helps us compare the towers of exponentials appearing in such iterated compositions:
\begin{lemma}\label{lemma:compare_towers_final}
	Let $s\geq r\geq 1$ be integers and let $a_1, \dots a_r$ and $ b_1,\dots, b_s>0$ be (strictly) positive real numbers.
	Let 
\begin{align}
	h_a(u)&=f(a_1f(a_2\dots f(a_r u)\dots))\\ h_b(u)&=f(b_1f(b_2\dots f(b_s u)\dots)).
\end{align}	
	Then if $r<s$, we have $h_a(u)\lll h_b(u)$. 
	The same is true if $r=s$ and if the largest $i\in [r]$ such $a_i\neq b_i$ is such that $a_i<b_i$ and $i>1$.
\end{lemma}
Note that the hypothesis $i>1$ is necessary in the second conclusion of the lemma. Indeed, for $0<a<b$, we have $e^{-1+au}=o( e^{-1+bu} )$ but $e^{-1+au}\lll \hspace{-4mm}/\hspace{4mm} e^{-1+bu}$. However, we have $f(e^{-1+au})\lll f(e^{-1+bu}).$
\begin{proof}
	We start with the following direct observation: if $A=A(u)>0$ and $B=B(u)>0$ are such that $B(u)\rightarrow \infty$ as $u\to \infty$ and  $A(u)=o(B(u))$, then for any $a,b>0$ we have 
	$f(aA(u))\lll f(bB(u))$. 

	For any $q\in [r]$, by descending induction on $1\leq p\leq q$, the claim implies that under the same hypotheses: 
	\begin{align}\label{eq:comp}
	f(a_pf(a_{p+1}\dots f(a_{q}A(u))\dots )) \lll
	f(b_pf(b_{p+1}\dots f(b_{q}B(u))\dots )),
	\end{align}
	and this is true in particular for $p=1$.

	In the case $r<s$ we take $q=r$, $A(u)=u$ and $B(u)=f(b_{r+1}f(b_{r+2}\dots f(b_{s}u)\dots ))$.
	In the case $r=s$ we take $q=i-1$ (so $q\in [r-1]$), $A(u)=f(a_if(a_{i+1}\dots f(a_{r}u)\dots ))$ and $B(u)=f(b_if(a_{i+1}\dots f(a_{r}u)\dots ))$. In both cases,~\eqref{eq:comp} with $p=1$ is what we wanted to prove.
\end{proof}

We will also need the following variant of the previous lemma, where we allow one of the constants $a_i$ (or $b_i$) to be itself a tower of exponentials, called the \emph{branching tower}, of height strictly smaller than $r-i$ (or $s-i$) so the \emph{main tower} is the highest one. As an example, consider the function
$$
h_a(u)=g_2(g_2(u)g_4(u))= e^{e^{e^{e^{u-1}-1} e^{e^{e^{e^{u-1}-1}-1}-1}-1}-1},
$$
that has a main tower of height $6$ and a branching tower of height $2$ at level $2$ (so $2+2<6$); in the notation of the lemma below, we have $a_2(u)=g_2(u)$ and the rest of the $a_i$ equal $1$.

The next result essentially says that the criterion in Lemma~\ref{lemma:compare_towers_final} is unaffected by the existence of  branching towers, as long as their height is smaller than the main tower.
\begin{lemma}\label{lemma:compare_branches}
	Let $s\geq r\geq 1$ be integers, $i_0\in [r]$ and $j_0\in [s]$.  
	Let $a_1, \dots, a_{i_0-1}$, $a_{i_0+1},\dots, a_r$, and $b_1,\dots,b_{j_0-1}$, $b_{j_0+1},\dots, b_s$ be (strictly) positive real numbers. Let $a_{i_0}(u)$ and $b_{j_0}(u)$ be  towers of exponentials of height (strictly) less than $r-i_0$ and $s-j_{0}$ respectively, namely
	\begin{align}\label{eq:FDIDS}	
	a_{i_0}(u) &= f(\alpha_1f(\alpha_2 \dots f(\alpha_{r-i_1}u)\dots)),\\ 
	b_{j_0}(u) &=f(\beta_1f(\beta_2 \dots f(\beta_{s-j_1}u)\dots)), 
	\end{align}
	for positive real numbers $\alpha_1,\dots,\alpha_{r-i_1},\beta_1,\dots,\beta_{s-j_1}$ with $i_0<i_1\leq r$ and $j_0< j_1\leq s$ (if $i_1=r$ or $j_1=s$ the functions are interpreted as $u$).

	Let 
	\begin{align}\label{eq:form_branch_tower}
	\begin{aligned}
		h_a(u)&=f(b_1f(a_2\dots f(a_{i_0}(u)\dots f(a_r u)\dots)\dots)),\\
		h_b(u)&=f(b_1f(b_2\dots f(b_{j_0}(u)\dots f(b_s u)\dots)\dots)).
	\end{aligned}
	\end{align}

	Then $h_a(u)\lll h_b(u)$ in the following cases:
\begin{itemize}
\item[(i)] if $r<s$;
\item[(ii)] if $r=s$ and the largest $i\in [r]$ such that $a_i\neq b_i$ is at least $2$ and is such that $a_i<b_i$;
\item[(iii)] if $r=s$ and the largest $i\in [r]$ such that $a_i\neq b_i$ is $1$ and $a_i=o(b_i)$.
\end{itemize}	
	In case $(ii)$, in the case where $a_i$ or $b_i$ is a function, the inequality $a_i<b_i$ is meant for $u$ large enough.

\end{lemma}

\begin{proof}
	The initial observation in the proof of Lemma~\ref{lemma:compare_towers_final} admits the following (even more direct) variant:  if $A=A(u)>0$ and $B=B(u)>0$, and  $a=a(u)>0$, $b=b(u)>0$, are such that $b(u)B(u)\rightarrow \infty$ and  $a(u)A(u)=o(b(u)B(u))$, then 
	$f(a(u)A(u))\lll f(b(u)B(u))$.

Now, since $i_0<i_1$, and respectively $j_0<j_1$, we have from the case $r<s$ of previous lemma that, respectively
	\begin{align} \label{eq:comp2}
		a_{i_0}(u) \lll f(a_{i_0+1}\dots f(a_r u)\dots) \ \ \ , \ \ \ 
	  b_{j_0}(u) \lll f(b_{j_0+1}\dots f(a_s  u)\dots).
	\end{align}
If follows that if  $f(a_{i_0+1}\dots f(a_r  u)\dots)\lll f(b_{i_0+1}\dots f(b_s  u)\dots)$
then we also have 
	\begin{align}\label{eq:comp3} a_{i_0}(u) f(a_{i_0+1}\dots f(a_r  u)\dots) \lll
	  b_{i_0}  f(b_{i_0+1}\dots f(a_s  u)\dots).
	\end{align}
	and respectively the same statement is true with $i_0$ replaced by $j_0$, provided that $j_0\leq r-1$.

We can now conclude similarly as in the previous lemma, but with a separate induction in each case.
	First consider the case $r<s$. Take $A(u)=u$ and $B(u)=f(b_{r+1}f(b_{r+2}\dots f(b_{s}u)\dots ))$,  which satisfy $A(u)\lll B(u)$ and so $f(a_rA(u))\lll f(b_rB(u))$. As in the previous lemma, by descending induction on $1\leq p\leq r-1$ and using either the claim or its variant at each step, relying on~\eqref{eq:comp3} if the variant is needed to ensure its hypotheses, we obtain  
	$$
	f(a_pf(a_{p+1}\dots f(a_{r}A(u))\dots )) \lll
	f(b_pf(b_{p+1}\dots f(b_{r}B(u))\dots )).
	$$

	In the case $r=s$ we take $A(u)=f(a_if(a_{i+1}\dots f(a_{r}u)\dots ))$ and $B(u)=f(b_if(a_{i+1}\dots f(a_{r}u)\dots ))$, and again by descending induction on $1\leq p\leq i-1$ (using either the claim or its variant at each step, relying on~\eqref{eq:comp3} if the variant is needed to ensure its hypotheses) we obtain that 
	$$
	f(a_pf(a_{p+1}\dots f(a_{i-1}A(u))\dots )) \lll
	f(b_pf(b_{p+1}\dots f(b_{i-1}B(u))\dots )).
	$$
	In both cases setting $p=1$ gives what we want to prove.
\end{proof}

\begin{lemma}[\bf Asymptotic independence]\label{lemma:AI}
	Let $\mathcal{X}$ and $\mathcal{Y}$ be two finite sets of functions from $\mathbb{R}_{\geq 0}$ to $\mathbb{R}_{\geq 0}$.
	Let $p:=|\mathcal{X}|$ and assume that we can write $\mathcal{X}=\{f_1(u),\dots,f_p(u)\}$ with $f_1(u)\rightarrow \infty$ and
$$ {f_1(u)} \lll {f_2(u)}\lll \dots\lll {f_p(u)}.$$

	Assume moreover that there are finite sets $I_1, \dots, I_p \subset (0,1)$ such that we can write
	$\mathcal{Y} = \bigcup_{i\in [p]} \mathcal{Y}_i$, 
	with $\mathcal{Y}_i = \{   [f_i(u)]^{x_i} , x_i \in I_i\}$. Write $I_i^0:=I_i\cup \{0\}$ for $i\in [p]$.

	Consider a polynomial in the variables $\mathcal{X}$ and $\mathcal{Y}$ which is at most linear in each $\mathcal{Y}_i$, i.e:
\begin{align}\label{eq:defP}
P(u)=
\sum_{j_1,\dots,j_p \geq 0}^{finite}
\sum_{ (x_1,\dots,x_p) \in I_1^0\times \dots \times I_p^0}
a_{j_1,\dots,j_p;x_1,\dots,x_p}
	\prod_{i=1}^p [f_i(u)]^{(j_i+x_i)}
.
\end{align}

	Then the real function $u\mapsto P(u)$ determines $P$ as a polynomial, i.e. it determines the set of coefficients $$(a_{j_1,\dots,j_p;x_1,\dots,x_p} ; \ j_1,\dots,j_p \geq 0 , \  (x_1,\dots,x_p) \in I_1^0\times \dots \times I_p^0).$$
\end{lemma}
\begin{proof}
The hypotheses on $f_1,\dots,f_p$ ensure that set of functions 
$$
	u\mapsto	\prod_{i=1}^p [f_i(u)]^{(j_i+x_i)}
$$
	for $j_1,\dots,j_p \geq 0$ and $(x_1,\dots,x_p) \in I_1^0\times \dots \times I_p^0$ is totally ordered for the relation ``being a little-o of'', with total order given by the right-to-left lexicographic order on $(j_1+x_1,\dots,j_p+x_p)$. (Note that this includes the constant function $\mapsto 1$, and this is why we assume $f_1(u)\rightarrow \infty$).

	In other words, these functions form a valid asymptotic scale, and the lemma follows by uniqueness of asymptotic expansion.
\end{proof}

\subsection{Reconstruction of $w$ from $c(w,u)$}

We will apply the independence lemma to the expression of $c(w,u)$ summed up in Section~\ref{subsec:def-cw-selfcontained}. We first need to define the functions that will play the role of ``variables'' in the polynomial expressions. More precisely, our goal is to be able to express all the functions $F_1^{i,j}$ and $F_2^{i,j,\ell}$ of Proposition~\ref{prop:cross_monster} as a polynomial with complex coefficients of these chosen variables.

We first introduce the set of real functions 
\begin{align}
 \mathcal{G}&=	\big\{ g_{k-i}(u\eta_k), 0\leq i < k  \big\}.
\end{align}
Note that we exclude the case $i=k$, that would correspond to the function $g_0(u \eta _k)=u\eta_k$, indeed we prefer to keep the variable ``$u$'' aside. It is clear that the function $F_1^{i,j}$ is a polynomial in $ \mathcal{G} \cup \{u\}$, for any $0\leq i\ < j <2k$.

We now consider the set of functions $F_2^{i,j,\ell}$, which requires some case disjunction. Consider the following sets of real functions
\begin{align}\label{eq:set_def}
\begin{aligned}
	\mathcal{B}&=	\big\{g_{k-i}(u g_{i+\ell-k}(\eta_{2k-i-\ell}\eta_{2k-j-\ell}));  0\leq i< k, k <j< 2k, k-i\leq \ell \big\}\\
\mathcal{C}&=	\big\{ g_{\ell}(g_{k-i-\ell}(u\eta_{k})\eta_{2k-j-\ell});  0\leq i< k, k <j< 2k, 1 \leq \ell < k-i\big\}\\
\mathcal{D}_1&=	\big\{ g_{k-j}(u g_{j-i}(u g_{i+\ell-k}(\eta_{2k-i-\ell}\eta_{2k-j-\ell}))); 0\leq i< k, i< j< k, k-i\leq\ell \big\}\\
\mathcal{D}_0&=	\big\{ g_{k-i}(u g_{i+\ell-k}(\eta_{2k-i-\ell}\eta_{k-\ell})); 0\leq i< k, i<  k, k-i\leq\ell \big\}\\
\mathcal{E}_1&=	\big\{ g_{k-j}(u g_{\ell+j-k} (g_{k-i-\ell}(u\eta_{k})\eta_{2k-j-\ell})); 0\leq i< k, i< j< k,k-j < \ell <k-i  \big\}\\
\mathcal{E}_0&=	\big\{ g_{\ell} (g_{k-i-\ell}(u\eta_{k})\eta_{k-\ell}); 0\leq i< k, i< k,0 < \ell <k-i  \big\}\\
\mathcal{F}&=	\big\{ g_{\ell}(g_{k-i-\ell}(u\eta_{k})g_{k-j-\ell}(u\eta_{k})); 0\leq i< k, i< j\leq k,1\leq \ell\leq k-j \big\}
\end{aligned}
\end{align}
The functions appearing in the sets $\mathcal{B}, \mathcal{C}, \mathcal{D}_1\cup\mathcal{D}_0,\mathcal{E}_0\cup\mathcal{E}_1,\mathcal{F}$ are the functions of $u$ appearing as $F^{i,j,\ell}_2$ in Proposition~\ref{prop:cross_monster}, in cases B.--F. of its proof (see also Figure~\ref{fig:cases}), up to some restrictions and variants, namely:
\begin{compactitem}
	\item The cases $i\geq k$ and $\ell=0$ have been set aside, as they will play a special role.
	\item In $\mathcal{D}_1$ we further impose $j<k$, which ensures that the first function $g_{k-j}$ is not the identity. Setting $k=j$ in the expression appearing in $\mathcal{D}_1$ would result in a function which is a multiple of $u$: the set $\mathcal{D}_0$ precisely captures this function, divided by $u$.
	\item Similarly in $\mathcal{E}_1$, we enforce $k<j$, and the corresponding case $k=j$, once divided by $u$, is captured in the set $\mathcal{E}_0$.
\end{compactitem}
We let
$$\mathcal{V} := \{u\} \cup \mathcal{B}\cup \mathcal{C}\cup \mathcal{D}_1\cup\mathcal{D}_0\cup\mathcal{E}_0\cup\mathcal{E}_1\cup\mathcal{F}.$$
All the elements of $\mathcal{V}$ are towers or branched towers, in the previous terminology. We need to set aside certain of them for which the comparison Lemma~\ref{lemma:compare_branches} will not be applicable. 
We let
\begin{align}
\mathcal{Y}&:=  \mathcal{B}^{[i=k-1]} \cup
		      \mathcal{C}^{[\ell=1]} \cup
		     \mathcal{D}_0^{[i=k-1]} \cup
		      \mathcal{E}_0^{[\ell=1]}
\end{align}
where the notation $\mathcal{B}^{[i=k-1]}$ means that  we restrict to the case $i=k-1$ in the definition of $\mathcal{B}$ in~\eqref{eq:set_def} (for the three other sets the notation is similar). Informally, $\mathcal{Y}$ is composed of all functions which are towers but not branched towers (i.e. there is a unique ``$u$'' in their expression), and that start with the function $g_1$ in their expression given by~\eqref{eq:set_def}.

Finally, we set
\begin{align}
	\mathcal{X}&= \big( \mathcal{G} \cup \mathcal{V}  \big) \setminus \mathcal{Y}. 
\end{align}

With this setup, we can now state:

\begin{proposition}\label{prop:asymptotic_order}
	\begin{compactitem}
	\item[(i)]The sets $\mathcal{X}$ and $\mathcal{Y}$ satisfy the hypotheses of Lemma~\ref{lemma:AI}, up to multiplying each element of $\mathcal{Y}$ by some non-zero constant. 

	\item[(ii)]	the function $u\mapsto c(w,u)+g_k(u^2\eta_k)$ is a polynomial on $\mathcal{X}\cup\mathcal{Y}$ that is at most linear in $\mathcal{Y}$. 
	\end{compactitem}
\end{proposition}
\begin{proof}[Proof of~(i)]
First, we will make use of Lemma~\ref{lemma:compare_branches} to show that the set $\mathcal{X}$ is totally ordered for the relation $\lll$.  
	For the sake of contradiction, let us assume that there exist two different functions $h_1(u),h_2(u)\in \mathcal{X}$ that are not comparable with respect to the relation $\lll$.
As any function in $\mathcal{X}$ different from $u$ is a tower of height at least $1$, we may assume that $h_1(u),h_2(u)\neq u$.

	The functions ${h}_1$ and ${h}_2$ are towers or branched towers in our previous terminology.
By the contrapositive of Lemma~\ref{lemma:compare_branches}, they have the same height as towers, and in the way to express them as~\eqref{eq:form_branch_tower}, they have the same coefficient $a_i=b_i$ at any height $i>1$. Moreover, coefficients differ at height 1, and both $a_1$ and $b_1$ are constants (because if they were functions of $u$ they would be towers themselves, and we would have $a_1 = o(b_1)$ or $b_1=o(a_1)$, since the family of towers we consider is totally ordered for the little-o relation).  

	First assume that $h_1(u)\in \mathcal{G}$. Since $h_2$ has the same height as $h_1$, it cannot be in $\mathcal{G}$.
We split into cases. In each case, the fact that the coefficients need to differ at height $1$ imposes restrictions on the possible indices:
\begin{itemize}
	\item[-] $h_2(u)\in \mathcal{B}\cap \mathcal{X}$: Then $k-i=1$, but these are not in $\mathcal{X}$ by construction. 
\item[-] $h_2(u)\in \mathcal{C}\cap \mathcal{X}$: Then $\ell=1$, but these are not in $\mathcal{X}$ by construction.
\item[-] $h_2(u)\in \mathcal{D}_1\cap \mathcal{X}$: This is impossible, as even for $k-j=1$ the coefficient at height $1$ goes to infinity.
\item[-] $h_2(u)\in \mathcal{D}_0\cap \mathcal{X}$: Then  $k-i=1$, but these are not in $\mathcal{X}$ by construction.
\item[-] $h_2(u)\in \mathcal{E}_1\cap \mathcal{X}$: This is impossible, by the argument used when $h_2\in \mathcal{D}_1\cap \mathcal{X}$.
\item[-] $h_2(u)\in \mathcal{E}_0\cap \mathcal{X}$: Then $\ell=1$, but these are not in $\mathcal{X}$ by construction.
\item[-] $h_2(u)\in \mathcal{F}\cap \mathcal{X}$: This is impossible, even for $\ell=1$, as above.
\end{itemize}

	When $h_1(u)\in\mathcal{V}$, the argument is analogous.
\medskip
	
	We now check that any function in $\mathcal{Y}$ can be expressed as $c [g(u)]^{x}$ for some $g(u)\in \mathcal{X}$ and $x\in (0,1)$, and $c>0$. In fact we will show that this holds with $g(u)\in \mathcal{G}$. 
Recall that $g_t$ is monotonically increasing for any $t\geq 0$, $(\eta_t)_{t\geq 0}$ is an increasing sequence and strictly bounded above by 1. Let $h(u)\in \mathcal{Y}$. 
We split into cases:
\begin{itemize}
	\item[-] $h(u)\in \mathcal{B}\cap \mathcal{Y}$: then $k-i=1$ so 
		$h(u)=e^{g_{\ell-1}(\eta_{k+1-\ell}\eta_{2k-j-\ell}) u -1}=
		e^{-1+g_{\ell-1}(\eta_{k+1-\ell}\eta_{2k-j-\ell})/\eta_{k}}(g_{1}(u \eta_k))^{g_{\ell-1}(\eta_{k+1-\ell}\eta_{2k-j-\ell})/\eta_{k}}$
and, as $\ell\geq 1$,
$$
0<g_{\ell-1}(\eta_{k+1-\ell}\eta_{2k-j-\ell})< g_{\ell-1}(\eta_{k+1-\ell})=\eta_k,
$$ 
		so we do have $h(u)=c [g_{1}(u \eta_k)]^x$ for $x\in (0,1)$ and $c>0$.

	\item[-] $h(u)\in \mathcal{C}\cap \mathcal{Y}$: then $\ell=1$ so $h(u)=e^{\eta_{2k-j-\ell}g_{k-i-1}(u\eta_k)-1}=e^{-1+\eta_{2k-j-\ell}}(g_{k-i}(u\eta_k))^{\eta_{2k-j-\ell}}$ with $0<\eta_{2k-k-\ell}<1$. 

	\item[-] $h(u)\in \mathcal{D}_0\cap \mathcal{Y}$: then $k-i=1$ so, using that $0<g_{\ell-1}(\eta_{k+1-\ell}\eta_{k-\ell})< \eta_k$, we get
		$h(u)=e^{ g_{\ell-1}(\eta_{k+1-\ell}\eta_{k-\ell})u-1}= e^{-1+g_{\ell-1}(\eta_{k+1-\ell}\eta_{2k-j-\ell})/\eta_{k}}(g_1(u\eta_k))^{g_{\ell-1}(\eta_{k+1-\ell}\eta_{k-\ell})/\eta_k}$. 

\item[-] $h(u)\in \mathcal{E}_0\cap \mathcal{Y}$: then $\ell=1$ so $h(u)= e^{\eta_{k-1}g_{k-i-1}(u\eta_k)-1}=e^{-1+\eta_{k-1}}(g_{k-i}(u \eta_k))^{\eta_{k-1}}$ but $\eta_{k-1}<1$.
\end{itemize}

	Thus, we are done with the first statement of the lemma.
\end{proof}

\begin{proof}[Proof of~(ii)]
We will show that $c(w,u)+g_k(u^2\eta_k)$ is a polynomial on $\mathcal{X}\cup\mathcal{Y}$, which is at most linear in $\mathcal{Y}$. We refer the reader to Section~\ref{subsec:def-cw-selfcontained} for the notation herein.
Recall that 
$$
 c(w,u)+ g_k(u^2\eta_k)= 2\capprox^1(w,u)g_k(u\eta_k)-\capprox^2(w,u)
+ \cint(w,u) 	
$$
So it suffices to show that each of the three constants are themselves polynomials that satisfy our constraints. The constant $\capprox^1(w,u)$ is a linear combination of
$$
\Big(\EBGW\bigl[u^{Z_k}Z_i Z_j\mathbf{1}_{Z_{2k}=0}\bigr]\Big)_{0\leq i\leq j< 2k}
$$
Using the spine decomposition, we can write
$$
\EBGW\bigl[u^{Z_k}Z_i Z_j\mathbf{1}_{Z_{2k}=0}\bigr]= F_1^{i+1,j+1} (u).
$$
Since $F_1^{i,j}(u)$ is a polynomial in $\mathcal{X}$, so is $\capprox^1(w,u)$. An analogous argument shows the same for $\capprox^2(w,u)$.

Finally, $\cint(w,u)$ is a linear combination of the set of functions
$$
\big\{F_1^{i,j}(u) F_2^{i,j,\ell}(u):\, 0\leq i<j <2k, 0\leq \ell\leq 2k-j\big\}.
$$
	We already know that $F_1^{i,j}(u)$ is a polynomial in $\mathcal{X}$. We thus look at $F_2^{i,j,\ell}(u)$. If $(i,j,\ell)$ corresponds to one of the cases captured by the inequalities defining the sets in~\eqref{eq:set_def}, then it is either a function in $\mathcal{X}\cup\mathcal{Y}$, or a product of two functions, one of them being $u\in \mathcal{X}$ and the other one from $\mathcal{Y}$, so the term is always at most linear in $\mathcal{Y}$. 

	We now examine the cases $(i,j,\ell)$ not covered by the inequalities describing the sets in~\eqref{eq:set_def}:
	\begin{compactitem}
	\item If $i>k$, then $F_2^{i,j,\ell}(u)$ is a constant and so, a polynomial in $\mathcal{X}$.
	\item If $i=k$, then $F_2^{i,j,\ell}(u)$ is a tower of height $0$ and so, a polynomial in $u$, which belongs to $\mathcal{X}$.
	\item If $\ell=0$, then $F_2^{i,j,\ell}(u)= g_{k-i}(u^{\mathbf{1}_{\{i\leq k\}}}\eta_{k})g_{k-j}(u^{\mathbf{1}_{\{j\leq k\}}}\eta_{k})$, which is a product of variables in $\mathcal{X}$. 
	\end{compactitem}
	This concludes the proof of the second statement.
\end{proof}

\begin{corollary}
	The function $u\mapsto c(w,u) + g_k(u^2 \eta_k)$  has a unique expression as a polynomial in $\mathcal{X} \cup \mathcal{Y}$ which is at most linear in $\mathcal{Y}$.
\end{corollary}

It follows that we can consider this polynomial formally and ``extract its coefficients''.
Namely, we focus on the ``variables'' $F_2^{i,j,\ell}(u)$ for $i=0$, $j<k$, and $\ell \geq 1$.
Note that all these functions appear in $\mathcal{D}_1 \cup \mathcal{E}_1  \cup \mathcal{F}$ and it is direct to verify that they are all distinct.
The coefficient
$$
[F_2^{i,j,\ell}(u)] (c(w,u) + g_k(u^2 \eta_k) ) ,
$$
is thus well-defined.

\begin{lemma}
	For any $1\leq j<k$, the following holds: there exists $\ell>0$ such that the coefficient $[F_2^{0,j,\ell}(u)] (c(w,u) + g_k(u^2 \eta_k) )$ is non-zero if and only if $\ell_{0,j}>0$.
\end{lemma}
\begin{proof}
	The functions $F_2^{i,j,\ell}(u)$ for $i=0$, $j<k$, appear only once in the list of variables and only in the sets $\mathcal{D}_1$, $\mathcal{E}_1$, $\mathcal{F}$.
	They appear neither in the expression of $\capprox^1(w,u)$ nor in the one of $\capprox^2(w,u)$, so they can only appear in $\cint(w,u)$. Each monomial of $\cint(w,u)$ is the contribution of some $(i,j)\in \mathcal{M}(w^2)$, given by Proposition~\ref{prop:cross_monster}. Thus it is clear that $F_2^{i,j,\ell}(u)$ appears if and only if $\ell=\ell_{0,j}(w^2)$, and if this is the case it appears in a unique monomial, so that the coefficient is non-zero.
\end{proof}

\begin{corollary}
	If $w$ and $w'$ are two non-isomorphic words of the same length, the functions $u\mapsto c(w,u)$ and $u\mapsto c(w',u)$ are distinct. In particular, they are different almost everywhere on $[0,1]$.
\end{corollary}
\begin{proof}
	From the previous lemma, the function $u\mapsto (c(w,u) + g_k(u^2 \eta_k))$, or equivalently the function $u\mapsto c(w,u)$, determines the set $\{1\leq j<k, \ell_{0,j}(w^2)>0\}$. As already noted in the second part of Remark~\ref{rem:M_identifies_w}, this data determines $w$ up to isomorphism.

	The only thing left to prove is that if the functions $u\mapsto c(w,u)$ are distinct then they differ almost everywhere on $(0,1)$. But this is clear since the function $c(w,u)+g_k(u^2 \eta_k)$, being a polynomial in $\mathcal{X}\cup \mathcal{Y}$, is the restriction to $\mathbb{R}$ of an analytic function, and two analytic functions which coincide almost everywhere on $[0,1]$ are equal -- because, for example, all their derivatives at zero must coincide.
\end{proof}
Note that the above corollary concludes the proof of Theorem~\ref{thm:main_variance}.

\begin{remark}[Schanuel's conjecture]\label{rem:old_schanuel}
	In the first version of this paper that was made public, we only considered the observable $Q_n^w(u)$ for $u=0$, i.e. the number of leaves. We could then only prove TV-distance separation (here Theorem~\ref{thm:cw_separate}) in the case where $c(w,0)\neq c(w',0)$. We conjectured that this is true for any pair of non-isomorphic words $w,w'$, and proved this fact conditionally to the Schanuel's conjecture~\cite{Waldschmidt:Schanuel} from transcendental number theory.

	The discussion of this section comparing the asymptotic behaviour of towers of exponentials at infinity plays the role that was played by the Schanuel's conjecture in the first version (the Schanuel's conjecture was used to ensure that the numbers $\eta_k$ and some of their variants are algebraically independent over $\mathbb{Q}$). The analogue of Schanuel's conjecture for functions is known~\cite{Ax}, and it is tempting to try to use it to prove the algebraic independence of the different functions appearing in this section. However we preferred to consider the asymptotic behaviour of exponential towers at infinity, which, as we just showed, is an elementary and convenient way to derive the independence we need.

We have removed in this new version the proof that $c(w,0)\neq c(w',0)$ conditional on Schanuel's conjecture since this statement is no longer needed, but the interested reader can still access it in the first version of the paper  on arXiv~\cite{ourpaper_v1}.
\end{remark}

\section{Higher-order moments of quasi-leaves and total variation distance}
\label{sec:high_moments}

In this section we estimate shifted  moments of $Q_n^w(u)$.  As a consequence of a  bounded fourth moment, we will
conclude the proof of Theorem~\ref{thm:cw_separate}.

\subsection{Shifted moments}\label{sec:higher_moments}

Our goal is to estimate the $m$-th moment of the random variable $N^w_n(u)\defi Q_n^w(u)-  g_k(u\eta_k) n$, for all $m\in \mathbb{N}$.
To do this we interpret this moment as an observable on $m$ randomly sampled points (with replacement), and we show that the dominant contribution comes from configurations in which these points are ``matched together in pairs''   leading a growth of the form $O(n^{\lfloor m/2\rfloor})$; see Theorem~\ref{thm:mth-moment}. This situation  is classical in the context of Gaussian limit laws; see Remark~\ref{rem:gaussian}.

Let $\mathbf{I}_v= \mathbf{I}_v (u)= u^{d(v)}\mathbf{1}_{v\in \cQ_n}$.
We may express the $m$-th moment as
\begin{align}
\bbE\left[N^w_n(u)^m\right]
	&=
	\bbE\left[\left(\sum_{v\in [n]} (\mathbf{I}_v- g_k(u\eta_k))\right)^m\right]
\\
	&=	\bbE\left[\sum_{v_1,\dots,v_m\atop\in [n]} 
	\prod_{i=1}^m(\mathbf{I}_{v_i}- g_k(u\eta_k))\right]
	\label{eq:mth}
\\
	&=\sum_{p=1}^m 
	   {n \choose p}  
	   \sum_{m_1+\dots+m_p=m \atop m_1,\dots,m_p\geq 1}
	   S(m;p;m_1,\dots,m_p)
	   T(m;p;m_1,\dots,m_p), \label{eq:mth2}
\end{align}
where $S(m;p;m_1,\dots,m_p)$ is the number of surjections from $[m]$ to $[p]$ in which $i$ has $m_i$ preimages for all $i\in[p]$, and  
\begin{align}
T(m;p;m_1,\dots,m_p)\defi
	\bbE\Bigl[\prod_{i=1}^p(\mathbf{I}_{V_i}-g_k(u\eta_k))^{m_i}\Bigr] \label{eq:Tmpr},
\end{align}
with $(V_1,\dots,V_p)$ being a $p$-tuple of distinct vertices in $[n]$ chosen uniformly at random.
For the last equality, we grouped the terms according to $p=|\{v_1,\dots,v_m\}|$. 

\medskip

For the rest of Section~\ref{sec:higher_moments} we let $\omega:=w^2$.

We let $\cP=(P_1,\dots,P_R)$ be the random ordered partition of $[p]$ defined as the lexicographically smallest one that satisfies the following: $i$ and $i'$ are in the same part of $\cP$ if and only if $V_i$ and $V_{i'}$ are in the same connected component of the underlying undirected graph of $B_\omega(V_1)\cup\dots\cup B_\omega(V_p)$.

\begin{definition}[Valid ordering]
Given a partition $\cP$ (random or not) of $[p]$, an ordering (permutation) $\sigma=(\sigma_1,\dots,\sigma_p)$ of $[p]$ is \emph{valid} for $\cP$ if
\begin{itemize}
\item[-] for every $j,j'\in [r]$ with $j<j'$, all the elements of $P_j$ appear before any of the elements of $P_{j'}$ in $\sigma$; and
\item[-] for every $j\in [r]$ and $i\in P_j$, if $i$ is not the first element of $P_j$ to appear in the sequence $\sigma$, then $B_\omega(V_i)$ intersects at least one in-ball $B_\omega(V_{i'})$ for $i'\in P_j$ and $i'$ appearing before $i$ in $\sigma$.
\end{itemize}
\end{definition}

Any partition $\cP$ admits at least one valid ordering. Indeed, consider the graph with set of vertices $P_j$ where an edge between $i$ and $i'$ indicate intersection of the corresponding balls. This graph is connected by definition, hence it can be constructed as a sequence of growing connected graphs, adding one vertex at a time. We associate to each partition the lexicographically smallest ordering that is valid for it. For a random partition $\cP$, we denote by $\Sigma=(\Sigma_1,\dots,\Sigma_p) \in \mathfrak{S}_p$ the associated random ordering that is valid for $\cP$.

Note that the partition $\cP=(P_1,\dots,P_R)$ is encoded by $\Pi\defi  (R;|P_1|,\dots,|P_R|;\Sigma)$.
We will estimate $T(m;p;m_1,\dots,m_p)$ by summing over the (finitely many) possible values of this parameter,
and estimating the product in the expectation via a product of conditional expectations:
\begin{align}
	T(m;p;m_1,\dots,m_p)
	&=
	\sum_{{1\leq r\leq p\atop p_1+\dots+p_r=p} \atop \sigma= (\sigma_1,\dots,\sigma_p)}
	\bbE\left[\mathbf{1}_{\Pi= (r;p_1,\dots,p_r;\sigma)}\prod_{i=1}^{p}(\mathbf{I}_{V_i}-g_k(u\eta_k))^{m_i}\right]
	\\
	&=
	\sum_{{1\leq r\leq p\atop p_1+\dots+p_r=p} \atop \sigma=(\sigma_1,\dots,\sigma_p)  }
	\bbE\left[\mathbf{1}_{\Sigma=\sigma}
		\prod_{i=1}^{p}
	\bbE\left[\mathbf{1}_{\mathcal{E}_{i}}
		(\mathbf{I}_{V_{\sigma_i}}-g_k(u\eta_k))^{m_{\sigma_i}}
	\Big| F_{i-1}
	 \right]\right],\label{eq:product}
\end{align}
where we let 
$F_{i-1} = F_{\textbf{V},\textbf{H*}}$ as defined in Remark~\ref{rem:in-ball_events}, with $\textbf{V}= (V_{\sigma_1},\dots, V_{\sigma_{i-1}})$,
and where $\mathcal{E}_i$ denotes 
the following event:
\begin{compactitem}
\item[-] if  $i=1+p_1+\dots+p_{j-1}$ for some $j\in [r]$, then $B_\omega(V_{\sigma_i})$ does not intersect $\Supp(F_{i-1})$;
\item[-] if $1+p_1+\dots+p_{j-1}< i\leq p_1+\dots+p_{j}$, for some $j\in [r]$,  then $B_\omega(V_{\sigma_i})$ does intersect $\Supp(F_{i-1})$,
	but it does not intersect $\Supp(F_{i'})$ for $i'=p_1+\dots+p_{j-1}$.
\end{compactitem}
Indeed, these two cases correspond respectively to $i$ being the first element and not a first element of its part $P_j$, so the intersection of $\mathcal{E}_1,\dots,\mathcal{E}_p$ precisely controls the properties of the ordering $\sigma$ with respect to the partition.
The indicator $\mathbf{1}_{\Sigma=\sigma}$ is there to avoid overcounting in cases where there would have been multiple valid partitions.

We handle the conditional expectation in both cases, with the following two lemmas:
\begin{lemma}\label{lemma:part-conditional1}
	Assume that $1+p_1+\dots+p_{j-1}< i\leq p_1+\dots+p_{j}$, for some $j\in [r]$. Then
	\begin{align}
		\bbE\left[\mathbf{1}_{\mathcal{E}_{i}}
		(\mathbf{I}_{V_{\sigma_i}}-g_k(u\eta_k))^{m_{\sigma_i}}
	\mid F_{i-1}
	 \right]	=
				O\left(\frac{1+|F_{i-1}|}{n}\right).
\end{align}
\end{lemma}
\begin{proof}
	Because $|\mathbf{I}_{V_{\sigma_i}}-g_k(u\eta_k)|\leq 1$ it is enough to show that $\bbP\left(\mathcal{E}_{i}
	\mid F_{i-1}
	 \right)= O\left(\frac{1+|F_{i-1}|}{n}\right)$.
	But, for such an $i$, the event $\mathcal{E}_i$ requires that $B_\omega(V_{\sigma_i})$ intersects the set $\Supp(F_{i-1})$, so the result is a direct application of Lemma~\ref{lemma:prob-ball-intersects}, as $F_{i-1}$ is an induced event.
\end{proof}
\begin{lemma}\label{lemma:part-conditional2}
	Assume that $1+p_1+\dots+p_{j-1}=i$, for some $j\in [r]$, and that $m_{\sigma_i}=1$. Then
	\begin{align}
		\bbE\left[\mathbf{1}_{\mathcal{E}_{i}}
		(\mathbf{I}_{V_{\sigma_i}}-g_k(u\eta_k))^{m_{\sigma_i}}
	\mid F_{i-1}
	 \right]	=
				O\left(\frac{1+|F_{i-1}|^2}{n}\right).
\end{align}
\end{lemma}

\begin{proof}
We need to estimate
	\begin{align}\label{eq:int36}
		\bbE[
			\mathbf{1}_{\{B_\omega(V_{\sigma_i}) \cap \Supp(F_{i-1})=\emptyset\}}
		(\mathbf{I}_{V_{\sigma_i}}-g_k(u\eta_k))
	\mid F_{i-1}].
\end{align}

	We first look at the contribution of the term $\mathbf{I}_{V_{\sigma_i}}$ in the difference $(\mathbf{I}_{V_{\sigma_i}}-g_k(u\eta_k))$, and we split it according to whether the in-ball $B_{\omega}(V_{\sigma_i})$ is a tree or not. We have
\begin{align}
		\bbE[
			&\mathbf{1}_{\{B_\omega(V_{\sigma_i}) \cap \Supp(F_{i-1})=\emptyset\}}
		\mathbf{I}_{V_{\sigma_i}}
	\mid F_{i-1}] \\
	&\leq 
	\bbP\left(
		V_{\sigma_i}\in \cC_n\mid F_{i-1}\right)
		+
\bbP(V_{\sigma_i}\in \mathcal{Q}_n \cap \mathcal{T}_n, B_\omega(V_{\sigma_i}) \cap \Supp(F_{i-1})=\emptyset
	\mid F_{i-1})
\\		
&= O\Bigl(\frac{|F_{i-1}|}{n}\Bigr)+		
	\sum_{T\in \mathfrak{L}_{<2k}} u^{Z_k}
	\bbP( B_\omega(V_{\sigma_i}) \conglab T, B_\omega(V_{\sigma_i})\cap \Supp(F_{i-1}) =\emptyset
	\mid F_{i-1})	\\
	&= O\Bigl(\frac{|F_{i-1}|^2}{n}\Bigr)+
	\sum_{T\in\mathfrak{L}_{<2k}}
	u^{Z_k}\PBGW(\tau_{\mathbf{Z}}^{(lab)}=T)\left(1+O\Bigl(\frac{(|T|+|F_{i-1}|)^2}{n}\Bigr)\right)\\
	&= g_k(u\eta_k)+ O\Bigl(\frac{1+|F_{i-1}|^2}{n}\Bigr).\label{eq:SAIOD}
\end{align}
	In the first equality we used Lemma~\ref{lemma:prob-nontree-small} for the first term and we split according to the shape of the in-ball for the second one. In the second equality we used Lemma~\ref{lemma:tree-approx-conditional} for the terms for which $|T|+|F_{i-1}|\leq\sqrt{n}/2$ and we added the term $O(|F_{i-1}|^2/n)$ to account for the other cases. In the last equality we used that $E_{BGW}[u^{d_k(T)}\mathbf{1}_{Z_{2k}=0}]=g_k(u\eta_k)$ and the fact that $\EBGW[|Z|^2\mathbf{1}_{Z_{2k}=0}]=O(1)$.

	We now looking at the contribution of the term $-g_k(u\eta_k)$ in the difference. By Lemma~\ref{lemma:prob-ball-intersects}, we have
	\begin{align}
		\bbE[
			\mathbf{1}_{\{ B_\omega(V_{\sigma_i}) \cap \Supp(F_{i-1})=\emptyset\}}
		g_k(u\eta_k)
	\mid F_{i-1}]
&=g_k(u\eta_k) \bbP(
			 B_\omega(V_{\sigma_i}) \cap \Supp(F_{i-1})=\emptyset)
	\mid F_{i-1})\\
		&= g_k(u\eta_k)+O\Bigl(\frac{|F_{i-1}|}{n}\Bigr).\label{eq:SAIOD2}
	\end{align}
	Plugging~\eqref{eq:SAIOD} and~\eqref{eq:SAIOD2} into~\eqref{eq:int36}, we conclude the proof of the lemma.
\end{proof}

\begin{proposition}Let 	$1\leq r \leq p \leq m$ and let $m_1,\dots,m_p\geq 1$, $p_1,\dots,p_r\geq 1$ with $m_1+\dots+m_p=m$ and $p_1+\dots+p_r=p$. Let $\sigma=(\sigma_1,\dots,\sigma_p)$ be a permutation of $[p]$. Let $\cP=(P_1,\dots, P_R)$ be a random partition of $[p]$ and $\Pi=(R;|P_1|,\dots,|P_r|;\Sigma)$, where $\Sigma$ is the valid ordering associated to $\cP$.
Then we have
	\begin{align}
		\bbE\left[\mathbf{1}_{\Pi= (r;p_1,\dots,p_r;\sigma)}\prod_{i=1}^{p}(\mathbf{I}_{V_i}-g_k(u\eta_k))^{m_i}\right]
	=O\left(\frac{1}{n^{p-\beta}} \right),
		\end{align}
where $s\defi |\{j\in [r]: \sum_{i\in P_j} m_i=1 \}|$ and $\beta\defi r-s$.
\end{proposition}
\begin{proof}
	We write it as a product of conditional expectations as in~\eqref{eq:product}, and for each value of $i\in[p]$ we bound the corresponding factor using the previous two lemmas. The condition of Lemma~\ref{lemma:part-conditional1} holds for each element which is not the first in its part, which happens precisely $\sum_{j=1}^r (p_j-1)=p-r$ times. The condition  of Lemma~\ref{lemma:part-conditional2} applies precisely $s$ times. For the remaining values of $i$ 
	we just upper bound the expectation by $1$. The quantity to estimate is thus equal to:
	\begin{align}
		\bbE\left[\mathbf{1}_{\Sigma=\sigma}
		\prod_{i=1}^{p}
	\bbE[\mathbf{1}_{\mathcal{E}_{i}}
		(\mathbf{I}_{V_{\sigma_i}}-g_k(u\eta_k))^{m_{\sigma_i}}
	\mid F_{i-1}
	 ]\right]	=O\left(\frac{1}{n^{p-r +s}} \right) \cdot \bbE\left[  |F_p|^{2p} \right].
	\end{align}
	From Lemma~\ref{lemma:exponential-in-ball}, the last expectation is finite, and since $p-r+s=p-\beta$, we are done.
\end{proof}

The previous proposition directly implies the main result of this section: 
\begin{theorem}\label{thm:mth-moment} We have
\begin{align}
	\bbE\left[N^w_n(u)^m\right]
	=O\bigl(n^{\lfloor\frac{m}{2}\rfloor}\bigr).
\end{align}
\end{theorem}
\begin{proof}[Proof of Theorem~\ref{thm:mth-moment}]
Consider the exponent $\beta=r-s$ in the last proposition.
Each index $j\in [r]$ contributes one to this quantity unless the part $P_j$ is a singleton, $P_j=\{i\}$, with $m_i=1$.
Therefore, each index $j$ contributing to $\beta$, contributes at least $2$ in the sum $\sum_{j=1}^r \sum_{i\in P_j }m_i=p$.
It follows that $\beta\leq \left\lfloor\frac{p}{2}\right\rfloor$.
	Going back to~\eqref{eq:mth2} and~\eqref{eq:product}, we see that
$\bbE\left[N^w_n(u)^m\right]$
	is a finite sum of terms, each of which is of order $\binom{n}{p}\cdot  \frac{1}{n^{p-\beta}}=O\bigl(n^\beta\bigr)$ where $\beta\leq \lfloor\frac{p}{2}\rfloor \leq \lfloor\frac{m}{2}\rfloor$, so we are done.
\end{proof}

\begin{remark}\label{rem:gaussian}
	The proof shows that the dominant contribution to the $m$-th moment comes from the case where $\beta$ is maximal. In the case of even $m$, say $m=2q$, this corresponds to the cases where $p=m$ and all parts satisfy $\sum_{i\in P_j} m_i =2$. In this case, each part $P_j$ (which has size either $1$ or $2$) is naturally attached to two elements $[m]$, which naturally induces a matching of the elements of $[m]$ into pairs, among the $(m-1)!!$ possible such matchings. Our proof using successive conditional expectations in fact shows that the contribution of each pair in such a matching can be  evaluated essentially independently and gives rise to a multiplicative factor which is the same as the one arising in the case $m=2$. By Theorem~\ref{thm:observables}, we have
\begin{align}
\bbE[N^w_n(u)^2] = \Var[Q_n^w(u)]+O(1) = c(w,u)n+O(1).
\end{align}
Thus, a more careful analysis would show that for all $q\geq 0$
	\begin{align}
	\bbE\left[N^w_n(u)^{2q}\right]
	&\sim n^q \cdot   (c(w,u))^q (2q-1)!! \quad \text{ and }\quad
	\bbE\left[N^w_n(u)^{2q+1}\right]
	= O(n^q).
	\end{align}
In other words, all the moments of $(c(w,u)n)^{-1/2}N^w_n(u)$ converge to the corresponding moments of a standard normal distribution. By the method of moments, the convergence also holds in distribution. Details are left to the reader.
\end{remark}

\subsection{Lower bounds on total variation distance}

The fact that two sequences of random variables have asymptotically different expectations, or variances, does not imply that their total variation distance does not go to zero. However, this is the case if one has control on their tails, or on some higher-order moments. This was our main motivation for computing higher-order moments in the previous section.

We will need the following variant of the well-known ``separating statistics'' lemma, which relies on the Cauchy-Schwarz inequality (see, e.g,~\cite[Proposition 7.12]{LevinPeres2nd}). We could obtain the next lemma by applying the classical one 
to the random variables $(X-\bbE[X])^2, (Y-\bbE[Y])^2$ and using the Cauchy-Schwarz inequality again, but it is not much longer to give a direct proof.
\begin{lemma}\label{lemma:TVvariance}
	For all $\epsilon>0$ and $C\geq 1$ there is $\delta>0$ such that if $X, Y$ are real-valued random variables with 
	$|\Var(X)-\Var(Y)|\geq \epsilon $ and 
	$\bbE[X^4],\bbE[Y^4]\leq C$, then $d_{TV}(X,Y)\geq\delta$.
\end{lemma}
\begin{proof}
	Let $\mu$ and $\nu$ be the probability distributions of $X$ and $Y$ respectively, and let $\tau$ be another probability measure on $\mathbb{R}$ such that $\mu$ and $\nu$ are absolutely continuous with respect to $\tau$, say $\tau = \frac{1}{2}(\mu+\nu)$. Let $f=\frac{d\mu}{d\tau}$ and $g=\frac{d\nu}{d\tau}$ be the associated Radon-Nikodym derivatives. We use the shortcut notation $\int \cdot \defi  \int_{\mathbb{R}} (\cdot) \,d\tau(x)$. Note that $d_{TV}(X,Y) \defi \frac{1}{2}\int |f(x)-g(x)|$. Now, we have:

	\begin{align}
		|\Var(X)-&\Var(Y)|
	=\left|
	\int f(x)x^2 -\int g(x)x^2 
		- 
		\left( \int f(x) x\right)^2
		+
		\left( \int g(x) x\right)^2
\right|\\
	&=\left|
		\int (f(x)-g(x))x^2
		- 
		 \left(\int (f(x)-g(x)) x\right)
		 \left(\int (f(x)+g(x)) x\right)
\right|\\
&\leq\left|
		\int (f(x)-g(x))x^2\right|
		+\left|\int (f(x)-g(x)) x\right|
		 \int (f(x)+g(x)) |x|
\\
&\leq 
		\sqrt{\int|f(x)-g(x)|}
		\left(
		\sqrt{\int |f(x)-g(x)|x^4}
		+
		\sqrt{\int |f(x)-g(x)|x^2}
		\int (f(x)+g(x)) |x| \right)\\
		&=
		\sqrt{2 d_{TV}(X,Y)}
		\left(\sqrt{\bbE[X^4]+\bbE[Y^4]}
		+
\sqrt{\bbE[X^2]+\bbE[Y^2]}
		(	\bbE[|X|] + \bbE[|Y|])
		\right) \\
		&\leq 6C \sqrt{d_{TV}(X,Y)},
		\end{align}
		where the three inequalities follow by the triangular inequality, Cauchy-Schwarz inequality (used twice) and Jensen inequality (i.e. $\bbE[|X|^k]\leq \bbE[|X|^4]^{k/4}\leq C$ for $k=1,2$, as $C\geq 1$), respectively.
	This is enough to conclude, with
	$\delta=\epsilon^2/6 C$.  
\end{proof}
With the previous results in hand we now obtain Theorem~\ref{thm:cw_separate}.
\begin{proof}[Proof of Theorem~\ref{thm:cw_separate}]
	The first statement is a direct consequence of Theorem~\ref{thm:observables}, Theorem~\ref{thm:mth-moment} with $m=4$,
Lemma~\ref{lemma:TVvariance}, and the fact that the number of quasi-leaves, $Q_n^w(u)$, is a measurable observable of the random function $\mathbf{w}$, hence $d_{TV}(\mathbf{w},\mathbf{w}')\geq d_{TV}(Q_n^w(u), Q_n^{w'}(u))$.
\end{proof}

\section{Counting cycles}
\label{sec:cycles}

The goal of this last section is to study the distribution on short cycles of the random function $\bfw$, and to prove Theorem~\ref{thm:Dn_limit} and Corollary~\ref{cor:distinguish_d}. Both follow from Theorem~\ref{thm:main_cycle} below.  In this section, for $x,y$ integers, we denote by $\gcd{x}{y}$ and $\lcm{x}{y}$ their greatest common divisor and least common multiple respectively.

\subsection{Limit law for short cycles}

A \emph{cycle} of length $\ell\geq 1$ of a function $\bfr:[n]\longrightarrow [n]$ is a sequence $(i_1,\dots,i_\ell)$ of distinct elements of $[n]$ such that $\bfr(i_j)=i_{j+1}$ for $j\in [\ell]$, with $i_{\ell+1}=i_1$. Let $\mathrm{Cyc}_\ell(\bfr)$ be the number of cycles of length $\ell$ of the function $\bfr$.

We start with a general remark about cycles in powers of functions, that holds for any function in the deterministic setting. 
Let $\bfr\!:\![n]\rightarrow[n]$ and $\bfs=\bfr^d$ for some $d\geq 1$.
If $(i_1,\dots,i_\ell)$ is a cycle of $\bfr$, then it splits into $\gcd{\ell}{d}$ cycles of $\bfs$, each of length $\ell/\gcd{\ell}{d}$. We thus have
\begin{align}\label{eq:cycles-split}
	\mathrm{Cyc}_{i}(\bfs) = \sum_{\ell\geq i \atop \ell/\gcd{\ell}{d}=i} \gcd{\ell}{d}\mathrm{Cyc}_\ell(\bfr)
	= \sum_{c\in[d]:c|d\atop \gcd{i}{d/c}=1} c \cdot \mathrm{Cyc}_{ci}(\bfr)= \sum_{c\in[d]:c|d\atop \gcd{i}{c}=1}\frac{d}{c} \cdot \mathrm{Cyc}_{di/c}(\bfr).
\end{align}

Going back to random functions, if two words $w$ and $\pi$ satisfy $w=\prim^d$, we can realize $\bfw$ and $\bfprim$ on a joint probability space such that $\bfw=\bfprim^d$, and~\eqref{eq:cycles-split} now holds with $\mathrm{Cyc}_{\cdot}(\bfs)$ and $\mathrm{Cyc}_{\cdot}(\bfr)$ replaced respectively by $C^w_{\cdot, n}$ and $C^\prim_{\cdot,n}$. For example, if $w=\prim^2$  we have
\begin{align}
C^w_{i,n}
	= 2  C^\prim_{2i,n} + \mathbf{1}_{\{i\text{ odd}\}} C^\prim_{i,n}.
\end{align}

\medskip

Equation~\eqref{eq:cycles-split} reduces, up to elementary number theory, the study of cycles in random functions to the case of primitive words: if $w=\prim^d$ with $\prim$ primitive and $d\geq 1$, $(C^w_{i,n})_{i\geq 1}$ is determined by $(C^\prim_{\ell,n})_{\ell\geq 1}$. We have the following theorem. 
\begin{theorem}\label{thm:main_cycle}
Let $\prim$ be a primitive word and let $w=\prim^d$ for some $d\geq 1$.
For any fixed integer $L\geq 1$,  we have
\begin{align}\label{eq:main_cycle_u}
	(C^\prim_{1,n},C^\prim_{2,n}, \dots, C^\prim_{L,n})
	\stackrel{(d)}{\longrightarrow}
	(Z_1,Z_2,\dots,Z_L),
\end{align}
	where $Z_1,Z_2,\dots$ are independent random variables with $Z_\ell \sim \text{Po}(1/\ell)$ for $\ell \in \mathbb{N}$. 
Moreover, 
\begin{align}\label{eq:main_cycle}
(C^w_{1,n}, C^w_{2,n}, \dots, C^w_{L,n})
	\stackrel{(d)}{\longrightarrow}
	\Psi(Z_1,Z_2,\dots,Z_{dL})
\end{align}
	where $\Psi(x_1,x_2,\dots,x_{dL})=(y_1,y_2,\dots,y_L)$ with
	$y_i\defi \sum_{c\in[d]:c|d\atop \gcd{i}{c}=1} (d/c) x_{di/c}.$

	In both cases, the convergence also holds in the sense of  moments.
\end{theorem}

The theorem will be a direct consequence of the following lemma.
\begin{lemma}\label{lemma:prob_cycles}
Let $\prim$ be a primitive word.
Let $Q\geq 1$,  $\ell_1,\dots,\ell_Q\geq 1$ and $\ell=\max_{q\in [Q]}\ell_q$.
Let $V_1,\dots,V_Q$ be independent uniform vertices in $[n]$.
	The probability that for all $q\in [Q]$,  $V_q$ belongs to a cycle of length $\ell_q$ of $\bfprim$, and that these cycles are all distinct, is $\displaystyle n^{-Q}+O_{\ell}\bigl(n^{-(Q+1)}\bigr)$.
\end{lemma}

\begin{proof}
Let $k$ be the length of $\prim$. We let $y_0^{(q)}=V_q$ and for $0\leq r< \ell_q$ and $1\leq j \leq k$ we let $y_{rq+j}^{(q)}=\phi_{\prim_j}\bigl(y_{rq+j-1}^{(q)}\bigr)$. In other words, $(y_0^{(q)}, y_1^{(q)}, \dots, y_{\ell_qk}^{(q)})$ is the sequence of vertices revealed when we try to determine if $(V_q, \bfprim(V_q),\dots,\bfprim^{\ell_q-1}(V_q))$ is a cycle, ``intermediate'' vertices included. 

	We use a randomized procedure to expose the vertices $\bigl(y_j^{(q)}\bigr)_{q,j}$ successively, for $q$ from $1$ to $Q$ (rounds) and for $j$ from $0$ to $\ell_qk$.  The $q$-th round goes as follows: for $j=0$ we sample uniformly the new vertex $V_q$, and for $j>0$ we either sample a new random vertex uniformly and independently if the image of $y^{(q)}_{j-1}$ by the function $\phi_{\prim_{(j \text{ mod } k)}}$ has not been revealed yet, or follow a transition which has already been revealed. If at a given time we sample a new random vertex and if this vertex belongs to the set of already revealed vertices in previous rounds or in the same round but previous steps, we call such a time a ``hit''; this can happen at time $j=0$ if $V_q$ happens to be an already revealed vertex, or at a later time in each round.

Since the total number of revealed vertices is at most $k(\ell_1+\dots+\ell_Q)\leq kQ\ell=O_{\ell}(1)$ and since sampled vertices are uniform and independent, the probability to have at least $r$ hits in the whole exploration process, for $r\geq 1$ fixed, is $O_{\ell}\left(n^{-r}\right)$.

The event that $V_1,\dots, V_Q$ are in distinct cycles implies that there are at least $Q$ hits in the exploration process. Indeed, each vertex $V_q$ needs to have an incoming edge, so either $V_q$ is sampled among the already revealed vertices, or it is reached via a hit at some later step in the $q$-th round. The probability of having at least $Q+1$ hits is $O_{\ell}(n^{-(Q+1)})$.
We claim that the only configurations realizing the desired event and having exactly $Q$ hits are the ones in which $(y^{(q)}_0,\dots,y^{(q)}_{\ell_qk})$, for $q$ from $1$ to $Q$, form $Q$ disjoint directed simple cycles in the underlying $\{a,b\}$-digraph.

Indeed, if there are $Q$ hits in total, there is exactly one at each round, since there must be at least one per round. Moreover, this hit has to be the last sampled vertex of that round: if there was another sampled vertex after it, we would need at least one more hit to come back to the starting point $V_q$. We now examine two cases for the hit happening in the $q$-th round, for $q\in [Q]$. 

	First, if sampling the vertex $V_q$ is a hit, then it belongs to the previously revealed graph, which by induction we can assume is a union of directed simple cycles of lengths $k\ell_1,\dots,k\ell_{q-1}$. Say that $V_q$ belongs to the cycle generated during round $q_0$, for some $q_0 < q$. If $V_q$ appears along this cycle at a distance from $V_{q_0}$ multiple of $k$, then $V_q$ belongs the same cycle of $\bfprim$ as $V_{q_0}$, a contradiction. If the distance is not a multiple of $k$, then the $q_0$-th directed simple cycle already revealed does not induce a cycle of $\bfprim$ starting from $V_{i}$, because $\prim$ is primitive so no power of $\prim$ can coincide with one of its conjugates. This implies that there must be at least one other hit in the exploration of the $q$-th cycle, so there would be at least two hits in this round, a contradiction.

	Secondly, if the vertex $V_q$ was not revealed in the previous rounds (and thus it has no incoming edge in the revealed part), then the unique hit (and last sample) of round $q$ happens at some step $t>1$ during the exploration, i.e. $y^{(q)}_s=y^{(q)}_t$ for some $0\leq s <t$. Because there is no other edge revealed after step $t$ of round $q$, we need $y^{(q)}_s=V_q$ (otherwise $V_q$ has no incoming edge and cannot be on a cycle). Therefore the digraph spanned by $(y^{(q)}_0,\dots,y^{(q)}_{t})$ is a directed simple cycle. In order for this simple cycle to induce a cycle in $\bfprim$ of length $\ell_q$, we need $t=\ell_qk$: if $\pi^{\ell_q}$ is the power of some other word, then this word is itself a power of $\prim$ (since $\prim$ is primitive), and that power has to be $\pi^{\ell_q}$ itself otherwise the length of the cycle to which $V_q$ belongs would not be $\ell_q$ but strictly smaller. Thus the hit happens at step $t=\ell_qk$, creating a directed simple cycle $(y^{(q)}_0,\dots,y^{(q)}_{\ell_qk})$.	This concludes the  proof of the claim. 

From the claim, the desired event is realized with exactly $Q$ hits if and only if at each round $q$, from $1$ to $Q$, during the first $\ell_q k-1$ steps we sample a new vertex and at the last step we sample precisely the vertex~$V_q$. Therefore the probability of the desired event is
\begin{align}
&O_{\ell}\left(\frac{1}{n^{Q+1}}\right)+\prod_{q=1}^Q \left[ \left(1-\frac{kQ\ell}{n}\right)^{\ell_q k-1} \cdot \frac{1}{n}\right] = \frac{1}{n^Q} + O_{\ell}\left(\frac{1}{n^{Q+1}}\right).\qedhere
\end{align}
\end{proof}

The following corollary (and Theorem~\ref{thm:main_cycle}) follows directly from the lemma. For integers $m,r \geq 0$ we let $m_r\defi m (m-1)\dots(m-r+1)$ denote the $r$-th descending factorial of $m$.
\begin{corollary}
Let $\prim$ be a primitive word.
Let $L\geq 1$ and let $r_1,\dots,r_L\geq 0$.
then
\begin{align}
	\bbE \left[ (C^\prim_{1,n})_{r_1} \cdots (C^\prim_{L,n})_{r_L}\right] = \frac{1}{1^{r_1} 2^{r_2}\cdots L^{r_L}} + O\left(\frac{1}{n}\right).
\end{align}
\end{corollary}
\begin{proof}
	Let $Q=r_1+\dots+r_L$ and let $(\ell_1,\dots,\ell_Q)=(1,\dots,1,2,\dots, 2,\dots,L\dots,L)$ be the vector obtained by concatenating vectors with entries $i$ of length $r_i$, for $i\in L$. Because there are $\ell$ ways to choose a vertex on a cycle of length $\ell$, the desired expectation multiplied by $\ell_1 \cdots \ell_Q=1^{r_1}2^{r_2}\cdots L^{r_L}$ is equal to $n^Q$ times the probability obtained in Lemma~\ref{lemma:prob_cycles}, and we are done. 
\end{proof}

\begin{proof}[Proof of Theorem~\ref{thm:main_cycle}]
	Since the $r$-th factorial moment of a Po($\lambda$) random variable is $\lambda^r$, 
	the corollary directly implies the convergence of moments in Theorem~\ref{thm:main_cycle} in the case of a primitive word $\prim$ (and the convergence in distribution since these moments decrease fast enough, for example by Carleman's criterion). The non-primitive case follows directly from~\eqref{eq:cycles-split}, applied to the (direct) coupling in which $\bfw=\bfprim^d$. 
\end{proof}

\subsection{Observables}

In order to use Theorem~\ref{thm:main_cycle} to discriminate random functions, we will count cycles of length at most $L$, for some large number $L$, with some ``weights'' related to cycle lengths.
To this end, let $\mathbf{z}=(z_j)_{j\geq 1}$ be a sequence of (real or complex) numbers and introduce the random variable 
\begin{align}
D_n^w(L;\mathbf{z}) \defi  \sum_{j\leq L} z_j C^w_{j,n}.
\end{align}
We will especially focus on the case where the sequence $\mathbf{z}$ is periodic, $z_j=z_{(j\textrm{ mod }g)}$ for some integer $g\geq1$ and all $j\geq 0$.
We write, as in the introduction,
\begin{align}
	D_n^w(L,g;z_0,\dots,z_{g-1}) \defi  D_n^w(L;\mathbf{z}) = \sum_{j\leq L} z_{(j\textrm{ mod } g)} C^w_{j,n}.
\end{align}
In order to state the next theorem we first define the quantity:
\begin{align}
\rho(d,g; z_0,\dots,z_{g-1})\defi 
		\frac{1}{\lcm{d}{g}}\sum_{r=1}^{\lcm{d}{g}}z_{(r\textrm{ mod } g) }|\{c\in [d]:c|d, \gcd{c}{r}=1\}|.
\end{align}

We will be interested in some notion of bivariate convergence in probability:
\begin{definition}[A notion of bivariate convergence in probability]\label{def:bivariate}
	If $(X_{n,L})_{n,L\geq 1}$ is a bi-indexed sequence of real-valued random variables, and $Y$ is a real-valued random variable on the same probability space, 
	we write
$	\lim_{L\rightarrow \infty}	\lim_{n\rightarrow \infty} X_{n,L}=Y,$
	if for any $\epsilon>0$ there exists $L_0=L_0(\epsilon)$ such that for any $L\geq L_0$,
\begin{align}
	\lim_{n\to\infty} \bbP(|X_{n,L}-Y|> \epsilon)=0.
\end{align}	
\end{definition}

Note that this definition is weaker than a full bivariate convergence in probability (that we could denote by $\lim_{L,n\rightarrow \infty}$) since the quantifiers in the large-$n$ limit a priori depend on $L$. However, for the purpose of designing a statistical test separating two random variables it will be as useful. We have the following proposition, already stated in the introduction (Theorem~\ref{thm:Dn_limit}):

\begin{proposition}\label{prop:Dn_limit}
	Let $w$ be a word of exponent $d\geq 1$, i.e. $w=\prim^d$ for some primitive word $\prim$. For any $g\geq 1$ and any tuple $(z_0,\dots,z_{g-1})\in \mathbb{C}^g$ we have the convergence in distribution
\begin{align}
		\lim_{L\rightarrow \infty}
	\lim_{n\rightarrow \infty}
		\frac{1}{\log L} D_n^w(L,g; z_0,\dots,z_{g-1})
=
		\rho(d,g; z_0,\dots,z_{g-1}).
	\end{align}
	Moreover, the convergence also holds in expectation, and in probability in the sense of Definition~\ref{def:bivariate}.
\end{proposition}

\begin{proof}[Proof of Proposition~\ref{prop:Dn_limit}, also stated as Theorem~\ref{thm:Dn_limit}]
We first prove the convergence of the expectation. From Theorem~\ref{thm:main_cycle} we have, for any fixed $L\geq 1$,
\begin{align}
	\lim_{n\rightarrow\infty}	\bbE[\sum_{j\leq L} z_{(j\textrm{ mod } g)} C^w_{j,n}]
	=&\sum_{j\leq L} z_{(j\textrm{ mod } g)}\sum_{c\in[d]:c|d \atop \gcd{c}{j}=1} (d/c) \bbE[Z_{dj/c}]
	= \sum_{j\leq L} z_{(j\textrm{ mod } g)} \sum_{c\in[d]:c|d \atop \gcd{c}{j}=1} \frac{1}{j}\\
	=& \sum_{j\leq L} \frac{z_{(j\textrm{ mod } g)}}{j} |\{c\in[d]:c|d, \gcd{c}{j}=1\}|,
\end{align}
where $Z_1, Z_{2}, \dots$ are independent random variables with $Z_{\ell}\sim \textrm{Po}(1/\ell)$.

Note that for $c|d$, the fact that $\gcd{c}{j}=1$ depends only on the congruence class of $j$ modulo $d$, which depends only on the congruence class of $j$ modulo $\lcm{d}{g}$. 
Therefore we have, for fixed any $L\geq 1$,
\begin{align}
	\lim_{n\rightarrow\infty}\bbE[D_n^w(L,g;z_0,\dots,z_{g-1})]
=& \sum_{j\leq L} \frac{z_{(j\textrm{ mod } g)}}{j} |\{c\in[d]:c|d, \gcd{c}{j}=1\}|\\
	=& \sum_{r=1}^{\lcm{d}{g}} z_{(r\textrm{ mod } g)}\sum_{j\leq L \atop j\equiv r\textrm{ mod } \lcm{d}{g} } \frac{1}{j} |\{c\in[d]:c|d, \gcd{c}{j}=1\}|\\
	=& \sum_{r=1}^{\lcm{d}{g}} z_{(r\textrm{ mod } g)}|\{c\in[d]:c|d, \gcd{c}{r}=1\}| \sum_{j\leq L \atop j\equiv r\textrm{ mod }\lcm{d}{g} } \frac{1}{j} \\
	=& \left(\sum_{r=1}^{\lcm{d}{g}}z_{(r\textrm{ mod } g) }|\{c\in[d]:c|d, \gcd{c}{r}=1\}|\right)\cdot  \frac{\log L +O_L(1) }{\lcm{d}{g}}\\
	=&  \rho(d,g;z_0,\dots,z_{g-1})\log L + O_L(1).
\end{align}
This gives the convergence of the expectation.

We proceed similarly for the variance:\begin{align}
	\lim_{n\rightarrow\infty}\Var(D_n^w(L,g;z_0,\dots, z_{g-1}))
	&= \Var\Bigl(\sum_{j\leq L} z_{(j\textrm{ mod } g)} \!\!\sum_{c\in[d]:c|d\atop \gcd{(d/c)}{j}=1}\!\! c Z_{cj}\Bigr)\\
	&= \sum_{\ell\leq dL} \Var \Bigl(Z_\ell \!\!\!\!\sum_{c,j: \ell=cj\atop c|d, \gcd{(d/c)}{j})=1} \!\!\!\!z_{(j\textrm{ mod } g)} c\Bigr)\\
	&= \sum_{\ell\leq dL} K_{\ell}^2\Var (Z_\ell)\\
	&= O_L(\log L),
\end{align}
where $K_\ell\defi \sum_{c,j: \ell=cj\atop c|d, \gcd{(d/c)}{j})=1} z_{(j\textrm{ mod } g)} c$.
	The convergence in probability thus follows from the Chebyshev inequality.
	The convergence in distribution follows from the fact that, for any fixed $L\geq 1$, the limit distribution when $n\rightarrow \infty$ exists by Theorem~\ref{thm:main_cycle}, and from the convergence in probability just proved.
\end{proof}

Finally, we now explain how the previous result implies that one can fully separate two random functions when the words have different exponents.

\begin{proof}[Proof of Corollary~\ref{cor:distinguish_d}]

	From the convergence in probability in Theorem~\ref{thm:Dn_limit}, it suffices to show that there exists a choice of the parameters $(g;z_0,\dots,z_{g-1})$ such that $\rho(d,g; z_0,\dots,z_{g-1})\neq \rho(d',g; z_0,\dots,z_{g-1})$.
	We will show a bit more, obtaining a direct way to recover the value of $d$.%

	Write the prime number decomposition of $d$ as $d=\prod_{i=1}^m p_i^{\alpha_i}$ and let $\sigma(d)\defi\prod_{i=1}^m (\alpha_i+1)$ be its number of divisors. Take $g$ to be a multiple of $d$ (to be fixed later), so $\lcm{d}{g}=g$.

	Now let $p\leq g$ be either 1 or a prime number, and let $z^{(p)}_r\defi \mathbf{1}_{\{r=p\text{ mod }g\}}$ for $r\in \{0,\dots,g-1\}$. 
	Since $p$ is $1$ or a prime, then
	\begin{align}
		\rho(d,g;z^{(p)}_0,\dots,z^{(p)}_{g-1}) &=\frac{1}{g} \sum_{r=0}^{g-1} z_r^{(p)} |\{c\in[d]: \,c|d, \gcd{c}{r}=1\}|\\
		&= \frac{1}{g}|\{c\in[d]: \,c|d, \gcd{c}{p}=1\}|\\
		&= \begin{cases}	
			\frac{\sigma(d)}{(\alpha_i+1)g}& \mbox{ if } p =p_i \mbox{ for some } i\in[m],\\
			\frac{\sigma(d)}{g} & \mbox{ otherwise}.
	\end{cases}
	\end{align}
	Note that $p=1$ is always in the second case, so we have $\rho(d,g;z^{(1)}_0,\dots, z^{(1)}_{g-1})=\sigma(d)/g$. Moreover, a prime $2\leq p\leq g$ divides $d$ if and only if $\rho(d,g;z^{(p)}_0,\dots,z^{(p)}_{g-1})\neq \rho(d,g;z^{(1)}_0,\dots,z^{(1)}_{g-1})$, and if so, its multiplicity is $\alpha=\rho(d,g;z^{(1)}_0,\dots,z^{(1)}_{g-1})/\rho(d,g;z^{(p)}_0,\dots,z^{(p)}_{g-1}) -1$. 

	Since a number $d$ is determined by the set of primes that divide it together with their multiplicity, if $d\neq d'$ and if we fix $g=\lcm{d}{d'}$, there exists at least one prime $p\leq g$  such that $\rho(d,g;z^{(p)}_0,\dots,z^{(p)}_{g-1})\neq \rho(d,g';z^{(p)}_0,\dots,z^{(p)}_{g-1})$, which is what we wanted to prove.
\end{proof}

\begin{remark}
The proof above shows the following ``direct'' formula to reconstruct $d$. For any $g\geq d$ multiple of $d$ , we have
\begin{align}
\lim_{L\to\infty} \lim_{n\to\infty} \prod_{p \leq g \atop p\text{ prime}}
	p^{\frac{D_n^w(L,g; z^{(1)}_0,\dots, z^{(1)}_{g-1})}{D_n^w(L,g;z^{(p)}_0,\dots, z^{(p)}_{g-1})}-1}
=d,
\end{align}
	where the limit is in probability in the sense of Definition~\ref{def:bivariate}, and where $D_n^w(L,g;z^{(p)}_0,\dots, z^{(p)}_{g-1})$ is the number of cycles of $\bfw$ of length at most $L$ and congruent to $p$ modulo $g$.

\end{remark}

\section*{Acknowledgements}

We thank Valentin Féray, Peleg Michaeli and Andrea Sportiello for interesting discussions and comments.
In particular, when one of us (G.C.) presented the results of the first version at the Flajolet seminar in I.H.P. in Paris in April 2026, Andrea Sportiello raised his hand and said: {\it maybe you could add a real parameter to your main observable, because algebraic independence is always easier to prove for functions than for numbers}. This sentence is at the origin of this second version and the fact that Theorem~\ref{thm:cw_separate} is now unconditional. We thank Andrea very much for this suggestion, and for his kindness, which is only matched by his creativity and quickness of thought.

We acknowledge hospitality and perfect working environment of CIRM in Luminy, France, and the CWI in Amsterdam, Netherlands, where the workshops Combinatorics and Discrete Probability (November 2025) and PhaseCAP (March-April 2026) were held, respectively.

G.C. was supported by the grant ANR-23-CE48-0018 ``CartesEtPlus''. G.P. was supported by the grants PID2023-147202NB-I00 and CEX2020-001084-M, both funded by MICIU/AEI/10.13039/501100011033.
Both authors are supported by the grant MSCA-RISE-2020-101007705 (Horizon 2020), RandNET: Randomness and learning in networks.

\bibliographystyle{alpha}
\bibliography{biblio}

\end{document}